\documentclass[a4paper,12pt]{article}

\usepackage[utf8]{inputenc}
\usepackage[english]{babel}
\usepackage{amsthm,amsmath,amsfonts,amssymb,amsbsy}
\usepackage[misc]{ifsym}
\usepackage[T1]{fontenc}
\usepackage{cite,graphicx,color,euscript,epstopdf}
\usepackage{setspace}

\newtheorem{theorem}{Theorem}[section]
\newtheorem{proposition}[theorem]{Proposition}
\newtheorem{assumption}[theorem]{Assumption}
\newtheorem{remark}[theorem]{Remark}
\newtheorem{example}[theorem]{Example}

\makeatletter
\renewcommand{\@biblabel}[1]{#1.}
\makeatother

\newcounter{count}

\begin{document}

\title{Perspectives on characteristics based curse-of-dimensionality-free numerical approaches for solving
Hamilton--Jacobi equations}

\author{Ivan~Yegorov$^1$ $ ^* $ $\dagger$, \quad Peter~Dower$^1$}

\maketitle

\let\thefootnote\relax\footnote{$^1$ The University of Melbourne, Department of Electrical and Electronic Engineering,
Parkville Campus, Melbourne, Victoria 3010, Australia}

\let\thefootnote\relax\footnote{$\dagger$ Corresponding author. E-mail: ivanyegorov@gmail.com}

\let\thefootnote\relax\footnote{$^*$ Also known as Ivan~Egorov}

\begin{abstract}
This paper extends the considerations of the works~\cite{DarbonOsher2016,ChowDarbonOsherYin2017} regarding
curse-of-dimensio\-nality-free numerical approaches to solve certain types of Hamilton--Jacobi equations arising in optimal control
problems, differential games and elsewhere. A rigorous formulation and justification for the extended Hopf--Lax formula of
\cite{ChowDarbonOsherYin2017} is provided together with novel theoretical and practical discussions including useful recommendations.
By using the method of characteristics, the solutions of some problem classes under convexity/concavity conditions on Hamiltonians
(in particular, the solutions of Hamilton--Jacobi--Bellman equations in optimal control problems) are evaluated separately at different
initial positions. This allows for the avoidance of the curse of dimensionality, as well as for choosing arbitrary computational
regions. The corresponding feedback control strategies are obtained at selected positions without approximating the partial
derivatives of the solutions. The results of numerical simulations demonstrate the high potential of the proposed techniques.
It is also pointed out that, despite the indicated advantages, the related approaches still have a limited range of applicability,
and their extensions to Hamilton--Jacobi--Isaacs equations in zero-sum two-player differential games are currently developed only for
sufficiently narrow classes of control systems. That is why further extensions are worth investigating.
\end{abstract}


\section{Introduction}

It is well known that first-order Hamilton--Jacobi equations constitute a central theoretical framework to describe feedback
(closed-loop) solutions of deterministic continuous-time optimal control problems and differential
games~\cite{Subbotin1995,Melikyan1998,YongZhou1999,FlemingSoner2006,BardiCapuzzoDolcetta2008,Yong2015}. A complete analytical
investigation of these equations is provided only for rather specific problems. Most of the widely used numerical approaches, such as
finite-difference schemes
\cite{FlemingSoner2006,CrandallLions1984,OsherShu1991,JiangPeng2000,ZhangShu2003,BokanForcadelZidani2010,
BokanCristianiZidani2010,ROCHJ2017},
semi-Lagrangian schemes
\cite{BardiCapuzzoDolcetta2008,FalconeFerretti1998,Falcone2006,CristianiFalcone2007,ROCHJ2017},
level set methods
\cite{OsherSethian1988,Osher1993,Sethian1999,OsherFedkiw2003,MitchellBayenTomlin2001,MitchellBayenTomlin2005,
MitchellLevelSetToolbox2012},
etc., require dense state space discretizations, and their computational cost grows exponentially with the increase of
the state space dimension~$ n $, so that they in general become almost inapplicable for $ n \geqslant 4 $. The related
circumstances were referred to as the curse of dimensionality by R.~Bellman \cite{Bellman1957,Bellman1961}. Thus, it is
relevant to develop techniques that can help to avoid the curse of dimensionality for nontrivial problems.

In
\cite{McEneaney2006,McEneaney2007,McEneaneyDeshpandeGaubert2008,McEneaneyKluberg2011,GaubertMcEneaneyQu2011,AkianGaubertLakhoua2008,
KaiseMcEneaney2010,AkianFodjo2016},
several approaches for overcoming the curse of dimensionality through the advanced framework of max-plus algebra and analysis were
proposed. They work for specific classes of nonlinear optimal control problems. For example, the methods of
\cite{McEneaney2006,McEneaney2007,McEneaneyDeshpandeGaubert2008,McEneaneyKluberg2011,GaubertMcEneaneyQu2011}
can be applied if the Hamiltonian is represented as the maximum or minimum of a finite number of elementary Hamiltonians
corresponding to linear-quadratic optimal control problems. Despite the evident potential of mitigating the curse of dimensionality,
some practical issues often appear when implementing such methods, i.\,e., the curse of complexity may be a formidable barrier
(see \cite[\S 7.5]{McEneaney2006}, \cite[Section~6]{McEneaney2007}, \cite[Sections~IV,VI]{McEneaneyDeshpandeGaubert2008},
\cite[Section~9]{McEneaneyKluberg2011}, \cite{GaubertMcEneaneyQu2011}).

Another promising direction is representing the solutions of Hamilton--Jacobi equations in a way that reduces their computation at
isolated positions to finite-dimensional optimization (throughout the current paper, a position means a vector of the form~$ (t, x) $,
where $ t $ is a time and $ x $ is a state at this time). As opposed to grid-based numerical approaches, this allows the solutions
to be evaluated separately at different initial positions. Then the curse of dimensionality can be avoided, and it becomes easy
to arrange parallel computations. However, there may still appear the curse of complexity when constructing global (or semi-global)
solution approximations, and sparse grid techniques \cite{KangWilcox2017} can be useful for this purpose if the state space dimension
is not too high ($ n \leqslant 6 $).

For state-independent Hamiltonians under certain conditions, the sought-after representations are given by Hopf--Lax and Hopf
formulae \cite{Hopf1965,Subbotin1995,Evans1998,Rublev2000,Evans2014}. The related practical considerations can be found in
\cite{DarbonOsher2016}. In \cite{ChowDarbonOsherYin2017}, some extensions of these representations to problems with state-dependent
Hamiltonians were conjectured. The method of characteristics (or, more precisely, a generalization of its classical form) for
first-order Hamilton--Jacobi equations played a key role there.

This paper encompasses the following goals:
\begin{itemize}
\item  to give a rigorous formulation and justification for the first main conjecture of \cite{ChowDarbonOsherYin2017}
(i.\,e., for the extended Hopf--Lax formula) concerning Hamilton--Jacobi equations whose Hamiltonians are convex or concave
with respect to the adjoint variable (Hamilton--Jacobi--Bellman equations for optimal control problems form a typical
subclass here);
\item  to provide novel theoretical and practical discussions including useful recommendations, as well as a detailed
numerical investigation;
\item  to point out the limited range of applicability of the related approaches, and, in particular, to highlight principal
issues in their extension to Hamilton--Jacobi--Isaacs equations (for zero-sum two-player differential games) whose Hamiltonians are
in general nonconvex and nonconcave with respect to the adjoint variable.
\end{itemize}

Note also the works \cite{Mirica1985,Subbotina1991,Subbotina1992,Subbotina2006} that used the method of characteristics
in order to derive fundamental representations for the minimax/viscosity solutions of Hamilton--Jacobi equations under specific
assumptions including some smoothness and convexity/concavity conditions on Hamiltonians. The boundary value problems for
characteristic systems of ordinary differential equations (describing the dynamics of the state and adjoint variables along
characteristic curves) were involved there. However, the related computational algorithms as formulated in
\cite{SubbotinaTokmantsev2009_1,SubbotinaTokmantsev2009_2,SubbotinaKolpakova2010} are rather complicated and also suffer
from the curse of dimensionality.

For simplifying the numerical implementation of such characteristics approaches and making them curse-of-dimensionality-free,
it is crucial to parametrize characteristic fields not with respect to the terminal state but with respect to the initial
adjoint vector~\cite{ChowDarbonOsherYin2017}. This allows Cauchy problems for characteristic systems to be solved instead of
boundary value problems that might have multiple solutions and thereby cause the practical dilemma of obtaining a needed solution.
The theoretical constructions of the current work employ this idea, with the considerations of \cite{ChowDarbonOsherYin2017}
serving as a primary motivation.

Our paper is organized as follows. In Section~2, a general Hamilton--Jacobi equation under smoothness and convexity/concavity
conditions on the Hamiltonian is considered. Under certain assumptions, we establish a representation of its minimax/viscosity
solution through the parametrization of the characteristic field with respect to the initial adjoint vector. Section~3 develops that
for Hamilton--Jacobi--Bellman equations in optimal control problems without the aforementioned smoothness conditions on Hamiltonians.
It also contains important practical recommendations and a detailed analysis for Eikonal type partial differential equations.
In Section~4 and Appendix, we provide an overview of existing curse-of-dimensionality-free techniques for solving particular
classes of Hamilton--Jacobi--Isaacs equations in linear zero-sum two-player differential games (besides, most of the introduced results
and examples were not available in English-language scientific literature according to our knowledge). This helps to better understand
the limited range of applicability of the curse-of-dimensionality-free approaches in our study and others. The results of our numerical
investigation are presented in Section~5.

Let us also indicate some basic notations that are used throughout the paper:
\begin{itemize}
\item  given integer numbers~$ j_1 $ and $ j_2 \geqslant j_1 $, we write $ i = \overline{j_1, j_2} $ instead of
$ \: i \, = \, j_1, j_1 + 1, \ldots, j_2 $;
\item  given $ j \in \mathbb{N} $, the origin in $ \mathbb{R}^j $ is denoted by $ O_j $ or simply $ 0 $ (if there is
no confusion with zeros in other spaces), $ \: \| \cdot \| = \| \cdot \|_j = \| \cdot \|_{\mathbb{R}^j} \: $ is
the Euclidean norm in $ \mathbb{R}^j $ (some other kinds of vector norms are specified in Section~4),
$ \| \cdot \|_{\mathbb{R}^{j \times j}} $ is the matrix norm in $ \mathbb{R}^{j \times j} $ induced by
the Euclidean vector norm, and, for vectors $ \: v_1, v_2 \, \in \, \mathbb{R}^j, \: $ we write $ v_1 \uparrow\uparrow v_2 $
if they have the same direction, i.\,e., if $ v_2 = \alpha v_1 $ for some $ \alpha > 0 $;
\item  if a vector variable~$ \xi $ consists of some arguments of a map $ \: F \, = \, F(\ldots, \xi, \ldots), \: $ then
$ \mathrm{D}_{\xi} F $ denotes the standard (Fr\'echet) derivative of $ F $ with respect to $ \xi $, and $ \mathrm{D} F $ is
the standard derivative with respect to the vector of all arguments (the exact definitions of the derivatives depend on
the domain and range of $ F $);
\item  given a function $ \: F: \, \Xi_1 \to \mathbb{R}, \: $ the sets of its global minimizers and maximizers on
$ \Xi \subseteq \Xi_1 $ are denoted by $ \: \mathrm{Arg} \min_{\xi \in \Xi} \, F(\xi) \: $ and
$ \: \mathrm{Arg} \max_{\xi \in \Xi} \, F(\xi), \: $ respectively, while the criteria for the related optimization problems are
written as $ \: F(\xi) \, \longrightarrow \, \inf_{\xi \in \Xi} \: $ (or
$ \: F(\xi) \, \longrightarrow \, \min_{\xi \in \Xi} \: $ if the minimum is reached),
$ \: F(\xi) \, \longrightarrow \, \sup_{\xi \in \Xi} \: $ (or
$ \: F(\xi) \, \longrightarrow \, \max_{\xi \in \Xi} \: $ if the maximum is reached);
\item  the convex hull of a set~$ M $ in some linear space is denoted by $ \, \mathrm{conv} \: M, \, $ and the effective domain of
a convex function $ \: F: \, M \, \to \, \mathbb{R} \cup \{ +\infty \} \: $ is written as
$$
\mathrm{dom} \: F \:\: \stackrel{\mathrm{def}}{=} \:\: \{ \xi \in M \: \colon \: F(\xi) \, < \, +\infty \};
$$
\item  the Minkowski sum of two sets $ M_1, M_2 $ in some linear space is defined as
$$
M_1 + M_2 \:\: \stackrel{\mathrm{def}}{=} \:\: \{ \xi_1 + \xi_2 \: \colon \: \xi_1 \in M_1, \:\: \xi_2 \in M_2 \},
$$
and, if $ M_1 = \{ \xi \} $ is singleton, it is convenient to write $ \xi + M_2 $ instead of $ \{ \xi \} + M_2 $.
\end{itemize}
Other notations are introduced as needed.

\section{A curse-of-dimensionality-free characteristics approach for solving Hamilton--Jacobi equations under \\ smoothness and
convexity/concavity conditions on \\ Hamiltonians}

Given a fixed finite time horizon~$ T \in (0, +\infty) $, consider the Cauchy problem for a general Hamilton--Jacobi equation
\begin{equation}
\frac{\partial V(t, x)}{\partial t} \: + \: \mathcal{H}(t, \, x, \, \mathrm{D}_x V(t, x)) \:\: = \:\: 0, \quad
(t, x) \: \in \: (0, T) \times \mathbb{R}^n,  \label{F_1_1}
\end{equation}
\begin{equation}
V(T, x) \: = \: \sigma(x), \quad x \in \mathbb{R}^n.  \label{F_1_2}
\end{equation}

Let the following basic conditions hold.

\begin{assumption}  \label{Ass_1}
The following properties hold{\rm :}
\begin{list}{\rm \arabic{count})}%
{\usecounter{count}}
\item  the Hamiltonian
\begin{equation}
[0, T] \times \mathbb{R}^n \times \mathbb{R}^n \: \ni \: (t, x, p) \:\: \longmapsto \:\:
\mathcal{H}(t, x, p) \: \in \: \mathbb{R}  \label{F_2}
\end{equation}
is continuous{\rm ;}
\item  there exist positive constants~$ C_1, C_2 $ such that{\rm ,} for all
$ \: (t, x) \, \in \, [0, T] \times \mathbb{R}^n \: $ and $ \: p', p'' \, \in \, \mathbb{R}^n, \: $
we have
\begin{equation}
|\mathcal{H}(t, x, p') \: - \: \mathcal{H}(t, x, p'')| \:\: \leqslant \:\: C_1 \, (1 + \| x \|) \, \| p' - p'' \|,  \label{F_3}
\end{equation}
and
\begin{equation}
|\mathcal{H}(t, x, 0)| \: \leqslant \: C_2 \, (1 + \| x \|);  \label{F_4}
\end{equation}
\item  for any compact set $ K \subset \mathbb{R}^n ${\rm ,} there exists a number~$ C_3(K) > 0 $ such that{\rm ,} for all
$ \: (t, p) \, \in \, [0, T] \times \mathbb{R}^n \: $ and $ \: x', x'' \, \in \, K, \: $
we have
\begin{equation}
|\mathcal{H}(t, x', p) \: - \: \mathcal{H}(t, x'', p)| \:\: \leqslant \:\: C_3(K) \, (1 + \| p \|) \, \| x' - x'' \|;
\label{F_5}
\end{equation}
\item  the terminal function
\begin{equation}
\mathbb{R}^n \: \ni \: x \:\: \longmapsto \:\: \sigma(x) \: \in \: \mathbb{R}  \label{F_6}
\end{equation}
is continuous.
\end{list}
\end{assumption}

For the considerations of this section, it is convenient to understand a generalized solution of
the problem~(\ref{F_1_1}),~(\ref{F_1_2}) in the minimax sense~\cite[Chapter~II]{Subbotin1995}. A continuous function
$$
(0, T] \times \mathbb{R}^n \: \ni \: (t, x) \:\: \longmapsto \:\: V(t, x) \: \in \: \mathbb{R}
$$
is called a minimax solution of (\ref{F_1_1}),~(\ref{F_1_2}) if it fulfills the terminal condition~(\ref{F_1_2}) and
the following: for any $ \: (t_0, x_0) \, \in \, (0, T) \times \mathbb{R}^n \: $ and $ p \in \mathbb{R}^n $,
there exist a time~$ t_1 \in (t_0, T) $ and a Lipschitz continuous function
$$
[t_0, t_1] \: \ni \: t \:\: \longmapsto \:\: (x(t), \, z(t)) \: \in \: \mathbb{R}^n \times \mathbb{R}
$$
such that
$$
(x(t_0), \, z(t_0)) \: = \: (x_0, \, V(t_0, x_0)), \quad z(t) \: = \: V(t, x(t))
$$
for all $ t \in [t_0, t_1] $ and the equation
$$
\dot{z}(t) \:\: = \:\: \left< \dot{x}(t), p \right> \: - \: \mathcal{H}(t, x(t), p)
$$
holds for almost every $ t \in [t_0, t_1] $.

Let us introduce a fundamental result on the existence of a unique minimax solution of (\ref{F_1_1}),~(\ref{F_1_2}).

\begin{theorem}{\rm \cite[Theorem~II.8.1]{Subbotin1995}}  \label{Thm_2}
Under Assumption~{\rm \ref{Ass_1},} there exists a unique minimax solution of
the Cauchy problem~{\rm (\ref{F_1_1}),~(\ref{F_1_2})}.
\end{theorem}

\begin{remark}  \label{Rem_3}  \rm
In the literature concerning Hamilton--Jacobi equations{\rm ,} the notion of minimax solutions is used less frequently than
the notion of viscosity solutions {\rm \cite{Melikyan1998,YongZhou1999,FlemingSoner2006,BardiCapuzzoDolcetta2008,Yong2015},}
although they are in principle equivalent. Indeed{\rm ,} due to {\rm \cite[\S I.4]{Subbotin1995}} and
{\rm \cite[\S II.4]{FlemingSoner2006},} the unique minimax solution in Theorem~{\rm \ref{Thm_2}} coincides with the unique
viscosity solution of the formally rewritten Cauchy problem
\begin{equation}
\left\{ \begin{aligned}
& -\frac{\partial V(t, x)}{\partial t} \: - \: \mathcal{H}(t, \, x, \, \mathrm{D}_x V(t, x)) \:\: = \:\: 0, \quad
(t, x) \: \in \: (0, T) \times \mathbb{R}^n, \\
& V(T, x) \: = \: \sigma(x), \quad x \in \mathbb{R}^n.
\end{aligned} \right.  \label{F_7}
\end{equation}
Note also that{\rm ,} if the problem~{\rm (\ref{F_1_1}),~(\ref{F_1_2})} is considered not in the whole
set~$ (0, T] \times \mathbb{R}^n $ but just in the subregion~$ (0, T] \times G $ with an open
domain~$ G \subset \mathbb{R}^n ${\rm ,} then the corresponding minimax/viscosity solution can be similarly defined.
\qed
\end{remark}

Let us briefly describe an approach for representing the minimax solution of (\ref{F_1_1}),~(\ref{F_1_2})
through boundary value problems for the related characteristic system under some additional conditions (see, for instance,
\cite[\S II.10.5, \S II.10.6]{Subbotin1995} and \cite[\S 7.3]{Subbotina2006}).

First, the convexity of the Hamiltonian with respect to the adjoint variable is required.

\begin{assumption}  \label{Ass_5}
For any $ \: (t, x) \, \in \, [0, T] \times \mathbb{R}^n, \: $ the reduction
$ \: \mathbb{R}^n \ni p \: \longmapsto \: \mathcal{H}(t, x, p) \: $ is convex.
\end{assumption}

Let $ \: (t, x) \, \in \, [0, T] \times \mathbb{R}^n $. For $ \mathcal{H}(t, x, \cdot) $, introduce the conjugate function
(convex dual)
\begin{equation}
\mathbb{R}^n \: \ni \: f \:\: \longmapsto \:\: \mathcal{H}^*(t, x, f) \:\: \stackrel{\mathrm{def}}{=} \:\:
\sup_{p \, \in \, \mathbb{R}^n} \{ \left< f, p \right> \: - \: \mathcal{H}(t, x, p) \}
\label{F_9}
\end{equation}
and its effective domain
\begin{equation}
\mathrm{dom} \: \mathcal{H}^*(t, x, \cdot) \:\: = \:\: \left\{ f \in \mathbb{R}^n \: \colon \:
\mathcal{H}^*(t, x, f) \, < \, +\infty \right\}.  \label{F_10}
\end{equation}
By virtue of the condition~(\ref{F_3}) and Assumption~\ref{Ass_5}, the set~(\ref{F_10}) is nonempty, bounded and
convex (recall the affine support properties of convex functions \cite{Rockafellar1970,Rockafellar1974}).

For the sake of simplicity, we also use the next technical assumption, even though it is in fact not essential
\cite[\S II.10.5]{Subbotin1995}.

\begin{assumption}  \label{Ass_6}
There exists a continuous function
$$
[0, T] \times \mathbb{R}^n \: \ni \: (t, x) \:\: \longmapsto \:\: \beta(t, x) \: \in \: \mathbb{R}
$$
such that
$$
\mathcal{H}^*(t, x, f) \: \leqslant \: \beta(t, x) \quad \forall \:
f \: \in \: \mathrm{dom} \: \mathcal{H}^*(t, x, \cdot) \quad
\forall \: (t, x) \: \in \: [0, T] \times \mathbb{R}^n.
$$
\end{assumption}

Consider the differential inclusion
\begin{equation}
\begin{aligned}
\left( \dot{x}(t), \, \dot{z}(t) \right) \:\: \in \:\: \{ (f, g) \: \in \: \mathbb{R}^n \times \mathbb{R} \: \colon \:
& f \: \in \: \mathrm{dom} \: \mathcal{H}^*(t, x(t), \cdot), \\
& \mathcal{H}^*(t, x(t), f) \: \leqslant \: g \: \leqslant \: \beta(t, x(t)) \}.
\end{aligned}  \label{F_12}
\end{equation}
For $ \: (t_0, x_0) \, \in \, [0, T] \times \mathbb{R}^n, \: $ let $ \mathcal{S}(t_0, x_0) $ be the set of its trajectories
$$
[0, T] \: \ni \: t \:\: \longmapsto \:\: (x(t), z(t)) \: \in \: \mathbb{R}^n \times \mathbb{R}
$$
satisfying the conditions
\begin{equation}
x(t_0) \, = \, x_0, \quad z(T) \, = \, \sigma(x(T)),  \label{F_13}
\end{equation}
and introduce the functional
\begin{equation}
\gamma(t_0, \, x(\cdot), \, z(\cdot)) \:\: \stackrel{\mathrm{def}}{=} \:\: \sigma(x(T)) \: - \:
\int\limits_{t_0}^T \dot{z}(t) \, dt \quad \forall \: (x(\cdot), \, z(\cdot)) \: \in \: \mathcal{S}(t_0, x_0).
\label{F_14}
\end{equation}

\begin{theorem}{\rm \cite[\S II.10.5]{Subbotin1995}}  \label{Thm_7}
Under Assumptions~{\rm \ref{Ass_1}, \ref{Ass_5}, \ref{Ass_6},} the function
\begin{equation}
V(t_0, x_0) \:\: \stackrel{\mathrm{def}}{=} \:\: \max_{(x(\cdot), \, z(\cdot)) \: \in \: \mathcal{S}(t_0, x_0)}
\gamma(t_0, \, x(\cdot), \, z(\cdot)) \quad \forall \: (t_0, x_0) \: \in \: [0, T] \times \mathbb{R}^n
\label{F_15}
\end{equation}
is the unique minimax solution of the Cauchy problem~{\rm (\ref{F_1_1}),~(\ref{F_1_2})}.
\end{theorem}

\begin{remark}  \label{Rem_8}  \rm
If convexity is replaced with concavity in Assumption~{\rm \ref{Ass_5},} then a similar characterization of the minimax
solution of {\rm (\ref{F_1_1}), (\ref{F_1_2})} can be obtained{\rm ,} and minimization appears instead of maximization
in {\rm (\ref{F_15})}. \qed
\end{remark}

Now impose some smoothness conditions on the Hamiltonian and terminal function.

\begin{assumption}  \label{Ass_9}
The Hamiltonian~{\rm (\ref{F_2})} and terminal function~{\rm (\ref{F_6})} are continuously differentiable{\rm ,} and
the derivatives $ \mathrm{D}_{tx} \mathcal{H} ${\rm ,} $ \mathrm{D}_{tp} \mathcal{H} $ exist for all
$ \: (t, x, p) \, \in \, [0, T] \times \mathbb{R}^n \times \mathbb{R}^n $.
\end{assumption}

The next statement can be derived by using basic results of convex analysis \cite{Rockafellar1970,Rockafellar1974}.

\begin{proposition}  \label{Pro_10}
Under Assumptions~{\rm \ref{Ass_1}, \ref{Ass_5}, \ref{Ass_6}, \ref{Ass_9},} the effective domain~{\rm (\ref{F_10})}
can be represented as
\begin{equation}
\begin{aligned}
& \mathrm{dom} \: \mathcal{H}^*(t, x, \cdot) \:\: = \:\: \mathrm{conv} \:
\{ \mathrm{D}_p \mathcal{H}(t, x, \psi) \: \colon \: \psi \in \mathbb{R}^n \} \\
& \forall \: (t, x) \: \in \: [0, T] \times \mathbb{R}^n,
\end{aligned}  \label{F_10_1}
\end{equation}
and the following formula also holds{\rm :}
\begin{equation}
\begin{aligned}
& \mathcal{H}^*(t, \, x, \, \mathrm{D}_p \mathcal{H}(t, x, \psi)) \:\: = \:\:
\left< \psi, \mathrm{D}_p \mathcal{H}(t, x, \psi) \right> \: - \: \mathcal{H}(t, x, \psi) \\
& \forall \: (t, x, \psi) \: \in \: [0, T] \times \mathbb{R}^n \times \mathbb{R}^n.
\end{aligned}  \label{F_10_2}
\end{equation}
\end{proposition}

For any $ y \in \mathbb{R}^n $, let
\begin{equation}
[0, T] \: \ni \: t \:\: \longmapsto \:\:
\left( \hat{x}(t; \, y), \: \hat{p}(t; \, y), \: \hat{z}(t; \, y) \right) \: \in \:
\mathbb{R}^n \times \mathbb{R}^n \times \mathbb{R}  \label{F_16}
\end{equation}
be the solution of the characteristic system
\begin{equation}
\left\{ \begin{aligned}
& \dot{\hat{x}}(t; \, y) \:\: = \:\: \mathrm{D}_p \mathcal{H} \left( t, \: \hat{x}(t; \, y), \: \hat{p}(t; \, y) \right), \\
& \dot{\hat{p}}(t; \, y) \:\: = \:\: -\mathrm{D}_x \mathcal{H} \left( t, \: \hat{x}(t; \, y), \: \hat{p}(t; \, y) \right), \\
& \dot{\hat{z}}(t; \, y) \:\: = \:\: \left< \hat{p}(t; \, y), \:
\mathrm{D}_p \mathcal{H} \left( t, \: \hat{x}(t; \, y), \: \hat{p}(t; \, y) \right) \right> \:\: - \:\:
\mathcal{H} \left( t, \: \hat{x}(t; \, y), \: \hat{p}(t; \, y) \right),
\end{aligned} \right.  \label{F_17}
\end{equation}
such that
\begin{equation}
\hat{x}(T; \, y) \: = \: y, \quad \hat{p}(T; \, y) \: = \: \mathrm{D} \sigma(y), \quad \hat{z}(T; \, y) \: = \: \sigma(y).
\label{F_18}
\end{equation}
Define the set
\begin{equation}
Y(t_0, x_0) \:\: \stackrel{\mathrm{def}}{=} \:\: \left\{ y \in \mathbb{R}^n \: \colon \:
\hat{x}(t_0; \, y) \: = \: x_0 \right\} \quad \forall \: (t_0, x_0) \: \in \: [0, T] \times \mathbb{R}^n.  \label{F_19}
\end{equation}

\begin{theorem}{\rm \cite[\S II.10.6]{Subbotin1995}}  \label{Thm_11}
Under Assumptions~{\rm \ref{Ass_1}, \ref{Ass_5}, \ref{Ass_6}, \ref{Ass_9},} the function
\begin{equation}
V(t_0, x_0) \:\: = \:\: \max_{y \: \in \: Y(t_0, x_0)} \hat{z}(t_0; \, y) \quad
\forall \: (t_0, x_0) \: \in \: [0, T] \times \mathbb{R}^n
\label{F_20}
\end{equation}
is the unique minimax solution of the Cauchy problem~{\rm (\ref{F_1_1}),~(\ref{F_1_2})}.
\end{theorem}

Such a representation was obtained first in \cite{Mirica1985} and independently in \cite{Subbotina1991,Subbotina1992}.

\begin{remark}  \label{Rem_12}  \rm
If convexity is replaced with concavity in Assumption~{\rm \ref{Ass_5},} then the formula~{\rm (\ref{F_20})} is rewritten as
\begin{equation}
V(t_0, x_0) \:\: = \:\: \min_{y \: \in \: Y(t_0, x_0)} \hat{z}(t_0; \, y) \quad
\forall \: (t_0, x_0) \: \in \: [0, T] \times \mathbb{R}^n
\label{F_21}
\end{equation}
{\rm (}minimization appears instead of maximization{\rm )}. \qed
\end{remark}

\begin{remark}  \label{Rem_12_0}  \rm
For numerical purposes{\rm ,} it is reasonable to parametrize characteristic fields not with respect to the terminal state but
with respect to the initial adjoint vector, so that Cauchy problems can be solved instead of boundary value problems. Indeed{\rm ,}
the latter may have multiple solutions{\rm ,} leading to the practical dilemma of finding a solution that provides the optimal cost.
The uniqueness of solutions of Cauchy problems avoids this dilemma. \qed
\end{remark}

The following result specifies the mentioned parametrization of characteristic fields with respect to the initial adjoint vector.

\begin{theorem}  \label{Thm_13}
Let Assumptions~{\rm \ref{Ass_1}, \ref{Ass_5}, \ref{Ass_6}, \ref{Ass_9}} hold. For any
$ \: (t_0, x_0, p_0) \, \in \, [0, T) \times \mathbb{R}^n \times \mathbb{R}^n, \: $ let
\begin{equation}
[t_0, T] \: \ni \: t \:\: \longmapsto \:\:
\left( \tilde{x}(t; \, t_0, x_0, p_0), \: \tilde{p}(t; \, t_0, x_0, p_0) \right) \: \in \: \mathbb{R}^n \times \mathbb{R}^n
\label{F_23}
\end{equation}
be the solution of the characteristic system
\begin{equation}
\left\{ \begin{aligned}
& \dot{\tilde{x}}(t; \, t_0, x_0, p_0) \:\: = \:\:
\mathrm{D}_p \mathcal{H} \left( t, \: \tilde{x}(t; \, t_0, x_0, p_0), \: \tilde{p}(t; \, t_0, x_0, p_0) \right), \\
& \dot{\tilde{p}}(t; \, t_0, x_0, p_0) \:\: = \:\:
-\mathrm{D}_x \mathcal{H} \left( t, \: \tilde{x}(t; \, t_0, x_0, p_0), \: \tilde{p}(t; \, t_0, x_0, p_0) \right), \\
& t \in [t_0, T],
\end{aligned} \right.  \label{F_24}
\end{equation}
such that
\begin{equation}
\tilde{x}(t_0; \, t_0, x_0, p_0) \: = \: x_0, \quad \tilde{p}(t_0; \, t_0, x_0, p_0) \: = \: p_0.  \label{F_25}
\end{equation}
Then the function defined by
\begin{equation}
\begin{aligned}
& V(T, x_0) \: = \: \sigma(x_0), \\
& V(t_0, x_0) \:\: = \:\: \max_{p_0 \, \in \, \mathbb{R}^n} \:
\left\{ \sigma \left( \tilde{x}(T; \, t_0, x_0, p_0) \right) \:\: -
{}^{{}^{{}^{{}^{{}^{{}^{{}^{{}^{{}^{}}}}}}}}} \right. \\
& \qquad\quad
- \:\: \int\limits_{t_0}^T \: \left( \left< \tilde{p}(t; \, t_0, x_0, p_0), \:
\mathrm{D}_p \mathcal{H} \left( t, \: \tilde{x}(t; \, t_0, x_0, p_0), \: \tilde{p}(t; \, t_0, x_0, p_0) \right) \right> \:\: -
\right. \\
& \qquad\qquad\qquad\qquad\qquad\quad \,\,
\left. {}^{{}^{{}^{{}^{{}^{{}^{{}^{{}^{{}^{}}}}}}}}}
\left. - \:\:
\mathcal{H} \left( t, \: \tilde{x}(t; \, t_0, x_0, p_0), \: \tilde{p}(t; \, t_0, x_0, p_0) \right) \right) \: dt \right\} \\
& \forall \: (t_0, x_0) \: \in \: [0, T) \times \mathbb{R}^n
\end{aligned}  \label{F_26}
\end{equation}
is the unique minimax solution of the Cauchy problem~{\rm (\ref{F_1_1}),~(\ref{F_1_2})}. If{\rm ,} moreover{\rm ,} the third
characteristic equation in {\rm (\ref{F_17})} does not explicitly depend on the adjoint variable{\rm ,} i.\,e.{\rm ,}
\begin{equation}
\mathcal{H}(t, x, \psi) \:\: = \:\: \left< \psi, \, \mathrm{D}_p \mathcal{H}(t, x, \psi) \right> \: + \: \eta(t, x) \quad
\forall \: (t, x, \psi) \: \in \: [0, T] \times \mathbb{R}^n \times \mathbb{R}^n  \label{F_22}
\end{equation}
for some function $ \: \eta \, \colon \, [0, T] \times \mathbb{R}^n \, \to \, \mathbb{R}, \: $ and the Hamiltonian~{\rm (\ref{F_2})}
satisfies the conditions
\begin{equation}
\begin{aligned}
& \mathrm{D}_x \mathcal{H} (t, x, \alpha \psi) \: = \: \alpha \, \mathrm{D}_x \mathcal{H} (t, x, \psi), \quad
\mathrm{D}_p \mathcal{H} (t, x, \alpha \psi) \: = \: \mathrm{D}_p \mathcal{H} (t, x, \psi) \\
& \forall \alpha > 0 \quad \forall \: (t, x, \psi) \: \in \: [0, T] \times \mathbb{R}^n \times \mathbb{R}^n,
\end{aligned}  \label{F_27_0}
\end{equation}
then the finite-dimensional maximization in {\rm (\ref{F_26})} can be performed over the union of the unit sphere and origin
in~$ \mathbb{R}^n ${\rm :}
\begin{equation}
\begin{aligned}
& V(t_0, x_0) \:\: = \:\: \max_{\substack{p_0 \, \in \, \mathbb{R}^n \: \colon \\
\| p_0 \| \, = \, 1 \:\: \mbox{\scriptsize \rm or} \:\: p_0 \, = \, 0}}
\left\{  \sigma \left( \tilde{x}(T; \, t_0, x_0, p_0) \right) \: + \:
\int\limits_{t_0}^T \eta \left( t, \, \tilde{x}(t; \, t_0, x_0, p_0) \right) \, dt \right\} \\
& \forall \: (t_0, x_0) \: \in \: [0, T) \times \mathbb{R}^n.
\end{aligned}  \label{F_27}
\end{equation}
\end{theorem}

\begin{proof}
The second part of the theorem is a direct corollary to the first part under the conditions~(\ref{F_22}) and (\ref{F_27_0}).
Indeed, (\ref{F_22}) yields that $ \: \tilde{p}(\cdot; \, t_0, x_0, p_0) \: $ does not explicitly appear in
the expression for the maximized functional in (\ref{F_26}), while (\ref{F_27_0}) guarantees that the right-hand
side of the adjoint system is positive homogeneous of degree~$ 1 $ with respect to the adjoint variable and that
the state components of the characteristic curves do not change after multiplying $ p_0 $ by any positive number.

Hence, it remains to establish the first part of the theorem. Compared to the boundary value problems~(\ref{F_17}),~(\ref{F_18})
in Theorem~\ref{Thm_11}, the Cauchy problems~(\ref{F_24}),~(\ref{F_25}) generate a wider characteristic field (due to the absence of
the terminal condition on the adjoint variable). Let $ \: (t_0, x_0, p_0) \, \in \, [0, T) \times \mathbb{R}^n \times \mathbb{R}^n \: $
and denote
\begin{equation}
\begin{aligned}
& \tilde{z}(t; \, t_0, x_0, p_0) \:\: \stackrel{\mathrm{def}}{=} \:\:
\sigma \left( \tilde{x}(T; \, t_0, x_0, p_0) \right) \:\: - \\
& \qquad\quad
- \:\: \int\limits_t^T \: \left( \left< \tilde{p}(s; \, t_0, x_0, p_0), \:
\mathrm{D}_p \mathcal{H} \left( s, \: \tilde{x}(s; \, t_0, x_0, p_0), \: \tilde{p}(s; \, t_0, x_0, p_0) \right) \right> \:\: -
\right. \\
& \qquad\qquad\qquad\qquad\qquad\qquad\quad
\left. - \:\:
\mathcal{H} \left( s, \: \tilde{x}(s; \, t_0, x_0, p_0), \: \tilde{p}(s; \, t_0, x_0, p_0) \right) \right) \: ds \\
& \forall t \in [t_0, T]
\end{aligned}  \label{F_27_1}
\end{equation}
(if $ t = t_0 $, this is the expression for the maximized functional in (\ref{F_26})). We have
\begin{equation}
\tilde{z}(T; \, t_0, x_0, p_0) \:\: = \:\: \sigma \left( \tilde{x}(T; \, t_0, x_0, p_0) \right)  \label{F_27_2}
\end{equation}
and
\begin{equation}
\begin{aligned}
& \dot{\tilde{z}}(t; \, t_0, x_0, p_0) \:\: = \:\: \left< \tilde{p}(t; \, t_0, x_0, p_0), \:
\mathrm{D}_p \mathcal{H} \left( t, \: \tilde{x}(t; \, t_0, x_0, p_0), \: \tilde{p}(t; \, t_0, x_0, p_0) \right) \right> \:\: - \\
& \qquad\qquad\qquad\qquad\qquad\qquad\qquad\qquad \:
- \:\: \mathcal{H} \left( t, \: \tilde{x}(t; \, t_0, x_0, p_0), \: \tilde{p}(t; \, t_0, x_0, p_0) \right) \\
& \forall t \in [t_0, T].
\end{aligned}  \label{F_27_3}
\end{equation}
Consider also the solution~(\ref{F_23}) to (\ref{F_24}),~(\ref{F_25}). According to Theorem~\ref{Thm_7}, it suffices to show that
$ \: \left( \tilde{x}(\cdot; \, t_0, x_0, p_0), \, \tilde{z}(\cdot; \, t_0, x_0, p_0) \right) \: $ is a solution of
the differential inclusion~(\ref{F_12}) almost everywhere on $ [t_0, T] $ (the related initial and terminal conditions~(\ref{F_13})
are trivially satisfied by virtue of (\ref{F_25}) and (\ref{F_27_2})). The formulae~(\ref{F_10_1}) and (\ref{F_10_2}) in
Proposition~{\ref{Pro_10}} can be used for this purpose. From (\ref{F_24}) and (\ref{F_10_1}), we get
\begin{equation}
\begin{aligned}
\dot{\tilde{x}}(t; \, t_0, x_0, p_0) \:\: & = \:\:
\mathrm{D}_p \mathcal{H} \left( t, \: \tilde{x}(t; \, t_0, x_0, p_0), \: \tilde{p}(t; \, t_0, x_0, p_0) \right) \\
& \in \:\: \mathrm{dom} \: \mathcal{H}^* \left( t, \: \tilde{x}(t; \, t_0, x_0, p_0), \: \cdot \right)
\end{aligned}  \label{F_27_4}
\end{equation}
for almost every $ t \in [t_0, T] $. Due to (\ref{F_27_3}) and (\ref{F_10_2}), we obtain
\begin{equation}
\dot{\tilde{z}}(t; \, t_0, x_0, p_0) \:\: = \:\: \mathcal{H}^* \left( t, \:\, \tilde{x}(t; \, t_0, x_0, p_0), \:\,
\mathrm{D}_p \mathcal{H} \left( t, \: \tilde{x}(t; \, t_0, x_0, p_0), \: \tilde{p}(t; \, t_0, x_0, p_0) \right) \right)
\label{F_27_5}
\end{equation}
for all $ t \in [t_0, T] $. Finally, the sought-after property directly follows from (\ref{F_27_4}) and (\ref{F_27_5}).
\end{proof}

\begin{remark}  \label{Rem_14}  \rm
The representation~{\rm (\ref{F_22})} is typical for many optimal control problems with smooth Hamiltonians
{\rm (}see Section~{\rm 3)}. \qed
\end{remark}

\begin{remark}  \label{Rem_14_0}  \rm
For optimal control problems with smooth Hamiltonians{\rm ,} the second condition in {\rm (\ref{F_27_0})} appears to be
rather strict{\rm ,} but it allows an extension to a wide class of optimal control problems with Mayer cost functionals and
nonsmooth Hamiltonians {\rm (}see Theorem~{\rm \ref{Thm_20})}. \qed
\end{remark}

\section{A curse-of-dimensionality-free characteristics approach for solving Hamilton--Jacobi--Bellman equations in \\
optimal control problems}

Let $ G $ and $ U $ be sets in the state and control spaces, respectively. Consider the control system
\begin{equation}
\left\{ \begin{aligned}
& \dot{x}(t) \: = \: f(t, x(t), u(t)), \quad t \in [t_0, T], \\
& x(t_0) \, = \, x_0 \, \in \, G \:\:\, \mbox{is fixed}, \\
& T \in (0, +\infty) \:\:\, \mbox{and} \:\:\, t_0 \in [0, T) \:\:\, \mbox{are fixed}, \\
& u(\cdot) \, \in \, \mathcal{U}_{t_0, \, T} \, , \\
& \mathcal{U}_{t_0, \, T} \:\: \mbox{is the class of measurable functions defined on $ [t_0, T] $ with values in} \:\: U,
\end{aligned} \right.  \label{F_28}
\end{equation}
and the optimization criterion
\begin{equation}
\begin{aligned}
J_{t_0, \, T, \, x_0}(u(\cdot)) \:\: \stackrel{\mathrm{def}}{=} \:\: & \sigma(x(T; \, t_0, x_0, u(\cdot))) \: + \\
& + \: \int\limits_{t_0}^T \eta(t, \: x(t; \, t_0, x_0, u(\cdot)), \: u(t)) \: dt \:\: \longrightarrow \:\:
\inf_{u(\cdot) \: \in \: \mathcal{U}_{t_0, \, T}} \, ,
\end{aligned}  \label{F_29}
\end{equation}
where $ \: x(\cdot; \, t_0, x_0, u(\cdot)) \: $ denotes a solution to the given Cauchy problem for the system of ordinary
differential equations with a control function $ u(\cdot) $. Define the value function of interest by
\begin{equation}
\begin{aligned}
& V(T, x_0) \: \stackrel{\mathrm{def}}{=} \: \sigma(x_0), \\
& V(t_0, x_0) \: \stackrel{\mathrm{def}}{=} \: \inf_{u(\cdot) \: \in \: \mathcal{U}_{t_0, \, T}} J_{t_0, \, T, \, x_0}(u(\cdot)) \\
& \forall \: (t_0, x_0) \: \in \: [0, T) \times G.
\end{aligned}  \label{F_30}
\end{equation}

First, let us formulate a general existence result and first-order necessary optimality conditions, i.\,e.,
Pontryagin's principle \cite{Pontryagin1964}. Some basic assumptions are adopted.

\begin{assumption}  \label{Ass_15}
The following properties hold{\rm :}
\begin{list}{\rm \arabic{count})}%
{\usecounter{count}}
\item  $ U \subseteq \mathbb{R}^m $ is a closed set in the control space{\rm ;}
\item  $ G $ and $ G_1 $ are open domains in the state space~$ \mathbb{R}^n ${\rm ,} and $ G_1 $ contains
the closure~$ \bar{G} $ of~$ G ${\rm ;}
\item  the functions
\begin{equation}
\arraycolsep=1.5pt
\def\arraystretch{2}
\begin{array}{c}
[0, T] \times G_1 \times U \: \ni \: (t, x, u) \:\: \longmapsto \:\: f(t, x, u) \: \in \: \mathbb{R}^n, \\ \relax
[0, T] \times G_1 \times U \: \ni \: (t, x, u) \:\: \longmapsto \:\: \eta(t, x, u) \: \in \: \mathbb{R}, \\
G_1 \: \ni \: x \:\: \longmapsto \:\: \sigma(x) \: \in \: \mathbb{R}
\end{array}  \label{F_31}
\end{equation}
are continuous{\rm ;}
\item  $ G $ is a strongly invariant domain in the state space for the control system~{\rm (\ref{F_28}),} i.\,e.{\rm ,}
$ \: (t_0, x_0) \, \in \, [0, T) \times G \: $ and $ \, u(\cdot) \in \mathcal{U}_{t_0, \, T} \, $ imply that any
corresponding state trajectory $ \, x(\cdot; \, t_0, x_0, u(\cdot)) \, $ defined on a subinterval of $ [t_0, T] $
stays inside $ G $ and cannot reach the boundary~$ \partial G $ {\rm (}$ G = \mathbb{R}^n $ is a trivial example of
a strongly invariant domain{\rm );}
\item  there exist an~$ a \in G $ and positive constants~$ A_1, A_2 $ such that
\begin{equation}
\| f(t, a, u) \|_{\mathbb{R}^n} \: \leqslant \: A_1, \quad |\eta(t, a, u)| \: \leqslant \: A_2 \quad
\forall \: (t, u) \: \in \: [0, T] \times U;  \label{F_31_0}
\end{equation}
\item  if $ U $ is unbounded{\rm ,} then there exists a constant~$ C_1 > 0 $ and a modulus of continuity
$ \: \omega_1 \, \colon \, [0, +\infty) \to [0, +\infty) \: $ such that
\begin{equation}
\begin{aligned}
& \| f(t, x', u') \: - \: f(t, x'', u'') \|_{\mathbb{R}^n} \:\: \leqslant \:\: C_1 \, \| x' - x'' \|_{\mathbb{R}^n} \: + \:
\omega_1 \left( \| u' - u'' \|_{\mathbb{R}^m} \right) \\
& \forall t \in [0, T] \quad \forall \: x', x'' \, \in \, \bar{G} \quad \forall \: u', u'' \, \in \, U;
\end{aligned}  \label{F_31_1}
\end{equation}
\item  if $ U $ is bounded{\rm ,} then the condition~{\rm (\ref{F_31_1})} is relaxed so that there exists
a constant~$ C_1 > 0 $ satisfying
\begin{equation}
\begin{aligned}
& \| f(t, x', u) \: - \: f(t, x'', u) \|_{\mathbb{R}^n} \:\: \leqslant \:\: C_1 \, \| x' - x'' \|_{\mathbb{R}^n} \\
& \forall t \in [0, T] \quad \forall \: x', x'' \, \in \, \bar{G} \quad \forall u \in U;
\end{aligned}  \label{F_31_2}
\end{equation}
\item  if $ U $ is unbounded{\rm ,} then there exist positive constants~$ C_2, C_3 $ and a modulus of continuity
$ \: \omega_2 \, \colon \, [0, +\infty) \to [0, +\infty) \: $ such that
\begin{equation}
\begin{aligned}
& |\eta(t, x', u') \: - \: \eta(t, x'', u'')| \:\: \leqslant \:\: C_2 \, \| x' - x'' \|_{\mathbb{R}^n} \: + \:
\omega_2 \left( \| u' - u'' \|_{\mathbb{R}^m} \right), \\
& |\sigma(x') \: - \: \sigma(x'')| \:\: \leqslant \:\: C_3 \, \| x' - x'' \|_{\mathbb{R}^n} \\
& \forall t \in [0, T] \quad \forall \: x', x'' \, \in \, \bar{G} \quad \forall \: u', u'' \, \in \, U;
\end{aligned}  \label{F_31_3}
\end{equation}
\item  if $ U $ is bounded{\rm ,} then the conditions~{\rm (\ref{F_31_3})} are relaxed so that{\rm ,} for any compact
set~$ K \subseteq \bar{G} ${\rm ,} there exist positive numbers~$ C_{2, K}, C_{3, K} $ satisfying
\begin{equation}
\begin{aligned}
& |\eta(t, x', u) \: - \: \eta(t, x'', u)| \:\: \leqslant \:\: C_{2, K} \, \| x' - x'' \|_{\mathbb{R}^n}, \\
& |\sigma(x') \: - \: \sigma(x'')| \:\: \leqslant \:\: C_{3, K} \, \| x' - x'' \|_{\mathbb{R}^n} \\
& \forall t \in [0, T] \quad \forall \: x', x'' \, \in \, K \quad \forall \: u \in U;
\end{aligned}  \label{F_31_4}
\end{equation}
\item  the set
$$
[f, \eta](t, x, U) \:\: \stackrel{\mathrm{def}}{=} \:\: \{ (f(t, x, u), \, \eta(t, x, u)) \: \colon \: u \in U \} \:\: \subseteq \:\:
\mathbb{R}^{n + 1}
$$
is convex and closed for all $ \: (t, x) \, \in \, [0, T] \times G_1 ${\rm ;}
\item  the infimum in {\rm (\ref{F_30})} is finite for every initial position $ \: (t_0, x_0) \, \in \, [0, T) \times G $.
\end{list}
\end{assumption}

\begin{remark}  \label{Rem_15_0}  \rm
For any $ t_0 \in [0, T) $ and $ X_0 \subseteq G ${\rm ,} let $ \: \mathcal{W}(t_0, T, X_0) \subseteq \mathbb{R}^{n + 1} \: $ be
the integral funnel of the system~{\rm (\ref{F_28})} for all $ \, x(t_0) = x_0 \in X_0 \, $ and for all
$ u(\cdot) \in \mathcal{U}_{t_0, \, T} $ on the time interval~$ [t_0, T] ${\rm ,} i.\,e.{\rm ,}
\begin{equation}
\mathcal{W}(t_0, T, X_0) \:\: \stackrel{\mathrm{def}}{=} \:\: \{ (t, \: x(t; \, t_0, x_0, u(\cdot))) \: \colon \:
t \in [t_0, T], \:\: x_0 \in X_0, \:\: u(\cdot) \, \in \, \mathcal{U}_{t_0, \, T} \}.  \label{F_31_11}
\end{equation}
If $ X_0 = \{ x_0 \} $ is singleton{\rm ,} let us write $ \mathcal{W}(t_0, T, x_0) $ instead of $ \mathcal{W}(t_0, T, \{ x_0 \}) $.
Note that the last item in Assumption~{\rm \ref{Ass_15}} directly follows from the previous items in such cases as{\rm :}
\begin{itemize}
\item  boundedness below of the functions~$ \sigma, \eta ${\rm ;}
\item  boundedness of the funnel~$ \mathcal{W}(t_0, T, x_0) $ for all $ \: (t_0, x_0) \, \in \, [0, T) \times G \: $
together with either boundedness of $ U $ or boundedness below of $ \eta $ or independence of $ \eta $ from $ u $.
\end{itemize}
\qed
\end{remark}

\begin{assumption}  \label{Ass_17}
The following properties hold{\rm :}
\begin{list}{\rm \arabic{count})}%
{\usecounter{count}}
\item  the functions~{\rm (\ref{F_31})} are continuously differentiable with respect to $ x \in G_1 $ for all
$ \: (t, u) \, \in \, [0, T] \times U ${\rm ;}
\item  if $ U $ is unbounded{\rm ,} then there exist moduli of continuity
$ \: \omega_i \, \colon \, [0, +\infty) \to [0, +\infty) ${\rm ,} $ i = \overline{3, 5}, \: $ such that
\begin{equation}
\begin{aligned}
& \| \mathrm{D}_x f(t, x', u') \: - \: \mathrm{D}_x f(t, x'', u'') \|_{\mathbb{R}^{n \times n}} \:\: \leqslant \:\:
\omega_3 \left( \| x' - x'' \|_{\mathbb{R}^n} \: + \: \| u' - u'' \|_{\mathbb{R}^m} \right), \\
& \| \mathrm{D}_x \eta(t, x', u') \: - \: \mathrm{D}_x \eta(t, x'', u'') \|_{\mathbb{R}^n} \:\: \leqslant \:\:
\omega_4 \left( \| x' - x'' \|_{\mathbb{R}^n} \: + \: \| u' - u'' \|_{\mathbb{R}^m} \right), \\
& \| \mathrm{D} \sigma(x') \: - \: \mathrm{D} \sigma(x'') \|_{\mathbb{R}^n} \:\: \leqslant \:\:
\omega_5 \left( \| x' - x'' \|_{\mathbb{R}^n} \right) \\
& \forall t \in [0, T] \quad \forall \: x', x'' \, \in \, \bar{G} \quad \forall \: u', u'' \, \in \, U;
\end{aligned}  \label{F_31_5}
\end{equation}
\item  if $ U $ is bounded{\rm ,} then the conditions~{\rm (\ref{F_31_5})} are relaxed so that{\rm ,} for any compact
set~$ K \subseteq \bar{G} ${\rm ,} there exist moduli of continuity
$ \: \omega_{i, K} \, \colon \, [0, +\infty) \to [0, +\infty) ${\rm ,} $ i = \overline{3, 5}, \: $ satisfying
\begin{equation}
\begin{aligned}
& \| \mathrm{D}_x f(t, x', u') \: - \: \mathrm{D}_x f(t, x'', u'') \|_{\mathbb{R}^{n \times n}} \:\: \leqslant \:\:
\omega_{3, K} \left( \| x' - x'' \|_{\mathbb{R}^n} \: + \: \| u' - u'' \|_{\mathbb{R}^m} \right), \\
& \| \mathrm{D}_x \eta(t, x', u') \: - \: \mathrm{D}_x \eta(t, x'', u'') \|_{\mathbb{R}^n} \:\: \leqslant \:\:
\omega_{4, K} \left( \| x' - x'' \|_{\mathbb{R}^n} \: + \: \| u' - u'' \|_{\mathbb{R}^m} \right), \\
& \| \mathrm{D} \sigma(x') \: - \: \mathrm{D} \sigma(x'') \|_{\mathbb{R}^n} \:\: \leqslant \:\:
\omega_{5, K} \left( \| x' - x'' \|_{\mathbb{R}^n} \right) \\
& \forall t \in [0, T] \quad \forall \: x', x'' \, \in \, K \quad \forall \: u', u'' \, \in \, U.
\end{aligned}  \label{F_31_6}
\end{equation}
\end{list}
\end{assumption}

Assumptions~\ref{Ass_15} and \ref{Ass_17} contain somewhat relaxed versions of the conditions that were imposed in
\cite{YongZhou1999} for establishing an existence theorem and Pontryagin's principle for deterministic optimal control
problems. The next remark explains the validity of these relaxations.

\begin{remark}  \label{Rem_15_1}  \rm
In {\rm \cite[\S 2.5.1, \S 3.2, \S 4.2]{YongZhou1999},} the case $ G = \mathbb{R}^n $ is considered{\rm ,} and the Lipschitz type
conditions~{\rm (\ref{F_31_1}), (\ref{F_31_3}), (\ref{F_31_5})} are imposed for all $ \: x', x'' \, \in \, \mathbb{R}^n $.
In fact{\rm ,} a strongly invariant domain~$ G $ can be taken into account if it exists. Then the reasonings of
{\rm \cite[\S 2.5.1, \S 3.2, \S 4.2]{YongZhou1999}} can still be used. Only the items~{\rm 7,9} of
Assumption~{\rm \ref{Ass_15}} and the item~{\rm 3} of Assumption~{\rm \ref{Ass_17}} need an additional clarification.
Let $ U $ be bounded. First{\rm ,} note that{\rm ,} since $ f $ is uniformly continuous on every compact subset of
$ \, [0, T] \times G_1 \times U \, $ {\rm (}due to the item~{\rm 3} of Assumption~{\rm \ref{Ass_15}),}
the conditions~{\rm (\ref{F_31_2}), (\ref{F_31_4})} imply the following relaxations of the conditions on $ f, \eta $ in
{\rm (\ref{F_31_1}), (\ref{F_31_3}):} for any compact set~$ K \subseteq \bar{G} ${\rm ,} there exist
moduli of continuity $ \: \omega_{i, K} \, \colon \, [0, +\infty) \to [0, +\infty) ${\rm ,} $ i = 1,2, \: $ satisfying
\begin{equation}
\begin{aligned}
& \| f(t, x', u') \: - \: f(t, x'', u'') \|_{\mathbb{R}^n} \:\: \leqslant \:\: C_1 \, \| x' - x'' \|_{\mathbb{R}^n} \: + \:
\omega_{1, K} \left( \| u' - u'' \|_{\mathbb{R}^m} \right), \\
& |\eta(t, x', u') \: - \: \eta(t, x'', u'')| \:\: \leqslant \:\: C_{2, K} \, \| x' - x'' \|_{\mathbb{R}^n} \: + \:
\omega_{2, K} \left( \| u' - u'' \|_{\mathbb{R}^m} \right) \\
& \forall t \in [0, T] \quad \forall \: x', x'' \, \in \, K \quad \forall \: u', u'' \, \in \, U.
\end{aligned}  \label{F_31_7}
\end{equation}
Next{\rm ,} it is reasonable to relax the conditions~{\rm (\ref{F_31_3}), (\ref{F_31_5})} if
the functions~$ \: \eta, \, \sigma, \, \mathrm{D}_x f, \, \mathrm{D}_x \eta, \, \mathrm{D}_x \sigma \: $ are not
necessarily uniformly continuous with respect to $ x \in \bar{G} $. In order to justify the replacement of
{\rm (\ref{F_31_1}), (\ref{F_31_3}), (\ref{F_31_5})} with {\rm (\ref{F_31_2}), (\ref{F_31_4}), (\ref{F_31_6}),} it suffices
to prove the boundedness of the integral funnel~$ \mathcal{W}(t_0, T, K_0) $ for any $ t_0 \in [0, T) $ and any bounded
set~$ K_0 \subseteq G $ {\rm (}see Remark~{\rm \ref{Rem_15_0}} for the definition of $ \mathcal{W}(t_0, T, K_0) ${\rm )}.
Then the projection of this funnel on the state space is a bounded subset of $ G $ {\rm (}according to the item~{\rm 4} of
Assumption~{\rm \ref{Ass_15})}. Let us verify the sought-for property. From the conditions~{\rm (\ref{F_31_0})} and
{\rm (\ref{F_31_2}),} we derive the existence of a constant~$ C_4 > 0 $ such that
\begin{equation}
\| f(t, x', u') \| \: \leqslant \: C_4 \, (1 + \| x' \|) \quad \forall \: (t, x', u') \: \in \: [0, T] \times \bar{G} \times U
\label{F_31_9}
\end{equation}
and{\rm ,} consequently{\rm ,}
$$
\left| \left< x', \, f(t, x', u') \right> \right| \: \leqslant \: C_4 \, \| x' \| \, (1 + \| x' \|) \quad
\forall \: (t, x', u') \: \in \: [0, T] \times \bar{G} \times U.
$$
Hence{\rm ,} one can choose a constant~$ C_5 > 0 $ satisfying
\begin{equation}
\left| \left< x', \, f(t, x', u') \right> \right| \: \leqslant \: C_5 \, \left( 1 + \| x' \|^2 \right) \quad
\forall \: (t, x', u') \: \in \: [0, T] \times \bar{G} \times U.  \label{F_31_8}
\end{equation}
Now take an arbitrary $ \: (t_0, x_0) \, \in \, [0, T) \times G \: $ with an arbitrary
$ \, u(\cdot) \in \mathcal{U}_{t_0, \, T} \, $ and consider a solution $ \: x(\cdot) \, = \, x(\cdot; \, t_0, x_0, u(\cdot)) \: $
of {\rm (\ref{F_28})} defined on a subinterval $ I \subseteq [t_0, T] $ {\rm (}$ t_0 \in I ${\rm )}. By virtue of
{\rm (\ref{F_31_8}),} we have
$$
\frac{d}{dt} \left( 1 \, + \, \| x(t) \|^2 \right) \:\: = \:\: 2 \, \left< x(t), \, f(t, x(t), u(t)) \right> \:\: \leqslant \:\:
2 \, C_5 \, \left( 1 \, + \, \| x(t) \|^2 \right)
$$
almost everywhere on $ [t_0, T] $. Therefore{\rm ,}
\begin{equation}
\begin{aligned}
& 1 \, + \, \| x(t) \|^2 \:\: \leqslant \:\: \left( 1 \, + \, \| x_0 \|^2 \right) \: e^{2 \, C_5 \, (t - t_0)}, \\
& \| x(t) \| \:\: \leqslant \:\: \sqrt{1 \, + \, \| x_0 \|^2} \: e^{C_5 \, (t - t_0)} \:\: \leqslant \:\:
\sqrt{1 \, + \, \| x_0 \|^2} \: e^{C_5 \, T} \\
& \forall t \in [t_0, T].
\end{aligned}  \label{F_31_10}
\end{equation}
This yields the sought-for statement. One can also see that{\rm ,} in case of an unbounded $ U ${\rm ,} {\rm (\ref{F_31_0})} and
{\rm (\ref{F_31_1})} imply {\rm (\ref{F_31_9})} with some constant~$ C_4 > 0 ${\rm ,} and the same subsequent reasonings again
lead to {\rm (\ref{F_31_10})}. Thus{\rm ,} Remark~{\rm \ref{Rem_15_1}} can be simplified as follows{\rm :} the last item in
Assumption~{\rm \ref{Ass_15}} is a corollary to the previous items either if $ U $ is bounded or if $ \eta $ is bounded below or
if $ \eta $ does not depend on $ u $. \qed
\end{remark}

\begin{theorem}{\rm \cite[\S 2.5.1]{YongZhou1999}}  \label{Thm_16}
Let Assumption~{\rm \ref{Ass_15}} hold with a fixed time horizon~$ T \in (0, +\infty) $. Then{\rm ,} for any fixed initial
position $ \: (t_0, x_0) \, \in \, [0, T) \times G, \: $ there exists an optimal control in
the problem~{\rm (\ref{F_28}), (\ref{F_29})}.
\end{theorem}

The following theorem is Pontryagin's principle.

\begin{theorem}{\rm \cite[\S 3.2]{YongZhou1999}}  \label{Thm_18}
Let Assumptions~{\rm \ref{Ass_15}, \ref{Ass_17}} hold with a fixed time horizon~$ T \in (0, +\infty) ${\rm ,} and let
$ (x^*(\cdot), \, u^*(\cdot)) $ be an optimal pair in the problem~{\rm (\ref{F_28}), (\ref{F_29})} for a fixed
initial position $ \: (t_0, x_0) \, \in \, [0, T) \times G $. Denote
\begin{equation}
\begin{aligned}
& H(t, x, u, p) \:\: \stackrel{\mathrm{def}}{=} \:\: \left< p, f(t, x, u) \right> \: + \: \eta(t, x, u), \\
& \mathcal{H}(t, x, p) \:\: \stackrel{\mathrm{def}}{=} \:\: \inf\limits_{u' \, \in \, U} H(t, x, u', p) \\
& \forall \: (t, x, u, p) \: \in \: [0, T] \times G \times U \times \mathbb{R}^n.
\end{aligned}  \label{F_32}
\end{equation}
Then there exists a function $ \: p^* \, \colon \, [t_0, T] \to \mathbb{R}^n \: $ such that $ (x^*(\cdot), p^*(\cdot)) $ is
a solution of the characteristic boundary value problem
\begin{equation}
\left\{ \begin{aligned}
& \dot{x^*}(t) \:\: = \:\: f(t, \, x^*(t), \, u^*(t)) \:\: = \:\: \mathrm{D}_p H (t, \, x^*(t), \, u^*(t), \, p^*(t)), \\
& \dot{p^*}(t) \:\: = \:\: -\mathrm{D}_x H(t, \, x^*(t), \, u^*(t), \, p^*(t)), \\
& t \in [t_0, T], \\
& x^*(t_0) \: = \: x_0, \quad p^*(T) \: = \: \mathrm{D} \sigma \, (x^*(T)),
\end{aligned} \right.  \label{F_33}
\end{equation}
and the condition
\begin{equation}
\begin{aligned}
H(t, \, x^*(t), \, u^*(t), \, p^*(t)) \:\: = \:\: & \min_{u \, \in \, U} \: H(t, \, x^*(t), \, u, \, p^*(t)) \\
= \:\: & \mathcal{H}(t, \, x^*(t), \, p^*(t))
\end{aligned}  \label{F_34}
\end{equation}
holds for almost every $ t \in [t_0, T] $.
\end{theorem}

\begin{remark}  \label{Rem_19}  \rm
For any $ \: (t, x) \, \in \, [0, T] \times G, \: $ the function
\begin{equation}
\mathbb{R}^n \: \ni \: p \:\: \longmapsto \:\: \mathcal{H}(t, x, p)  \label{F_35}
\end{equation}
is concave{\rm ,} since this is the infimum of the linear function $ H(t, x, u, \cdot) $ over $ u \in U $
{\rm (}see~{\rm (\ref{F_32}))}. \qed
\end{remark}

Introduce the set of minimizers
\begin{equation}
U^*(t, x, p) \:\: \stackrel{\mathrm{def}}{=} \:\: \mathrm{Arg} \min_{u \, \in \, U} \: H(t, x, u, p) \quad
\forall \: (t, x, p) \: \in \: [0, T] \times G \times \mathbb{R}^n  \label{F_36}
\end{equation}
(it is either empty or convex if $ H $ is convex with respect to $ u $).

In line with Remark~\ref{Rem_12_0}, it is reasonable to modify Theorem~\ref{Thm_18} in order to parametrize
characteristic fields with respect to the initial adjoint vector.

\begin{theorem}  \label{Thm_20}
Let Assumptions~{\rm \ref{Ass_15}, \ref{Ass_17}} hold with a fixed $ T \in (0, +\infty) $. For any
$ \: (t_0, x_0) \, \in \, [0, T) \times G, \: $ the value function~{\rm (\ref{F_30})} can be represented as
the minimum of
$$
\sigma(x^*(T)) \: + \: \int\limits_{t_0}^T \eta(t, \, x^*(t), \, u^*(t)) \, dt
$$
over the solutions of the characteristic Cauchy problems
\begin{equation}
\left\{ \begin{aligned}
& \dot{x^*}(t) \:\: = \:\: f(t, \, x^*(t), \, u^*(t)), \\
& \dot{p^*}(t) \:\: = \:\: -\mathrm{D}_x H(t, \, x^*(t), \, u^*(t), \, p^*(t)), \\
& u^*(t) \:\: \in \:\: U^*(t, \, x^*(t), \, p^*(t)), \\
& t \in [t_0, T], \\
& x^*(t_0) \: = \: x_0, \quad p^*(t_0) \: = \: p_0,
\end{aligned} \right.  \label{F_37}
\end{equation}
for all possible values $ \: p^*(t_0) = p_0 \in \mathbb{R}^n $. If{\rm ,} moreover{\rm ,} $ \eta \equiv 0 $
{\rm (}Mayer form of the cost functional~{\rm (\ref{F_29}))} and {\rm (\ref{F_36})} satisfies
\begin{equation}
\begin{aligned}
& U^*(t, x, p) \: = \: U^*(t, \, x, \, \alpha \, p) \\
& \forall \alpha > 0 \quad \forall \: (t, x, p) \: \in \: [0, T] \times G \times \mathbb{R}^n,
\end{aligned}  \label{F_37_1}
\end{equation}
then it is enough to consider a bounded set of parameter values{\rm ,} i.\,e.{\rm ,} for any
$ \: (t_0, x_0) \: \in \: [0, T) \times G, \: $ the value function~{\rm (\ref{F_30})} is the minimum of
$ \sigma(x^*(T)) $ over the solutions of the Cauchy problems~{\rm (\ref{F_37})} for all
\begin{equation}
p_0 \:\: \in \:\: \left\{ p \in \mathbb{R}^n \: \colon \: \| p \| = 1 \:\: \mathrm{or} \:\: p = 0 \right\}.  \label{F_37_0}
\end{equation}
\end{theorem}

\begin{proof}
The first statement directly follows from Theorems~\ref{Thm_16},~\ref{Thm_18} and the fact that, compared to the boundary
value problems~(\ref{F_33}),~(\ref{F_34}), the Cauchy problems~(\ref{F_37}) generate a wider characteristic field
(due to the absence of the terminal condition on the adjoint variable).

Under the conditions $ \eta \equiv 0 $ and (\ref{F_37_1}), the right-hand side of the adjoint system is positive homogeneous of
degree~$ 1 $ with respect to the adjoint variable, and the state components of the characteristic curves do not change after
multiplying $ p_0 $ by any positive number. This leads to the second part of the theorem.
\end{proof}

\begin{remark}  \label{Rem_21}  \rm
Similar properties can be obtained if $ \inf $ is replaced with $ \sup $ in {\rm (\ref{F_29}) (}maximization problem{\rm )}.
Then $ \sup $ appears instead of $ \inf $ in the Hamiltonian~{\rm (\ref{F_32}),} the reduction~{\rm (\ref{F_35})} becomes
convex{\rm ,} $ \mathrm{Arg} \min $ is replaced with $ \mathrm{Arg} \max $ in {\rm (\ref{F_36}),} and $ V $ is determined
through maximization {\rm (}rather than minimization{\rm )} in Theorem~{\rm \ref{Thm_20}}. \qed
\end{remark}

\begin{remark}  \label{Rem_22}  \rm
Let Assumption~{\rm \ref{Ass_15}} hold with a fixed $ T \in (0, +\infty) $. Furthermore{\rm ,} suppose that
the functions~{\rm (\ref{F_31})} are uniformly continuous on $ \, [0, T] \times \bar{G} \times U \, $ if $ U $ is unbounded{\rm ,}
and recall Remark~{\ref{Rem_15_1}} in case of a bounded~$ U $. Then{\rm ,} similarly to {\rm \cite[\S 4.2]{YongZhou1999},} one can
establish that the value function~{\rm (\ref{F_30})} is the unique viscosity solution of the Cauchy problem
\begin{equation}
\left\{ \begin{aligned}
& -\frac{\partial V(t, x)}{\partial t} \: - \: \mathcal{H}(t, \, x, \, \mathrm{D}_x V(t, x)) \:\: = \:\: 0, \quad
(t, x) \: \in \: (0, T) \times G, \\
& V(T, x) \: = \: \sigma(x), \quad x \in G,
\end{aligned} \right.  \label{F_39}
\end{equation}
or{\rm ,} equivalently{\rm ,} the unique minimax solution of
\begin{equation}
\left\{ \begin{aligned}
& \frac{\partial V(t, x)}{\partial t} \: + \: \mathcal{H}(t, \, x, \, \mathrm{D}_x V(t, x)) \:\: = \:\: 0, \quad
(t, x) \: \in \: (0, T) \times G, \\
& V(T, x) \: = \: \sigma(x), \quad x \in G.
\end{aligned} \right.  \label{F_38}
\end{equation}
\qed
\end{remark}

\begin{remark}  \label{Rem_23}  \rm
If{\rm ,} by using Pontryagin's principle {\rm (}Theorem~{\rm \ref{Thm_18}),} one can reasonably exclude singular
regimes from consideration{\rm ,} so that the nonuniqueness in the choice of control values occurs only at isolated time
instants{\rm ,} then each of the considered Cauchy problems~{\rm (\ref{F_37})} has a unique solution. Due to the second
part of Theorem~{\rm \ref{Thm_20},} the bounded set~{\rm (\ref{F_37_0})} of initial adjoint vectors is enough for
determining the value function if one has a Mayer cost functional and the condition~{\rm (\ref{F_37_1})} holds. This is
important from the computational point of view. \qed
\end{remark}

\begin{remark}  \label{Rem_24}  \rm
If one cannot guarantee the absence of singular regimes{\rm ,} then the multi-valued extremal control map may yield
more than one solution of a particular Cauchy problem~{\rm (\ref{F_37}).} This is rather difficult to handle in
a numerical algorithm. Besides{\rm ,} if{\rm ,} for example{\rm ,} the theoretical optimal control synthesis contains
a universal singular surface {\rm \cite[\S 3.5.1]{Melikyan1998} (}transversally entered by forward-time characteristics
from both sides{\rm ,} so that $ u = u' $ on one side and $ u = u'' \neq u' $ on the other{\rm ),} then computations
may lead to excessive bang-bang switchings {\rm (}from $ u = u' $ to $ u = u'' $ and vice versa{\rm )} around
the singular surface instead of just following the continuous singular control regime. This can essentially decrease
the numerical accuracy of integrating the characteristic system. In such situations{\rm ,} it is reasonable to consider
a smooth uniform approximation of the Hamiltonian~$ \mathcal{H} $ {\rm (}and then to apply Theorem~{\rm \ref{Thm_13})}
if the latter is not smooth{\rm ,} so that the choice of control values in {\rm (\ref{F_37})} becomes unique. A general
theoretical result on the stability of the value function with respect to problem data is given{\rm ,} for instance{\rm ,}
in {\rm \cite[\S 4.4.1]{YongZhou1999}}. Note that the required a priori estimates may directly follow from standard smoothing
techniques{\rm ,} such as adding a small positive-definite control-dependent quadratic form to $ H $ in the minimization
problem with a compact~$ U $. However{\rm ,} suitable smooth approximations of Hamiltonians often lead to the appearance of
an integral term in the cost functional{\rm ,} and then the second part of Theorem~{\rm \ref{Thm_20}} is not
applicable{\rm ,} i.\,e.{\rm ,} the finite-dimensional optimization with respect to $ p_0 $ has to be performed over
the whole space~$ \mathbb{R}^n $. Moreover{\rm ,} a standard transformation of a Bolza functional to Mayer form {\rm (}by
introducing a new state variable{\rm )} may violate the uniqueness in the choice of control values in {\rm (\ref{F_37})}. \qed
\end{remark}

\begin{remark}  \label{Rem_42}  \rm
Using the method of characteristics according to Theorem~\ref{Thm_20} (if possible) has a number of key advantages over
solving Cauchy problems for Hamilton--Jacobi--Bellman equations by well-known grid-based approaches, such as finite-difference
schemes
\cite{FlemingSoner2006,CrandallLions1984,OsherShu1991,JiangPeng2000,ZhangShu2003,BokanForcadelZidani2010,
BokanCristianiZidani2010,ROCHJ2017},
semi-Lagrangian schemes
\cite{BardiCapuzzoDolcetta2008,FalconeFerretti1998,Falcone2006,CristianiFalcone2007,ROCHJ2017},
level set methods
\cite{OsherSethian1988,Osher1993,Sethian1999,OsherFedkiw2003,MitchellBayenTomlin2001,MitchellBayenTomlin2005,
MitchellLevelSetToolbox2012},
etc. First, the method of characteristics allows the value function to be computed separately at different initial positions
(contrary to the generic nature of grid methods). Thereby, the curse of dimensionality can be mitigated, and parallel
computations can significantly increase the numerical efficiency, although constructing global (or semi-global) value function
approximations often suffers from the curse of complexity (sparse grid techniques \cite{KangWilcox2017} may help to overcome
the latter). Furthermore, the practical dilemma of choosing a suitable bounded computational domain in the state space often
arises when implementing grid methods. In fact, if one can analytically verify the existence of a bounded strongly invariant
domain in the state space, then it can be used in semi-Lagrangian iteration procedures, but the actual convergence of the latter
strongly depends on how the assumed initial approximation of the value function is taken, which is a heuristic choice.
Finite-difference schemes do not need initial approximations of solutions but cannot take possible strong invariance into account,
i.\,e., a sufficiently large computational domain has to be chosen in order to reduce boundary cutoff errors in a relevant
subdomain, and there is also lack of general recommendations for that. These difficulties do not appear when different initial
positions are treated separately by the method of characteristics. Next, the optimal feedback control strategy at any isolated
position can be obtained directly from integrated optimal characteristics, and one does not need to compute partial derivatives of
the value function, which may be an unstable procedure. However, despite the indicated strong points of the presented characteristic
approach, it still has a limited range of practical applicability, as follows from Remarks~\ref{Rem_14}, \ref{Rem_14_0}, \ref{Rem_23},
\ref{Rem_24} and the aforementioned curse of complexity. Besides, its extensions to zero-sum two-player differential games have been
developed only for sufficiently narrow classes of control systems (see Section~4 and Appendix).
\qed
\end{remark}

\begin{example}  \label{Exa_41}  \rm
Consider the Cauchy problem for an Eikonal equation
\begin{equation}
\left\{ \begin{aligned}
& \frac{\partial V(t, x)}{\partial t} \: + \: c(x) \, \left\| \mathrm{D}_x V(t, x) \right\| \:\: = \:\: 0, \quad
(t, x) \: \in \: (0, T) \times \mathbb{R}^n, \\
& V(T, x) \: = \: \sigma(x), \quad x \in \mathbb{R}^n,
\end{aligned} \right.  \label{F_73}
\end{equation}
in which $ \: c \colon \mathbb{R}^n \to \mathbb{R} \: $ and $ \: \sigma \colon \mathbb{R}^n \to \mathbb{R} \: $ are twice
continuously differentiable functions{\rm ,} $ c $ is Lipschitz continuous{\rm ,} and one of the following two conditions holds{\rm :}
\begin{equation}
c(x) \, < \, 0 \quad \forall x \in \mathbb{R}^n,  \label{F_74}
\end{equation}
\begin{equation}
c(x) \, > \, 0 \quad \forall x \in \mathbb{R}^n.  \label{F_75}
\end{equation}
For any fixed $ \: (t_0, x_0) \, \in \, [0, T) \times \mathbb{R}^n, \: $ the viscosity solution~$ V(t_0, x_0) $ of (\ref{F_73}) is
the value function in the control problem
\begin{equation}
\left\{ \begin{aligned}
& \dot{x}(t) \: = \: c(x(t)) \, u(t), \\
& x(t_0) \, = \, x_0, \\
& u(t) \: \in \: U \: = \: \left\{ v \in \mathbb{R}^n \: \colon \: \| v \| \leqslant 1 \right\}, \\
& t \in [t_0, T],
\end{aligned} \right.  \label{F_76}
\end{equation}
with the criterion
\begin{equation}
\sigma(x(T)) \: \longrightarrow \: \min  \label{F_77}
\end{equation}
in the case~(\ref{F_74}) and
\begin{equation}
\sigma(x(T)) \: \longrightarrow \: \max  \label{F_78}
\end{equation}
in the case~(\ref{F_75}). Then Theorem~\ref{Thm_20} and Remarks~\ref{Rem_21},~\ref{Rem_22} can be applied. The Cauchy
problems~(\ref{F_37}) become
\begin{equation}
\left\{ \begin{aligned}
& \dot{x^*}(t) \:\: = \:\: c(x^*(t)) \: u^*(t), \\
& \dot{p^*}(t) \:\: = \:\: -\mathrm{D} c \, (x^*(t)) \: \left< p^*(t), \, u^*(t) \right> \:\: = \:\:
-\mathrm{D} c \, (x^*(t)) \: \| p^*(t) \|, \\
& u^*(t) \: \in \: U^*(p^*(t)), \\
& t \in [t_0, T], \\
& x^*(t_0) \: = \: x_0, \quad p^*(t_0) \: = \: p_0,
\end{aligned} \right.  \label{F_79}
\end{equation}
where
\begin{equation}
U^*(p) \:\: = \:\: \begin{cases}
\left\{ \dfrac{1}{\| p \|} \: p \right\}, & p \neq 0, \\
U, & p = 0,
\end{cases}  \label{F_80}
\end{equation}
for all $ p \in \mathbb{R}^n $, and the set of possible values of $ p_0 $ can be taken as (\ref{F_37_0}). Note that
$ p^* \equiv 0 $ everywhere on $ [t_0, T] $ if $ p^*(t') = 0 $ at some instant~$ t' \in [t_0, T] $. Therefore,
$ p^*(t) \neq 0 $ and $ U^*(p^*(t)) $ is singleton for all $ t \in [t_0, T] $ if $ p_0 \neq 0 $, but no information
concerning $ u^*(\cdot) $ is available when $ p_0 = 0 $. Recall also the terminal condition in Pontryagin's principle
(Theorem~\ref{Thm_18}), which should be formulated as $ \: p^*(T) \, \uparrow\uparrow \, \mathrm{D} \sigma \, (x^*(T)), \: $ i.\,e., as
\begin{equation}
p^*(T) \:\: \in \:\: \{ \alpha \, \mathrm{D} \sigma \, (x^*(T)) \: \colon \: \alpha > 0 \},  \label{F_80_0}
\end{equation}
if $ p_0 $ is normalized. Then it is easy to conclude that the value~$ p_0 = 0 $ can be excluded from consideration if
\begin{equation}
\mathrm{D} \sigma(x) \neq 0 \quad \forall x \in \mathbb{R}^n  \label{F_81}
\end{equation}
or
\begin{equation}
\{ x \in \mathbb{R}^n \: \colon \: \mathrm{D} \sigma(x) \, = \, 0 \} \:\: \subseteq \:\: \begin{cases}
\mathrm{Arg} \, \max\limits_{x \in \mathbb{R}^n} \, \sigma(x) & \mbox{in the case~(\ref{F_74})}, \\
\mathrm{Arg} \, \min\limits_{x \in \mathbb{R}^n} \, \sigma(x) & \mbox{in the case~(\ref{F_75})}.
\end{cases}  \label{F_82}
\end{equation}

For the problems~(\ref{F_76}),~(\ref{F_77}) and (\ref{F_76}),~(\ref{F_78}) under the corresponding basic assumptions, it is
in fact possible to modify the statement of Theorem~\ref{Thm_20} in order to exclude the nonuniqueness in the choice of extremal
control values without imposing the particular conditions~(\ref{F_81}),~(\ref{F_82}). Let us demonstrate this for
the problem~(\ref{F_76}),~(\ref{F_77}) in the case~(\ref{F_74}). One can use similar reasonings for
the problem~(\ref{F_76}),~(\ref{F_78}) in the case~(\ref{F_75}).

Let $ x' \in \mathbb{R}^n $ be a zero point of $ \mathrm{D} \sigma $, i.\,e., $ \: \mathrm{D} \sigma(x') = 0 $.
Theorem~\ref{Thm_20} does not allow the characterization of the extremal trajectories fulfilling $ \: x^*(T) = x' $,
$ \: p^*(T) = \mathrm{D} \sigma(x') = 0 \: $ and $ p^* \equiv 0 $. Let such a trajectory emanating from $ x^*(t_0) = x_0 $
exist. Then the minimum time problem
\begin{equation}
\left\{ \begin{aligned}
& \dot{x}(t) \: = \: c(x(t)) \, u(t), \\
& x(t_0) \, = \, x_0, \quad x(T') \, = \, x', \\
& u(t) \, \in \, U, \\
& t \in [t_0, T'], \\
& T' \geqslant t_0 \:\: \mbox{is free}, \\
& T' \: \longrightarrow \: \min \, ,
\end{aligned} \right.  \label{F_83}
\end{equation}
admits a solution for which $ \: T' = T'_{\min} \leqslant T $. If we extend the related control function to the whole
time interval~$ [t_0, T] $ by taking it as zero on $ (T'_{\min}, T] $, then the resulting process fulfills $ x^*(t) = x' $
for all $ t \in [T'_{\min}, T] $ and thereby gives the cost~$ \sigma(x') $. This will be an optimal process for the original
problem~(\ref{F_76}),~(\ref{F_77}) if $ \: V(t_0, x_0) = \sigma(x') $. Furthermore, Pontryagin's principle for the minimum time
problem~(\ref{F_83}) (see \cite{Pontryagin1964}) leads to the same system of characteristic equations as in (\ref{F_33}), but
in the absence of the terminal condition on $ p^*(T') $ and under the requirement that $ p^*(t) \neq 0 $ for all
$ t \in [t_0, T'] $. Let us also emphasize that these reasonings are applicable to any zero point of $ \mathrm{D} \sigma $
which can be reached at $ t = T $ by an extremal state trajectory emanating from $ x^*(t_0) = x_0 $.

Thus, one arrives at the following statement: {\it for any $ \: (t_0, x_0) \, \in \, [0, T) \times \mathbb{R}^n, \: $
the minimax solution~$ V(t_0, x_0) $ of the problem~{\rm (\ref{F_73})} or{\rm ,} equivalently{\rm ,} the value function in
the problem~{\rm (\ref{F_76}),~(\ref{F_77}) (}under the formulated basic conditions on the functions~$ c, \sigma ${\rm ,}
including {\rm (\ref{F_74}))} is the minimum of the quantity~$ \sigma(x^*(T'')) $ over the solutions of the Cauchy
problems~{\rm (\ref{F_79}),~(\ref{F_80})} for all possible values $ T'' \in [t_0, T] $ and
\begin{equation}
p_0 \:\: \in \:\: \left\{ p \in \mathbb{R}^n \: \colon \: \| p \| = 1 \right\}  \label{F_84}
\end{equation}
{\rm (}the value~$ p_0 = 0 $ is excluded here{\rm )}. If{\rm ,} moreover{\rm ,}
$$
\{ x \in \mathbb{R}^n \: \colon \: \mathrm{D} \sigma(x) \, = \, 0 \} \:\: = \:\:
\mathrm{Arg} \, \min\limits_{x \in \mathbb{R}^n} \, \sigma(x) \:\: = \:\: \{ x' \}
$$
for some $ x' \in \mathbb{R}^n ${\rm ,} then{\rm ,} in the latter characterization of $ V(t_0, x_0) ${\rm ,} it is enough
to specify $ T'' $ as the minimum over all $ T' \in [t_0, T] $ at which $ x^*(T') = x' $ if such $ T' $ exist{\rm ,} and
as $ T $ otherwise}. \qed
\end{example}

\begin{example}  \label{Exa_41_2}  \rm
Consider the Cauchy problem
\begin{equation}
\left\{ \begin{aligned}
& \frac{\partial V(t, x)}{\partial t} \: + \: c(x) \, \left\| \mathrm{D}_x V(t, x) \right\| \: + \: \eta(x) \:\: = \:\: 0, \quad
(t, x) \: \in \: (0, T) \times \mathbb{R}^n, \\
& V(T, x) \: = \: \sigma(x), \quad x \in \mathbb{R}^n,
\end{aligned} \right.  \label{F_73_2}
\end{equation}
in which {\it $ \: c \colon \mathbb{R}^n \to \mathbb{R} ${\rm ,} $ \: \eta \colon \mathbb{R}^n \to \mathbb{R} \: $ and
$ \: \sigma \colon \mathbb{R}^n \to \mathbb{R} \: $ are twice continuously differentiable functions{\rm ,}
$ c $ is Lipschitz continuous{\rm ,} and one of the conditions~{\rm (\ref{F_74}),~(\ref{F_75})} holds.}

The special case of (\ref{F_73_2}) with $ \eta \equiv 0 $ was studied in Example~\ref{Exa_41}.

For any fixed $ \: (t_0, x_0) \, \in \, [0, T) \times \mathbb{R}^n, \: $ the viscosity solution~$ V(t_0, x_0) $ of (\ref{F_73_2}) is
the value function in the control problem~(\ref{F_76}) with the criterion
\begin{equation}
\sigma(x(T)) \: + \: \int\limits_{t_0}^T \eta(x(t)) \, dt \:\: \longrightarrow \:\: \min  \label{F_77_2}
\end{equation}
in the case~(\ref{F_74}) and
\begin{equation}
\sigma(x(T)) \: + \: \int\limits_{t_0}^T \eta(x(t)) \, dt \:\: \longrightarrow \:\: \max  \label{F_78_2}
\end{equation}
in the case~(\ref{F_75}). As per Example~\ref{Exa_41}, Theorem~\ref{Thm_20} and Remarks~\ref{Rem_21},~\ref{Rem_22} can be applied. The Cauchy
problems~(\ref{F_37}) become
\begin{equation}
\left\{ \begin{aligned}
& \dot{x^*}(t) \:\: = \:\: c(x^*(t)) \: u^*(t), \\
& \dot{p^*}(t) \:\: = \:\: -\mathrm{D} c \, (x^*(t)) \: \| p^*(t) \| \: - \: \mathrm{D} \eta \, (x^*(t)), \\
& u^*(t) \: \in \: U^*(p^*(t)), \\
& t \in [t_0, T], \\
& x^*(t_0) \: = \: x_0, \quad p^*(t_0) \: = \: p_0,
\end{aligned} \right.  \label{F_79_2}
\end{equation}
where $ U^*(p) $ is defined by (\ref{F_80}) for all $ p \in \mathbb{R}^n $, and $ p_0 $ takes values in the whole
space~$ \mathbb{R}^n $.

For the Bolza functional in (\ref{F_77_2}) and (\ref{F_78_2}), the set of possible values of $ p_0 $ cannot be reduced to
the bounded set~(\ref{F_37_0}). However, by introducing the new scalar state variable~$ \tilde{x} $ such that
\begin{equation}
\dot{\tilde{x}}(t) \: = \: \eta(x(t)), \quad t \in [t_0, T], \quad \tilde{x}(t_0) \: = \: 0,  \label{F_76_2}
\end{equation}
one arrives at the Mayer functional $ \: \sigma(x(T)) \, + \, \tilde{x}(T) $. Then the characteristic Cauchy problems
take the form
\begin{equation}
\left\{ \begin{aligned}
& \dot{x^*}(t) \:\: = \:\: c(x^*(t)) \: u^*(t), \\
& \dot{\tilde{x}}^*(t) \:\: = \:\: \eta(x^*(t)), \\
& \dot{p^*}(t) \:\: = \:\: -\mathrm{D} c \, (x^*(t)) \: \| p^*(t) \| \: - \: \tilde{p}^*(t) \: \mathrm{D} \eta \, (x^*(t)), \\
& \dot{\tilde{p}}^* \, \equiv \, 0 \:\:\: \Longrightarrow \:\:\: \tilde{p}^* \, \equiv \, \mathrm{const}, \\
& u^*(t) \: \in \: U^*(p^*(t)), \\
& t \in [t_0, T], \\
& x^*(t_0) \: = \: x_0, \quad \tilde{x}^*(t_0) \: = \: 0, \quad p^*(t_0) \: = \: p_0,
\end{aligned} \right.  \label{F_79_3}
\end{equation}
where
\begin{equation}
\left( p_0, \tilde{p}^* \right) \:\: \in \:\: \left\{ \left( p, \tilde{p} \right) \, \in \,
\mathbb{R}^n \times \mathbb{R} \: \colon \: \left\| \left( p, \tilde{p} \right) \right\| \, \in \, \{ 0, 1 \} \right\}
\label{F_37_2}
\end{equation}
according to the second part of Theorem~\ref{Thm_20}. Since the coefficient of $ \: \mathrm{D} \eta \, (x^*(t)) \: $ in
the adjoint system of the original Cauchy problems~(\ref{F_79_2}) does not vanish (it equals~$ -1 $), the case
$ \tilde{p}^* = 0 $ can be excluded when considering the transformed Cauchy problems~(\ref{F_79_3}), i.\,e., (\ref{F_37_2}) is
reduced to
\begin{equation}
\left( p_0, \tilde{p}^* \right) \:\: \in \:\: \left\{ \left( p, \tilde{p} \right) \, \in \,
\mathbb{R}^n \times \mathbb{R} \: \colon \: \left\| \left( p, \tilde{p} \right) \right\| \, = \, 1, \:\:
\tilde{p} \neq 0 \right\}.
\label{F_37_3}
\end{equation}

Hence, {\it the value function $ V(t_0, x_0) $ at any fixed position $ \: (t_0, x_0) \, \in \, [0, T) \times \mathbb{R}^n \: $
can be obtained by optimizing the functional}
\begin{equation}
\sigma(x^*(T)) \: + \: \tilde{x}^*(T) \:\: = \:\: \sigma(x^*(T)) \: + \:
\int\limits_{t_0}^T \eta(x^*(t)) \, dt  \label{F_79_4}
\end{equation}
{\it {\rm (}minimizing in the case~{\rm (\ref{F_74}),~(\ref{F_77_2})} and maximizing in
the case~{\rm (\ref{F_75}),~(\ref{F_78_2}))} over the solutions of the Cauchy problems~{\rm (\ref{F_79_3}),~(\ref{F_80})}
for all parameters~{\rm (\ref{F_37_3})}}.

The extremal control set~$ U^*(p) $ is not singleton only when $ p = 0 $. If $ \dot{p^*}(t) \neq 0 $ for all instants
$ t \in [t_0, T) $ at which $ p^*(t) = 0 $, then every zero of $ p^*(\cdot) $ on $ [t_0, T) $ is isolated, and any particular
choice of extremal control values at such isolated instants does not affect the corresponding characteristic curve.
A sufficient condition for that is
\begin{equation}
\mathrm{D} \eta(x) \: \neq \: 0 \quad \forall x \in \mathbb{R}^n.  \label{F_81_2}
\end{equation}
Indeed, the system~(\ref{F_79_3}) and condition~$ \tilde{p}^* \neq 0 $ yield that the expression
\begin{equation}
\dot{p^*}(t) \left|_{p^*(t) \, = \, 0} \right. \:\: = \:\: -\tilde{p}^* \: \mathrm{D} \eta \, (x^*(t))  \label{F_79_5}
\end{equation}
is nonzero for all $ t \in [t_0, T] $ if (\ref{F_81_2}) holds.

Finally, let us relax the condition~(\ref{F_81_2}) and modify the value function representations so as to avoid
the nonuniqueness in the choice of extremal control values. Instead of (\ref{F_81_2}), {\it assume that}
\begin{equation}
\{ x \in \mathbb{R}^n \: \colon \: \mathrm{D} \eta(x) \, = \, 0 \} \:\: \subseteq \:\:
\mathrm{Arg} \, \min\limits_{x \in \mathbb{R}^n} \, \eta(x) \:\: \cap \:\:
\mathrm{Arg} \, \min\limits_{x \in \mathbb{R}^n} \, \sigma(x)  \label{F_82_2}
\end{equation}
{\it in the case~{\rm (\ref{F_74})} and that}
\begin{equation}
\{ x \in \mathbb{R}^n \: \colon \: \mathrm{D} \eta(x) \, = \, 0 \} \:\: \subseteq \:\:
\mathrm{Arg} \, \max\limits_{x \in \mathbb{R}^n} \, \eta(x) \:\: \cap \:\:
\mathrm{Arg} \, \max\limits_{x \in \mathbb{R}^n} \, \sigma(x)  \label{F_82_3}
\end{equation}
{\it in the case~{\rm (\ref{F_75})}}. Then the following implication holds for the problem~(\ref{F_76}),~(\ref{F_77_2}) in
the case~(\ref{F_74}), as well as for the problem~(\ref{F_76}),~(\ref{F_78_2}) in the case~(\ref{F_75}): if an optimal
characteristic curve satisfies $ x^*(t') = x' $ for some $ t' \in [t_0, T) $ and $ \, \mathrm{D} \eta(x') = 0, \, $ it will
remain optimal after setting $ u^*(t) = 0 $ for all $ t \in (t', T] $ (which yields $ x^*(t) = x' $ and $ p^*(t) = 0 $ for all
$ t \in [t', T] $ due to Pontryagin's principle). This leads to the sought-for value function representations. Let us provide
the one for the problem~(\ref{F_76}),~(\ref{F_77_2}) in the case~(\ref{F_74}). A similar statement can be given for
the problem~(\ref{F_76}),~(\ref{F_78_2}) in the case~(\ref{F_75}).

{\it For any $ \: (t_0, x_0) \, \in \, [0, T) \times \mathbb{R}^n, \: $ the minimax solution~$ V(t_0, x_0) $ of
the problem~{\rm (\ref{F_73_2})} or{\rm ,} equivalently{\rm ,} the value function in
the problem~{\rm (\ref{F_76}),~(\ref{F_77_2}) (}under the formulated basic conditions on the functions~$ c, \eta, \sigma ${\rm ,}
including {\rm (\ref{F_74}),~(\ref{F_82_2}))} is the minimum of {\rm (\ref{F_79_4})} over such solutions of
{\rm (\ref{F_79_3}),~(\ref{F_80}),~(\ref{F_37_3})} that satisfy the following property{\rm :} if $ x^*(t') = x' $
for some $ t' \in [t_0, T) $ and $ \, \mathrm{D} \eta(x') = 0, \, $ then $ u^*(t) = 0 $ for all $ t \in (t', T] $
{\rm (}and{\rm ,} consequently{\rm ,} $ x^*(t) = x' $ for all $ t \in [t', T] ${\rm )}}.
\qed
\end{example}

\section{Curse-of-dimensionality-free approaches for solving \\ Hamilton--Jacobi--Isaacs equations in zero-sum \\ two-player
differential games: principal issues and \\ some applications}

The aim of this section is to discuss principal issues in overcoming the curse of dimensionality for zero-sum two-player
differential games and to indicate existing curse-of-dimensionality-free approaches for specific classes of systems.

Different classes of admissible control strategies lead to different notions of lower and upper values, saddle points and
Nash equilibrium in zero-sum two-player differential games. Under some standard technical assumptions and the so-called Isaacs
condition, there exists an equilibrium in feedback (closed-loop) strategies, and the corresponding value function is a unique
minimax/viscosity solution of the appropriate Cauchy problem for the Hamilton--Jacobi--Isaacs equation
\cite{KrasovskiiSubbotin1974,KrasovskiiSubbotin1988,SubbotinChentsov1981,Berkovitz1985,Subbotin1995}. For the class of
nonanticipative (Varaiya--Roxin--Elliot--Kalton) strategies, the value function appears to be the same \cite{Yong2015}. However,
the existence of an equilibrium in closed-loop or nonanticipative strategies does not imply the existence of an equilibrium in
open-loop (programmed) strategies, as the classical example of the ``lady in the lake'' game indicates~\cite{BasarOlsder1995}.

As opposed to optimal control problems, Pontryagin's principle for zero-sum two-player differential games gives necessary
conditions only for saddle open-loop strategies, but not for saddle closed-loop ones. The main qualitative difference in the behavior
of characteristics for optimal control problems and differential games lies in the corner conditions for switching surfaces that are
reached on one side and left on the other side \cite{Bernhard1977,Bernhard2014}. While Pontryagin's theorem extends to optimal
control theory the Weierstrass--Erdmann condition stating that the adjoint function is continuous along an extremal trajectory,
differential game theory allows discontinuities there. In general, these singularities cannot be found by a local analysis along
isolated characteristics and require the construction of a complete field of extremals, leading to a global synthesis map.
The related notions of equivocal, envelope and focal manifolds are discussed in
\cite{Bernhard1977,Melikyan1998,MelikyanBernhard2005,Evans2014}. Thus, developing curse-of-dimensionality-free characteristics
approaches for wide classes of Hamilton--Jacobi--Isaacs equations in differential games turns out to be extremely difficult.

Given a fixed finite time horizon~$ T \in (0, +\infty) $, consider the control system of ordinary differential equations
\begin{equation}
\left\{ \begin{aligned}
& \dot{x}(t) \:\: = \:\: f(t, \, x(t), \, u_1(t), \, u_2(t)), \\
& u_i(t) \, \in \, U_i, \quad i = 1,2, \\
& t \in [0, T],
\end{aligned} \right.  \label{F_40_0}
\end{equation}
where $ \: x \, \colon \, [0, T] \to \mathbb{R}^n \: $ is a state function and
$ \: u_i \, \colon \, [0, T] \, \to \, U_i \subseteq \mathbb{R}^{m_i} $, $ i = 1,2, \: $ are measurable control
functions corresponding to two players, labelled $ 1 $ and $ 2 $.

Suppose that the aim of the player~$ 1 $ is to minimize a terminal payoff~$ \sigma(x(T)) $, while the player~$ 2 $ intends
to maximize it. Hence, we arrive at the zero-sum differential game for (\ref{F_40_0}) that can formally be written as
\begin{equation}
\sigma(x(T)) \:\: \longrightarrow \:\: \inf_{u_1(\cdot)} \, \sup_{u_2(\cdot)} \:\: \mbox{or} \:\:
\sup_{u_2(\cdot)} \, \inf_{u_1(\cdot)} \, .  \label{F_41}
\end{equation}

Let the game~(\ref{F_40_0}),~(\ref{F_41}) fulfill standard technical assumptions and the Isaacs condition, so that the closed-loop
game value function~$ V $ exists (its rigorous definition relies on specific mathematical constructions \cite[\S III.11]{Subbotin1995}
and is omitted here for the sake of brevity).

For $ t_0 \in [0, T) $ and $ i = 1,2 $, let $ \mathcal{U}^i_{t_0, \, T} $ be the class of measurable functions defined on
$ [t_0, T] $ with values in $ U_i $. If, moreover, $ x_0 \in \mathbb{R}^n $ and $ \: u_i(\cdot) \in \mathcal{U}^i_{t_0, \, T} $,
$ \, i = 1,2, \: $ then $ \: x(\cdot; \, t_0, x_0, u_1(\cdot), u_2(\cdot))) \: $ denotes the solution of (\ref{F_40_0})
satisfying $ x(t_0) = x_0 $ and corresponding to the open-loop control strategies~$ u_1(\cdot), u_2(\cdot) $. The lower open-loop
game value function (or, in other words, the programmed maximin function) is defined by
\begin{equation}
\begin{aligned}
& V^*(T, x_0) \: \stackrel{\mathrm{def}}{=} \: \sigma(x_0), \\
& V^*(t_0, x_0) \:\: \stackrel{\mathrm{def}}{=} \:\: \sup_{u_2(\cdot) \, \in \, \mathcal{U}^2_{t_0, \, T}} \,
\inf_{u_1(\cdot) \, \in \, \mathcal{U}^1_{t_0, \, T}} \: \sigma(x(T; \, t_0, x_0, u_1(\cdot), u_2(\cdot))) \\
& \forall \: (t_0, x_0) \: \in \: [0, T) \times \mathbb{R}^n.
\end{aligned}  \label{F_41_1}
\end{equation}

If the programmed maximin and closed-loop game value functions are equal to each other in the whole considered domain of
initial positions, the game is called regular.

In \cite[\S IV.5]{SubbotinChentsov1981}, a so-called programmed iteration procedure starting from the programmed maximin
function~$ V^* $ was proposed, and a result on its pointwise convergence to the closed-loop game value function~$ V $ was
established. Since the corresponding constructions are in general rather complicated, it is important to investigate
the problem classes for which the sought-for limit is exactly reached after a small number of iterations. 

The rest of this section is organized as follows. Example~\ref{Exa_26} presents a differential game for which
the closed-loop game value cannot be reached in a finite number of programmed iterations. Theorem~\ref{Thm_27} gives
a regularity criterion for a certain class of differential games, and Example~\ref{Exa_28} shows a trivial
regular game. Theorem~\ref{Thm_38} describes a problem class for which the closed-loop game value is reached exactly
after one step of the programmed iteration procedure, and Examples~\ref{Exa_39},~\ref{Exa_40} indicate two particular
applications (besides, numerical simulation results for Example~\ref{Exa_40} can be seen in Section~5). Another
related class of differential games is introduced in Appendix. Note that the mentioned problem classes were earlier
described in \cite{SubbotinChentsov1981}. However, their review still seems to be useful, because most of the corresponding
general results and examples were not presented in English-language scientific literature according to our knowledge.
Moreover, these results reduce the computation of the closed-loop game value at any selected position to finite-dimensional
optimization and thereby allow to mitigate the curse of dimensionality (analogously to the characteristics approaches
considered in Sections~2 and 3).

\begin{example}{\rm \cite[\S IV.6]{SubbotinChentsov1981}}  \label{Exa_26}  \rm
First, let us mention a particular problem of the form~(\ref{F_40_0}), (\ref{F_41}), where the programmed iterations converge to
the closed-loop game value but never reach the latter exactly:
$$
\left\{ \begin{aligned}
& \dot{x}(t) \:\: = \:\: u_1(t) \: + \: u_2(t), \quad t \in [0, T], \quad T = 1, \\
& n \, = \, m_1 \, = \, m_2 \, = \, 1, \quad U_1 \, = \, [-1, 1], \quad U_2 \, = \, [-2, 2], \\
& \sigma(x) \: = \: \min\limits_{y \in Y} \, |x - y| \quad \forall x \in \mathbb{R}, \quad
Y \: \stackrel{\mathrm{def}}{=} \: (-\infty, -1] \, \cup \, [1, +\infty).
\end{aligned} \right.
$$
\qed
\end{example}

Let us formulate a game regularity criterion for a linear system
\begin{equation}
\left\{ \begin{aligned}
& \dot{x}(t) \:\: = \:\: A(t) \, x(t) \: + \: B_1(t) \, u_1(t) \: + \: B_2(t) \, u_2(t), \\
& u_i(t) \, \in \, U_i, \quad i = 1,2, \\
& t \in [0, T].
\end{aligned} \right.  \label{F_40}
\end{equation}
The following basic conditions are supposed to hold (together with the Isaacs condition which is obviously satisfied for
(\ref{F_40}), they guarantee the existence of the closed-loop game value function~$ V $).

\begin{assumption}  \label{Ass_25}
The matrix functions $ \: A \, \colon \, [0, T] \to \mathbb{R}^{n \times n} \: $ and
$ \: B_i \, \colon \, [0, T] \to \mathbb{R}^{n \times m_i} ${\rm ,} $ i = 1,2, \: $ are continuous{\rm ,} and the sets
$ \: U_i \subset \mathbb{R}^{m_i} ${\rm ,} $ i = 1,2 ${\rm ,} are compact and convex. Moreover{\rm ,} the terminal game payoff
is determined by a Lipschitz continuous convex function $ \: \sigma \, \colon \, \mathbb{R}^n \to \mathbb{R} $.
\end{assumption}

We use the notations
\begin{equation}
\arraycolsep=1.5pt
\def\arraystretch{2}
\begin{array}{c}
\sigma^*(l) \:\: \stackrel{\mathrm{def}}{=} \:\: \sup\limits_{x \in \mathbb{R}^n} \{ \left< l, x \right> \: - \:
\sigma(x) \} \quad \forall l \in \mathbb{R}^n, \\
L \:\: \stackrel{\mathrm{def}}{=} \:\: \mathrm{dom} \: \sigma^* \:\: = \:\: \left\{ l \in \mathbb{R}^n \: \colon \:
\sigma^*(l) < +\infty \right\}
\end{array}  \label{F_42}
\end{equation}
(the set~$ L $ is bounded due to the Lipschitz continuity of $ \sigma $), and the Cauchy matrix function
$ \: [0, T] \ni t \: \longmapsto \: \Phi(T, t) \, \in \, \mathbb{R}^{n \times n} \: $ is the solution of
\begin{equation}
\left\{ \begin{aligned}
& \frac{d}{dt} \, \Phi(T, t) \: = \: -\Phi(T, t) \, A(t), \\
& \Phi(T, T) \: = \: I_{n \times n}
\end{aligned} \right.  \label{F_43}
\end{equation}
($ I_{n \times n} $ denotes the unit matrix in $ \mathbb{R}^{n \times n} $).

\begin{theorem}{\rm (\cite[\S III.16.1]{Subbotin1995}, \cite[\S III.5]{SubbotinChentsov1981},
\cite[\S 5.2, \S 5.4]{KrasovskiiSubbotin1988})}  \label{Thm_27}
Let Assumption~{\rm \ref{Ass_25}} hold. For all $ \: (t_0, x_0) \, \in \, [0, T] \times \mathbb{R}^n, \: $ the programmed
maximin function for {\rm (\ref{F_40}), (\ref{F_41})} is determined by
\begin{equation}
\begin{aligned}
& V^*(t_0, x_0) \:\: = \:\: \max\limits_{l \, \in \, L} \, \varphi(t_0, x_0, l), \\
& \varphi(t_0, x_0, l) \:\: \stackrel{\mathrm{def}}{=} \:\: \left< l, \, \Phi(T, t_0) \, x_0 \right> \:\: + \:\:
\int\limits_{t_0}^T \varkappa_1(T, t, l) \: dt \:\: + \\
& \qquad\qquad\qquad\qquad\quad \:\:\:\,
+ \:\: \int\limits_{t_0}^T \varkappa_2(T, t, l) \: dt \:\: - \:\: \sigma^*(l), \\
& \varkappa_1(T, t, l) \:\: \stackrel{\mathrm{def}}{=} \:\: \min\limits_{u_1 \, \in \, U_1}
\left< l, \, \Phi(T, t) \, B_1(t) \, u_1 \right>, \\
& \varkappa_2(T, t, l) \:\: \stackrel{\mathrm{def}}{=} \:\: \max\limits_{u_2 \, \in \, U_2}
\left< l, \, \Phi(T, t) \, B_2(t) \, u_2 \right>.
\end{aligned}  \label{F_44}
\end{equation}
A necessary and sufficient condition for its coincidence with the closed-loop game value function~$ V $ on
$ \, [0, T] \times \mathbb{R}^n \, $ {\rm (}i.\,e.{\rm ,} the game regularity criterion{\rm )} is
\begin{equation}
\begin{aligned}
& \max\limits_{u_2 \, \in \, U_2} \, \min\limits_{u_1 \, \in \, U_1} \, \max\limits_{l \, \in \, L^*(t, x)} \,
\{ \left< l, \, \Phi(T, t) \, B_1(t) \, u_1 \right> \: - \: \varkappa_1(T, t, l) \:\: + \\
& \qquad\qquad\qquad\quad \:\:\:\,
+ \: \left< l, \, \Phi(T, t) \, B_2(t) \, u_2 \right> \: - \: \varkappa_2(T, t, l) \} \:\: \leqslant \:\: 0 \\
& \forall \: (t, x) \: \in \: [0, T) \times \mathbb{R}^n,
\end{aligned}  \label{F_45}
\end{equation}
where
\begin{equation}
L^*(t, x) \:\: \stackrel{\mathrm{def}}{=} \:\: \mathrm{Arg} \max_{l \, \in \, L} \: \varphi(t, x, l) \quad
\forall \: (t, x) \: \in \: [0, T] \times \mathbb{R}^n.  \label{F_47}
\end{equation}
In particular{\rm ,} if the reduction
\begin{equation}
L \: \ni \: l \:\: \longmapsto \:\: \varphi(t, x, l)  \label{F_46}
\end{equation}
is concave for any $ \: (t, x) \, \in \, [0, T] \times \mathbb{R}^n, \: $ then $ V^* \equiv V $ on
$ \, [0, T] \times \mathbb{R}^n $. If the set~{\rm (\ref{F_47})} consists of a single element~$ l^*(t, x) $ for any
$ \: (t, x) \, \in \, [0, T] \times \mathbb{R}^n \: $ {\rm (}a sufficient condition for that is the strict concavity of
{\rm (\ref{F_46})} for all such $ (t, x) ${\rm ),} then{\rm ,} almost everywhere on the time interval~$ [0, T] ${\rm ,}
the control functions corresponding to the programmed maximin can be chosen according to the extremal aiming rule
$$
u^*_i(t) \: \in \: U^*_i(t, x(t)), \quad i = 1,2,
$$
with the feedback maps
\begin{equation}
\begin{aligned}
& U^*_i(t, x) \:\: = \:\: \{ u_i \in U_i \: \colon \: \left< l^*(t, x), \: \Phi(T, t) \, B_i(t) \, u_i \right> \:\, = \:\,
\varkappa_i \left(T, \, t, \, l^*(t, x) \right) \}, \\
& i = 1,2, \quad \forall \: (t, x) \: \in \: [0, T] \times \mathbb{R}^n.
\end{aligned}  \label{F_48}
\end{equation}
\end{theorem}

\begin{example}{\rm \cite[\S IV.5]{SubbotinChentsov1981}}  \label{Exa_28}  \rm
In the following case, the game of the form~(\ref{F_40}), (\ref{F_41}) is regular:
$$
\left\{ \begin{aligned}
& \dot{x}(t) \:\: = \:\: u_1(t) \: + \: u_2(t), \quad t \in [0, T], \quad \mbox{$ T > 0 $ is fixed}, \\
& \mbox{$ U_i $ is the closed ball in $ \mathbb{R}^n $ with center at the origin and radius~$ a_i > 0 $}, \:\: i = 1,2, \\
& \sigma(x) \: = \: \| x \| \quad \forall x \in \mathbb{R}^n.
\end{aligned} \right.
$$
By using Theorem~\ref{Thm_27}, the related programmed maximin and closed-loop game value functions can be represented~as
$$
\begin{aligned}
& V^*(t, x) \:\: = \:\: V(t, x) \:\: = \:\: \max \{ \| x \| \: + \: (a_2 - a_1) \, (T - t), \:\: 0 \} \\
& \forall \: (t, x) \: \in \: [0, T] \times \mathbb{R}^n.
\end{aligned}
$$
\qed
\end{example}

Next, let us describe a problem class for which a single programmed iteration allows to obtain the closed-loop game value exactly.
The corresponding result can be formulated for a state-affine system with a more general control input term:
\begin{equation}
\left\{ \begin{aligned}
& \dot{x}(t) \:\: = \:\: A(t) \, x(t) \: + \: \hat{f}(t, \, u_1(t), \, u_2(t)), \\
& u_i(t) \, \in \, U_i, \quad i = 1,2, \\
& t \in [0, T].
\end{aligned} \right.  \label{F_61}
\end{equation}
The system (\ref{F_40}) is the special case of (\ref{F_61}) with
$ \: \hat{f}(t, u_1(t), u_2(t)) \: = \: B_1(t) u_1(t) \, + \, B_2(t) u_2(t) $.

\begin{assumption}  \label{Ass_35}
$ \hat{f} \, \colon \, [0, T] \times U_1 \times U_2 \, \to \, \mathbb{R}^n \: $ is a continuous function.
\end{assumption}

We need additional notations. Take any $ j \in \mathbb{N} $. Denote the origin in $ \mathbb{R}^j $ by $ O_j $ and
the Euclidean norm in $ \mathbb{R}^j $ by $ \| \cdot \|_j $ (previously we denoted this norm by
$ \| \cdot \|_{\mathbb{R}^j} $ or simply $ \| \cdot \| $). The corresponding unit sphere in $ \mathbb{R}^j $ is
\begin{equation}
L_j \:\: \stackrel{\mathrm{def}}{=} \:\: \left\{ l \in \mathbb{R}^j \: \colon \: \| l \|_j = 1 \right\}.  \label{F_49}
\end{equation}
For all $ \: y \, = \, (y_1, y_2, \ldots, y_j) \, \in \, \mathbb{R}^j \: $ and $ \varepsilon \geqslant 0 $, denote also
\begin{equation}
\mathcal{B}_j(y, \varepsilon) \:\: \stackrel{\mathrm{def}}{=} \:\: \left\{ v \in \mathbb{R}^j \: \colon \:
\| v - y \|_j \, \leqslant \, \varepsilon \right\}.  \label{F_50}
\end{equation}
For any nonempty convex compact set $ K \subset \mathbb{R}^j $, denote its support function by $ s(\cdot; \, K) $.
For any $ \: j \in \{1, 2, \ldots, n\} \: $ and $ x \in \mathbb{R}^n $, let $ \{x\}_j \in \mathbb{R}^j $ be the vector of
the first $ j $ coordinates of $ x $. For any set $ X \subseteq \mathbb{R}^n $, denote
$ \: \{X\}_j \: \stackrel{\mathrm{def}}{=} \: \{ \{x\}_j \, \colon \, x \in X \} \: \subseteq \: \mathbb{R}^j $.

Furthermore, suppose a particular form of the terminal payoff.

\begin{assumption}  \label{Ass_29}
$ k \in \{1, 2, \ldots, n\} \: $ is a fixed number{\rm ,} $ \mathcal{M} \subset \mathbb{R}^k $ is a nonempty convex compact
set{\rm ,} and
\begin{equation}
\sigma(x) \:\: = \:\: \min\limits_{y \, \in \, \mathcal{M}} \left\| \{ x \}_k \, - \, y \right\|_k \quad
\forall x \in \mathbb{R}^n.  \label{F_51}
\end{equation}
\end{assumption}

The programmed maximin function can be characterized as follows.

\begin{proposition}{\rm \cite[\S V.1]{SubbotinChentsov1981}}  \label{Pro_36}
Under Assumptions~{\rm \ref{Ass_25}, \ref{Ass_35}, \ref{Ass_29},} the programmed maximin for {\rm (\ref{F_61}), (\ref{F_41})} at
any position $ \: (t_0, x_0) \: \in \: [0, T] \times \mathbb{R}^n \: $ is determined by
\begin{equation}
V^*(t_0, x_0) \: = \: \max \left\{ \tilde{V}^*(t_0, x_0), \: 0 \right\},  \label{F_62}
\end{equation}
where
\begin{equation}
\begin{aligned}
& \tilde{V}^*(t_0, x_0) \:\: \stackrel{\mathrm{def}}{=} \:\: \max\limits_{l \, \in \, L_k} \,
\left\{ \left< l, \, \{ \Phi(T, t_0) \, x_0 \}_k \right> \:\: + {}^{{}^{{}^{{}^{{}^{{}^{{}^{{}^{{}^{}}}}}}}}} \right. \\
& \qquad
\left. + \:\: \int\limits_{t_0}^T \, \max\limits_{u_2 \, \in \, U_2} \, \min\limits_{u_1 \, \in \, U_1} \,
\left< l, \: \left\{ \Phi(T, t) \, \hat{f}(t, u_1, u_2) \right\}_k \right> \: dt \:\: - \:\: s(l; \, \mathcal{M}) \right\}.
\end{aligned}  \label{F_63}
\end{equation}
For the system~{\rm (\ref{F_40})} with a linear control input term{\rm ,} the representation~{\rm (\ref{F_63})} transforms into
\begin{equation}
\begin{aligned}
& \tilde{V}^*(t_0, x_0) \:\: = \:\: \max\limits_{l \, \in \, L_k} \,
\left\{ \left< l, \, \{ \Phi(T, t_0) \, x_0 \}_k \right> \:\: + {}^{{}^{{}^{{}^{{}^{{}^{{}^{{}^{{}^{}}}}}}}}} \right. \\
& \qquad\qquad \:\:
+ \:\: \int\limits_{t_0}^T \min\limits_{u_1 \, \in \, U_1}
\left< l, \: \{ \Phi(T, t) \, B_1(t) \, u_1(t) \}_k \right> \: dt \:\: + \\
& \qquad\qquad \:\:
\left. + \:\: \int\limits_{t_0}^T \max\limits_{u_2 \, \in \, U_2}
\left< l, \: \{ \Phi(T, t) \, B_2(t) \, u_2(t) \}_k \right> \: dt \:\: - \:\: s(l; \, \mathcal{M}) \right\}.
\end{aligned}  \label{F_53}
\end{equation}
\end{proposition}

One more assumption is required.

\begin{assumption}  \label{Ass_37}
There exist continuous functions $ \: h \, \colon \, [0, T] \to \mathbb{R} ${\rm ,}
$ \: \tilde{z} \, \colon \, [0, T] \to \mathbb{R}^k \: $ and a mapping $ \Pi $ of $ [0, T] $ into the family of all nonempty
convex compact sets in $ \mathbb{R}^k ${\rm ,} such that
\begin{equation}
\begin{aligned}
& \int\limits_t^T \, \max\limits_{u_2 \, \in \, U_2} \, \min\limits_{u_1 \, \in \, U_1} \,
\left< l, \: \left\{ \Phi(T, \xi) \, \hat{f}(\xi, u_1, u_2) \right\}_k \right> \: d \xi \:\: - \:\: s(l; \, \mathcal{M}) \\
& \qquad\qquad\qquad\qquad\qquad \:\:
= \:\: \left< l, \, \tilde{z}(t) \right> \:\: - \:\: s(l; \, \Pi(t)) \:\: + \:\: h(t) \\
& \forall \: (t, l) \: \in \: [0, T] \times L_k
\end{aligned}  \label{F_64}
\end{equation}
and
\begin{equation}
\max\limits_{l \, \in \, L_k} \, \{ \left< l, y \right> \: - \: s(l; \, \Pi(t)) \} \:\: \geqslant \:\: 0 \quad
\forall \: (t, y) \: \in \: [0, T] \times \mathbb{R}^k.  \label{F_65}
\end{equation}
\end{assumption}

The sought-after result can now be formulated.

\begin{theorem}{\rm \cite[\S V.1]{SubbotinChentsov1981}}  \label{Thm_38}
Under Assumptions~{\rm \ref{Ass_25}, \ref{Ass_35}, \ref{Ass_29}, \ref{Ass_37},} the lower {\rm (}$ \sup $--$ \inf ${\rm )}
closed-loop game value function for {\rm (\ref{F_61}), (\ref{F_41})} is represented as
\begin{equation}
\begin{aligned}
& V_{\mathrm{low}}(t_0, x_0) \:\: = \:\: \max \, \left\{ V^*(t_0, x_0), \:\: \max\limits_{t \, \in \, [t_0, T]} \, h(t) \right\} \\
& \forall \: (t_0, x_0) \: \in \: [0, T] \times \mathbb{R}^n,
\end{aligned}  \label{F_66}
\end{equation}
where $ V^* $ is the programmed maximin function specified in Proposition~{\rm \ref{Pro_36}}. If{\rm ,} moreover{\rm ,}
the Isaacs condition
\begin{equation}
\begin{aligned}
& \min_{u_1 \, \in \, U_1} \, \max_{u_2 \, \in \, U_2} \, \left< p, \, \hat{f}(t, u_1, u_2) \right> \:\: = \:\:
\max_{u_2 \, \in \, U_2} \, \min_{u_1 \, \in \, U_1} \, \left< p, \, \hat{f}(t, u_1, u_2) \right> \\
& \: \forall \: (t, p) \: \in \: [0, T] \times \mathbb{R}^n
\end{aligned}  \label{F_67}
\end{equation}
holds{\rm ,} then there exists a closed-loop game value function{\rm ,} i.\,e.{\rm ,}
$$
V(t_0, x_0) \: = \: V_{\mathrm{low}}(t_0, x_0) \: = \: V_{\mathrm{up}}(t_0, x_0) \quad
\forall \: (t_0, x_0) \: \in \: [0, T] \times \mathbb{R}^n,
$$
and{\rm ,} for all $ \: (t, x) \: \in \: [0, T] \times \mathbb{R}^n, \: $ the corresponding saddle feedback strategies can be
determined~by
$$
\bar{u}_i(t, x) \: \in \: \bar{U}_i(t, x),
$$
where the sets $ \: \bar{U}_i(t, x) ${\rm ,} $ i = 1,2, \: $ are described as follows{\rm :}
\begin{equation}
\begin{aligned}
& \bar{U}_i(t, x) \, = \, U_i, \:\: i = 1,2, \quad \mathrm{if} \:\:\: \tilde{V}^*(t, x) \: \leqslant \: \max \, \{ 0, \, h(t) \}, \\
& \bar{L}(t, x) \:\: \stackrel{\mathrm{def}}{=} \:\: \mathrm{Arg} \, \max\limits_{l \in L_k} \,
\left\{ \left< l, \: \{ \Phi(T, t) \, x \}_k \, + \, \tilde{z}(t) \right> \: - \:
s(l; \, \Pi(t)) \right\}, \\
& \bar{L}(t, x) \: = \: \left\{ \bar{l}(t, x) \right\} \:\:\: \mbox{\rm is singleton} \:\:\:
\mathrm{if} \:\:\: \tilde{V}^*(t, x) \: > \: \max \, \{ 0, \, h(t) \}, \\
& \bar{U}_1(t, x) \:\: = \:\: \left\{ \bar{u}_1 \in U_1 \: \colon \:
\max\limits_{u_2 \, \in \, U_2} \, \left< \bar{l}(t, x), \: \left\{ \Phi(T, t) \,
\hat{f} \left( t, \bar{u}_1, u_2 \right) \right\}_k \right> \:\: = \right. \\
& \qquad\qquad\qquad\quad \:
\left. = \:\: \min\limits_{u_1 \, \in \, U_1} \, \max\limits_{u_2 \, \in \, U_2} \,
\left< \bar{l}(t, x), \: \left\{ \Phi(T, t) \, \hat{f} \left( t, u_1, u_2 \right) \right\}_k \right> \right\} \\
& \qquad\qquad\qquad\qquad\qquad\qquad\qquad\qquad\qquad\quad
\mathrm{if} \:\:\: \tilde{V}^*(t, x) \: > \: \max \, \{ 0, \, h(t) \}, \\
& \bar{U}_2(t, x) \:\: = \:\: \left\{ \bar{u}_2 \in U_2 \: \colon \:
\min\limits_{u_1 \, \in \, U_1} \, \left< \bar{l}(t, x), \: \left\{ \Phi(T, t) \,
\hat{f} \left( t, u_1, \bar{u}_2 \right) \right\}_k \right> \:\: = \right. \\
& \qquad\qquad\qquad\quad \:
\left. = \:\: \max\limits_{u_2 \, \in \, U_2} \, \min\limits_{u_1 \, \in \, U_1} \,
\left< \bar{l}(t, x), \: \left\{ \Phi(T, t) \, \hat{f} \left( t, u_1, u_2 \right) \right\}_k \right> \right\} \\
& \qquad\qquad\qquad\qquad\qquad\qquad\qquad\qquad\qquad\quad
\mathrm{if} \:\:\: \tilde{V}^*(t, x) \: > \: \max \, \{ 0, \, h(t) \}.
\end{aligned}  \label{F_68}
\end{equation}
\end{theorem}

One can use Theorem~\ref{Thm_38} in the next two examples.

\begin{example}{\rm \cite[\S V.2]{SubbotinChentsov1981}}  \label{Exa_39}  \rm
For the game
\begin{equation}
\left\{ \begin{aligned}
& \dot{x}_1(t) \:\: = \:\: x_3(t) \: + \: v_1(t), \\
& \dot{x}_2(t) \:\: = \:\: x_4(t) \: + \: v_2(t), \\
& \dot{x}_3(t) \:\: = \:\: u_1(t), \\
& \dot{x}_4(t) \:\: = \:\: u_2(t), \\
& x(t) \: = \: (x_1(t), \, x_2(t), \, x_3(t), \, x_4(t))^{\top} \: \in \: \mathbb{R}^4, \\
& u(t) \: = \: (u_1(t), \, u_2(t))^{\top} \: \in \: U_1 \: = \: \mathcal{B}_2(O_2, a_1), \\
& v(t) \: = \: (v_1(t), \, v_2(t))^{\top} \: \in \: U_2 \: = \: \mathcal{B}_2(O_2, a_2), \\
& a_i \, = \, \mathrm{const} \, > \, 0, \quad i = 1,2, \\
& t \in [0, T], \quad \mbox{$ T > 0 $ is fixed}, \\
& k = 2, \quad \mathcal{M} = \{ O_2 \}, \\
& \sigma(x(T)) \:\: = \:\: \| \{ x(T) \}_2 \|_2 \:\: = \:\: \sqrt{x_1^2(T) \, + \, x_2^2(T)} \:\: \longrightarrow \\
& \qquad\qquad\qquad\qquad\qquad\quad \:\:\,
\longrightarrow \:\: \inf_{u(\cdot)} \, \sup_{v(\cdot)} \:\: \mbox{or} \:\:
\sup_{v(\cdot)} \, \inf_{u(\cdot)} \, ,
\end{aligned} \right.  \label{F_58_0}
\end{equation}
Theorem~\ref{Thm_38} leads to the representation
\begin{equation}
\begin{aligned}
& V \left( t_0, \, x^0 \right) \:\: = \:\: \max \, \left\{ V^* \left( t_0, \, x^0 \right), \:\:
\max\limits_{t \, \in \, [t_0, T]}
\left\{ \left( a_2 \, - \, \frac{a_1}{2} \, (T - t) \right) \, (T - t) \right\} \right\} \\
& \forall \: \left( t_0, \, x^0 \right) \: \in \: [0, T] \times \mathbb{R}^4,
\end{aligned}  \label{F_58}
\end{equation}
where $ V^* $ is determined according to Proposition~\ref{Pro_36}. \qed
\end{example}

\begin{example}{\rm \cite[\S V.1]{SubbotinChentsov1981}}  \label{Exa_40}  \rm
For the game
\begin{equation}
\left\{ \begin{aligned}
& \dot{x}_1(t) \:\: = \:\: x_3(t) \: + \: v_1(t), \\
& \dot{x}_2(t) \:\: = \:\: x_4(t) \: + \: v_2(t), \\
& \dot{x}_3(t) \:\: = \:\: -\alpha \, x_3(t) \: + \: u_1(t), \\
& \dot{x}_4(t) \:\: = \:\: -\alpha \, x_4(t) \: + \: u_2(t), \\
& x(t) \: = \: (x_1(t), \, x_2(t), \, x_3(t), \, x_4(t))^{\top} \: \in \: \mathbb{R}^4, \\
& u(t) \: = \: (u_1(t), \, u_2(t))^{\top} \: \in \: U_1 \: = \: U^0 \, + \, \mathcal{B}_2(O_2, a), \\
& U^0 \:\: \stackrel{\mathrm{def}}{=} \:\: \left\{ \mu \, u^0 \: \colon \: \mu \in [-1, 1] \right\}, \\
& v(t) \: = \: (v_1(t), \, v_2(t))^{\top} \: \in \: U_2 \: = \: \mathcal{B}_2(O_2, b), \\
& \alpha > 0, \:\: a > 0, \:\: b > 0 \:\: \mbox{are scalar constants}, \\
& u^0 \: = \: \left( u^0_1, u^0_2 \right)^{\top} \: \in \: \mathbb{R}^2 \:\: \mbox{is a constant vector}, \\
& t \in [0, T], \quad  \mbox{$ T > 0 $ is fixed}, \\
& k = 2, \quad \mathcal{M} = \{ O_2 \}, \\
& \sigma(x(T)) \:\: = \:\: \| \{ x(T) \}_2 \|_2 \:\: = \:\: \sqrt{x_1^2(T) \, + \, x_2^2(T)} \:\: \longrightarrow \\
& \qquad\qquad\qquad\qquad\qquad\quad \:\:\,
\longrightarrow \:\: \inf_{u(\cdot)} \, \sup_{v(\cdot)} \:\: \mbox{or} \:\:
\sup_{v(\cdot)} \, \inf_{u(\cdot)} \, ,
\end{aligned} \right.  \label{F_69_0}
\end{equation}
Proposition~\ref{Pro_36} and Theorem~\ref{Thm_38} lead to the representations
\begin{equation}
\begin{aligned}
& V \left( t_0, x^0 \right) \:\: = \:\: \max \, \left\{ V^* \left( t_0, x^0 \right), \:\:
\max\limits_{t \, \in \, [t_0, T]} \, h(t) \right\}, \\
& V^* \left( t_0, x^0 \right) \:\: = \:\: \max \left\{ \tilde{V}^* \left( t_0, x^0 \right), \: 0 \right\}, \\
& \tilde{V}^* \left( t_0, x^0 \right) \:\: = \:\: \max\limits_{l \: = \: (l_1, l_2) \: \in \: L_2}
\left\{ \left( x^0_1 \: + \: r_{\alpha}(T, t_0) \, x^0_3 \right) \, l_1 \:\: + \right. \\
& \qquad\quad
\left. + \:\: \left( x^0_2 \: + \: r_{\alpha}(T, t_0) \, x^0_4 \right) \, l_2 \:\: - \:\: R_{\alpha}(T, t) \,
\left| l_1 \, u^0_1 \: + \: l_2 \, u^0_2 \right| \right\} \:\: + \:\: h(t_0) \\
& \forall \: \left( t_0, \, \left( x^0 \right)^{\top} \right) \: = \:
\left( t_0, \, x^0_1, \, x^0_2, \, x^0_3, \, x^0_4 \right) \: \in \: [0, T] \times \mathbb{R}^4,
\end{aligned}  \label{F_69}
\end{equation}
where
\begin{equation}
\begin{aligned}
& h(t) \:\: = \:\: (T - t) \, b \: - \: a \, R_{\alpha}(T, t), \\
& r_{\alpha}(T, t) \:\: \stackrel{\mathrm{def}}{=} \:\: \frac{1 \, - \, e^{-\alpha \, (T - t)}}{\alpha} \:\: \geqslant \:\: 0, \\
& R_{\alpha}(T, t) \:\: \stackrel{\mathrm{def}}{=} \:\: \int\limits_t^T r_{\alpha}(T, \xi) \, d \xi \:\: = \:\:
\frac{T - t}{\alpha} \: - \: \frac{1 \, - \, e^{-\alpha \, (T - t)}}{\alpha^2} \:\: \geqslant \:\: 0 \\
& \forall t \in [0, T].
\end{aligned}  \label{F_69_2}
\end{equation}
Furthermore, at any position $ \: (t, x) \: \in \: [0, T] \times \mathbb{R}^n, \: $ the related saddle feedback strategies
can be chosen according to
$$
\bar{u}(t, x) \: \in \: \bar{U}_1(t, x), \quad \bar{v}(t, x) \: \in \: \bar{U}_2(t, x),
$$
where the sets $ \: \bar{U}_i(t, x) $, $ i = 1,2, \: $ are determined by (\ref{F_68}) with the following specifications:
\begin{equation}
\begin{aligned}
& \Phi(T, t) \:\: = \:\: \begin{pmatrix}
1 & 0 & r_{\alpha}(T, t) & 0 \\
0 & 1 & 0 & r_{\alpha}(T, t) \\
0 & 0 & e^{-\alpha \, (T - t)} & 0 \\
0 & 0 & 0 & e^{-\alpha \, (T - t)}
\end{pmatrix}, \\
& \tilde{z}(t) \: \equiv \: (0, 0)^{\top}, \\
& \{ \Phi(T, t) \, x \}_2 \, + \, \tilde{z}(t) \:\: = \:\: \begin{pmatrix}
x_1 \: + \: r_{\alpha}(T, t) \, x_3 \\
x_2 \: + \: r_{\alpha}(T, t) \, x_4
\end{pmatrix}, \\
& \Phi(T, t) \, \hat{f}(t, u, v) \:\: = \:\: \Phi(T, t) \: \begin{pmatrix}
v_1 \\
v_2 \\
u_1 \\
u_2
\end{pmatrix} \:\: = \:\: \begin{pmatrix}
v_1 \: + \: r_{\alpha}(T, t) \, u_1 \\
v_2 \: + \: r_{\alpha}(T, t) \, u_2 \\
e^{-\alpha \, (T - t)} \, u_1 \\
e^{-\alpha \, (T - t)} \, u_2
\end{pmatrix}, \\
& \left< l, \: \{ \Phi(T, t) \, \hat{f}(t, u, v) \}_2 \right> \:\: = \:\: \left< l, v \right> \: + \:
r_{\alpha}(T, t) \, \left< l, u \right>, \\
& \max\limits_{\tilde{v} \, \in \, U_2} \left< l, \tilde{v} \right> \:\:\: \mbox{is reached at} \:\:\:
\tilde{v} \: = \: b \, l, \\
& \min\limits_{\tilde{u} \, \in \, U_1} \left< l, \tilde{u} \right> \:\:\: \mbox{is reached at all} \:\:\:
\tilde{u} \: \in \: \left( \{ -a \, l \} \: + \: \mathrm{Arg} \, \min_{w \, \in \, U^0} \, \left< l, w \right> \right), \\
& \Pi(t) \:\: = \:\: R_{\alpha}(T, t) \, U^0 \:\: = \:\:
\left\{ \mu \, R_{\alpha}(T, t) \, u^0 \: \colon \: \mu \in [-1, 1] \right\}, \\
& \left< l, \: \{ \Phi(T, t) \, x \}_2 \, + \, \tilde{z}(t) \right> \: - \: s (l; \, \Pi(t)) \:\: = \:\:
(x_1 \: + \: r_{\alpha}(T, t) \, x_3) \, l_1 \:\: + \\
& \qquad\qquad\qquad\quad
+ \:\: (x_2 \: + \: r_{\alpha}(T, t) \, x_4) \, l_2 \:\: - \:\:
R_{\alpha}(T, t) \, \max\limits_{w \, \in \, U^0} \, \left< l, w \right> \\
& \forall \: \left( t, \, x^{\top}, \, u^{\top}, \, v^{\top}, \, l^{\top} \right) \: = \:
(t, \, x_1, \, x_2, \, x_3, \, x_4, \, u_1, \, u_2, \, v_1, \, v_2, \, l_1, \, l_2) \:\: \in \\
& \qquad\qquad\qquad\qquad\qquad\qquad\qquad\quad \:\:
\in \: [0, T] \times \mathbb{R}^4 \times U_1 \times U_2 \times L_2.
\end{aligned}  \label{F_70}
\end{equation}
Thus{\rm ,}
\begin{equation}
\begin{aligned}
& \bar{U}_1(t, x) \:\: = \:\: \left\{ -a \, \bar{l}(t, x) \right\} \: + \: \mathrm{Arg} \, \min_{w \, \in \, U^0} \,
\left< \bar{l}(t, x), \, w \right>, \\
& \mathrm{Arg} \, \min_{w \, \in \, U^0} \, \left< \bar{l}(t, x), \, w \right> \:\: = \:\: \begin{cases}
u^0, & \left< \bar{l}(t, x), \, u^0 \right> \: < \: 0, \\
-u^0, & \left< \bar{l}(t, x), \, u^0 \right> \: > \: 0, \\
U^0, & \left< \bar{l}(t, x), \, u^0 \right> \: = \: 0, \\
\end{cases} \\
& \bar{U}_2(t, x) \: = \: \{ b \, \bar{l}(t, x) \}, \\
& \bar{L}(t, x) \:\: = \:\: \mathrm{Arg} \, \max\limits_{l \: = \: (l_1, l_2) \: \in \: L_2} \,
\left\{ (x_1 \: + \: r_{\alpha}(T, t) \, x_3) \, l_1 \:\: + {}^{{}^{{}^{}}} \right. \\
& \qquad
\left. + \:\: (x_2 \: + \: r_{\alpha}(T, t) \, x_4) \, l_2 \:\: - \:\:
R_{\alpha}(T, t) \, \left| l_1 \, u^0_1 \: + \: l_2 \, u^0_2 \right| \right\} \:\: = \:\:
\left\{ \bar{l}(t, x) \right\} \\
& \mathrm{if} \:\: \tilde{V}^*(t, x) \: > \: \max \, \{ 0, \, h(t) \},
\end{aligned}  \label{F_71}
\end{equation}
and
\begin{equation}
\bar{U}_i(t, x) \, = \, U_i, \:\: i = 1,2, \quad \mathrm{if} \:\:\: \tilde{V}^*(t, x) \: \leqslant \: \max \, \{ 0, \, h(t) \}.
\label{F_72}
\end{equation}
\qed
\end{example}

In Appendix, we describe one more class of differential games for which a single programmed iteration leads to the closed-loop
game value.

\section{Numerical simulations}

In this section, we discuss our computational results. The numerical simulations have been conducted (without algorithm
parallelization) on a relatively weak machine with 1.4 GHz Intel 2957U CPU, and the corresponding runtimes are mentioned here.

\begin{example}  \label{Exa_43}  \rm
Consider the problem~(\ref{F_73}) from Example~\ref{Exa_41} with
\begin{equation}
\arraycolsep=1.5pt
\def\arraystretch{2}
\begin{array}{c}
c(x) \:\: = \:\: 1 \: + \: 3 \:
\exp \left( -4 \: \| x \: - \: (1, \, 1, \, 0, \, 0, \ldots, \, 0) \|_{\mathbb{R}^n}^2 \right) \:\: > \:\: 0, \\
\sigma(x) \:\: = \:\: \dfrac{1}{2} \: (\left< Ax, x \right> \: - \: 1), \\
A \:\: = \:\: \mathrm{diag} \: [0.25, \, 1, \, 0.5, \, 0.5, \, \ldots, \, 0.5] \:\: \in \:\: \mathbb{R}^{n \times n}.
\end{array}  \label{F_85}
\end{equation}
This particular problem appears from the problem of \cite[Section~5, Example~3]{ChowDarbonOsherYin2017} just by
changing the first diagonal element of the matrix~$ A $ from $ 2.5 $ to $ 0.25 $. Since $ \mathrm{D} \sigma $ vanishes only
at the point $ x = 0 $ which gives the global minimum to $ \sigma $, then Theorem~\ref{Thm_20} can be directly applied here
(together with Remark~\ref{Rem_21}), and the finite-dimensional optimization can be performed over the unit sphere~(\ref{F_84}),
so that the choice of extremal control values is unique (recall the reasonings of Example~\ref{Exa_41}).

Fig.~\ref{Fig_1} indicates the two value function approximations~$ V_{\mathrm{MoC}} $ and $ V_{\mathrm{FD}} $ constructed
respectively by the method of characteristics (Theorem~\ref{Thm_20}, Example~\ref{Exa_41}) and via the monotone Lax-Friedrichs
finite-difference scheme \cite{CrandallLions1984,OsherShu1991,ROCHJ2017} (which ensures a theoretical convergence property and
an error estimate) for $ n = 2 $ (two-dimensional state space) and $ \, T - t_0 = 0.5 $. Some related level sets are depicted in
Fig.~\ref{Fig_2}. They qualitatively resemble the corresponding results reported in \cite[Section~5, Example~3]{ChowDarbonOsherYin2017}.
Fig.~\ref{Fig_3} shows the optimal feedback control strategy obtained together with $ V_{\mathrm{MoC}} $ directly from
the integrated optimal characteristics.

\begin{figure}
\begin{center}
\includegraphics[ width = 8cm, height = 5.4cm ]{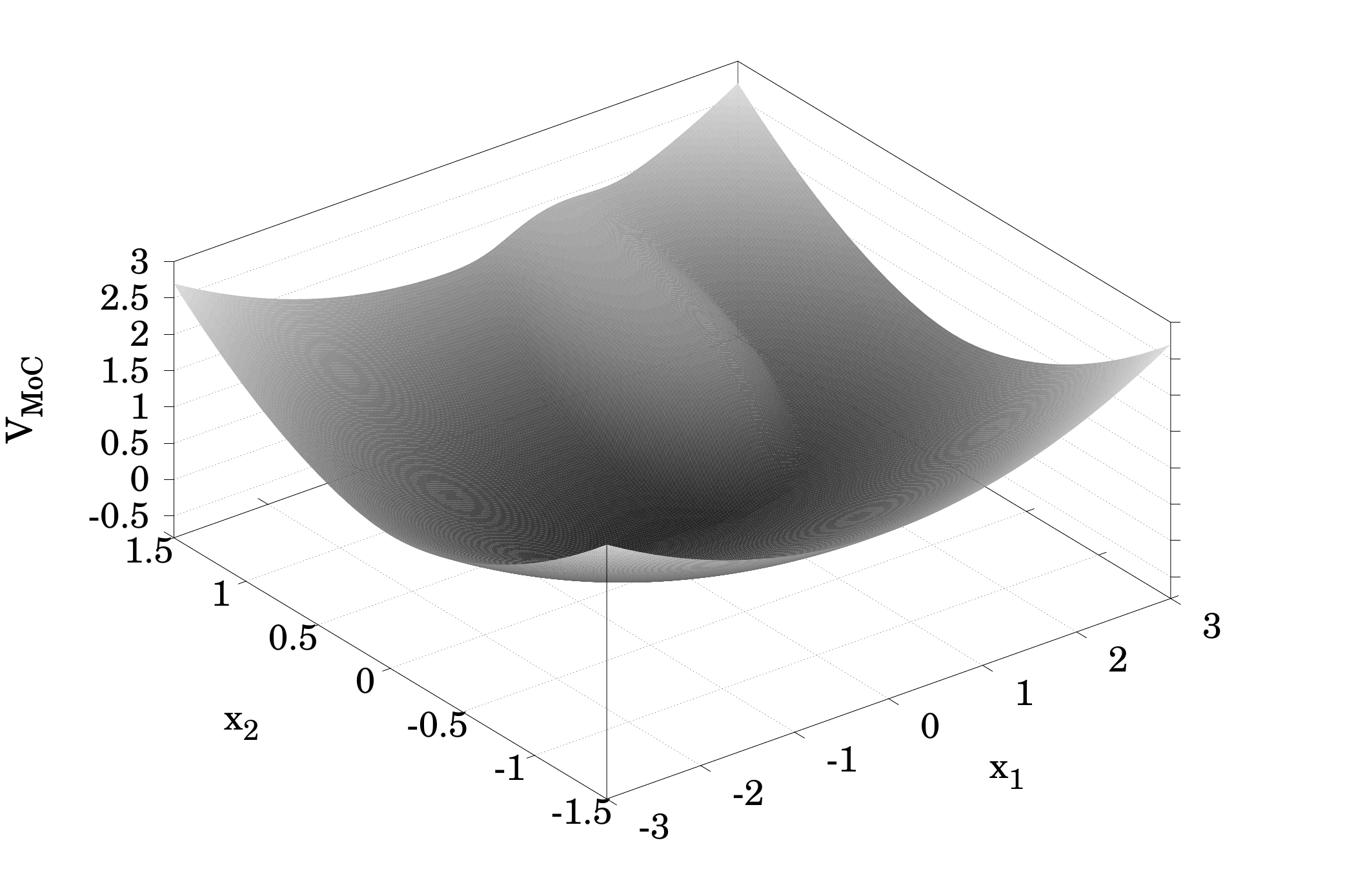}
\includegraphics[ width = 8cm, height = 5.4cm ]{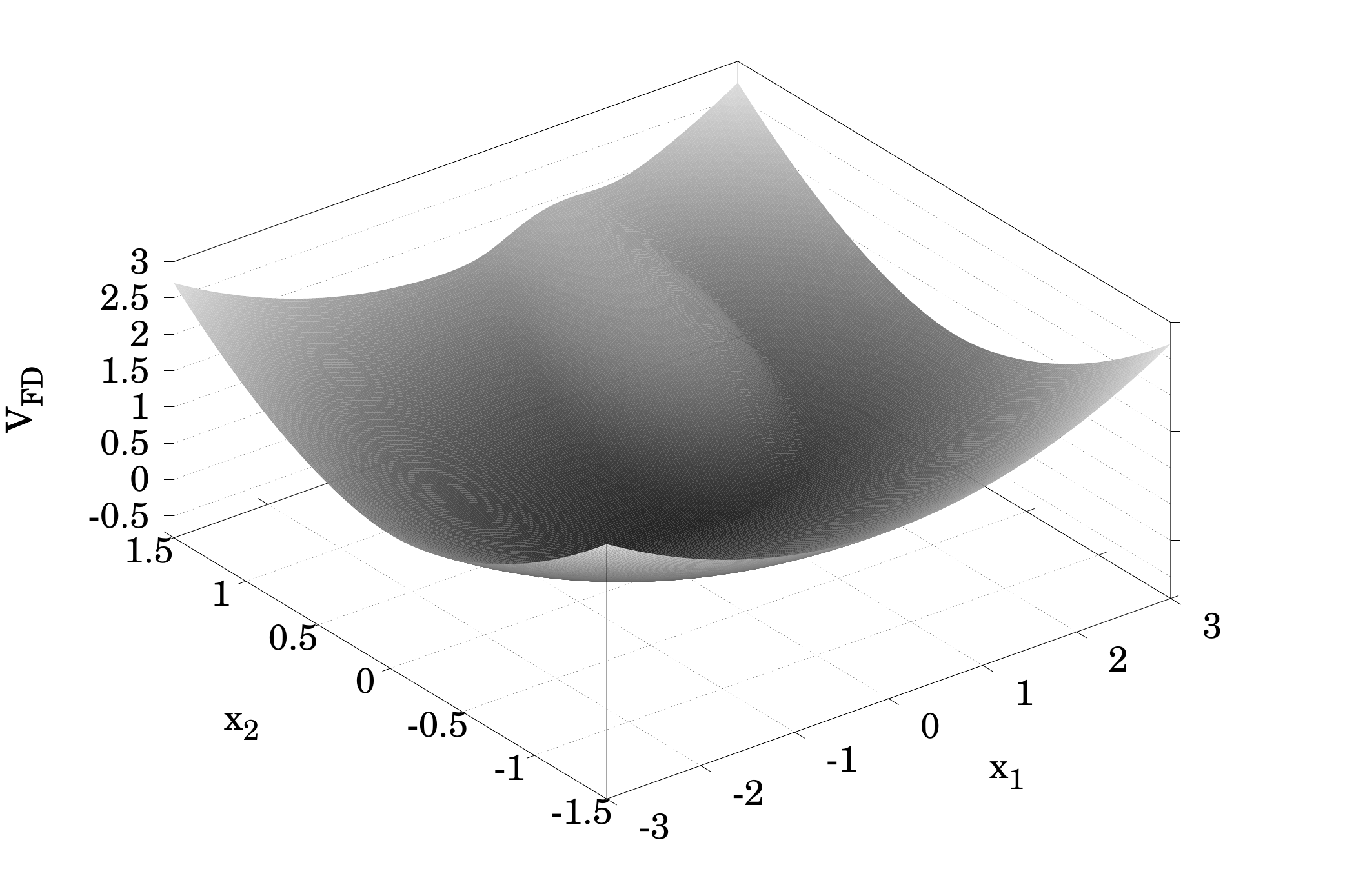} \\
\includegraphics[ width = 8cm, height = 5.4cm ]{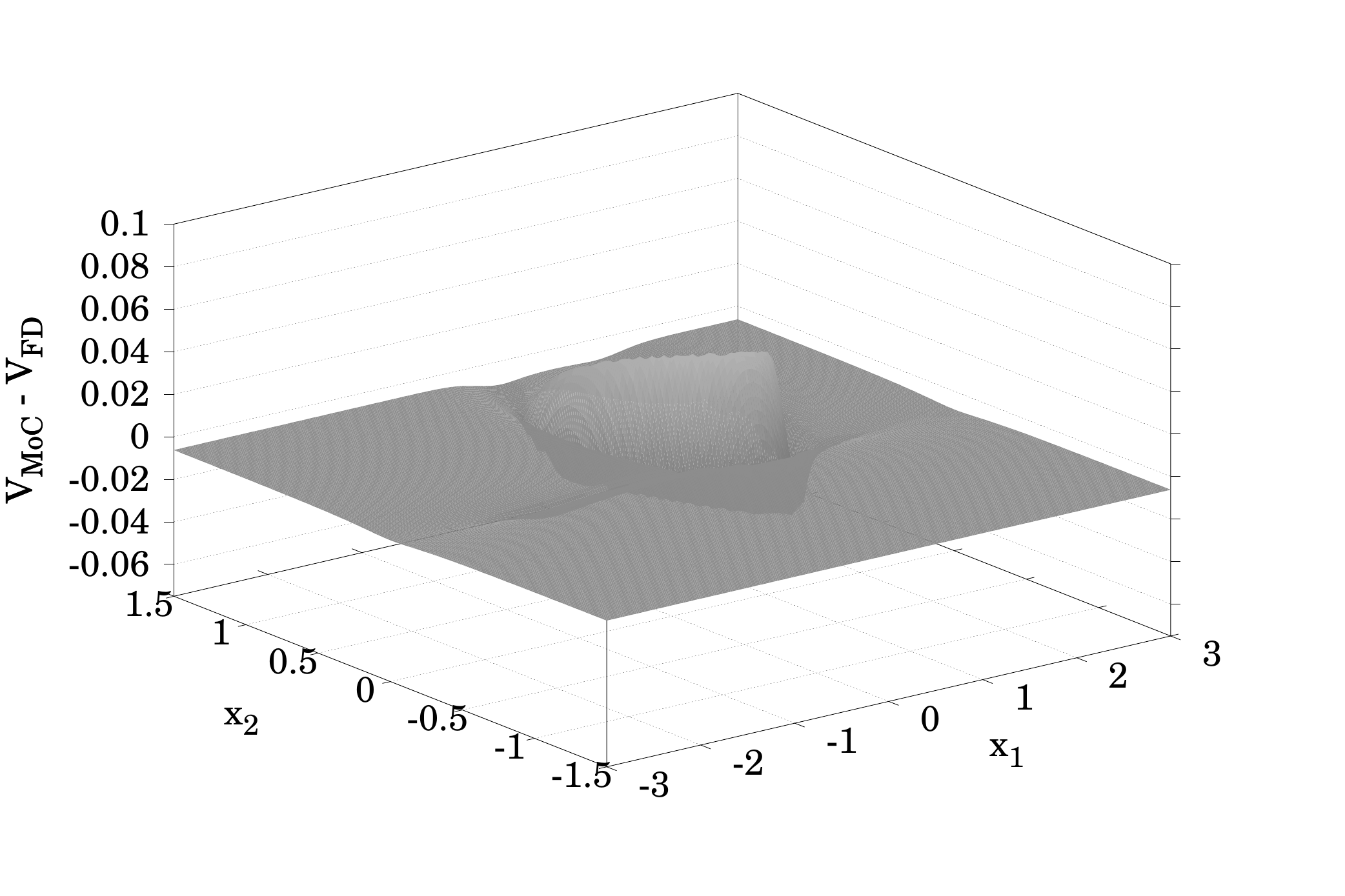}
\end{center}
\bf \caption{\it The value function approximations~$ V_{\mathrm{MoC}}, V_{\mathrm{FD}} $ and their
difference~$ \, V_{\mathrm{MoC}} - V_{\mathrm{FD}} \, $ in Example~{\rm \ref{Exa_43}} for $ n = 2 $ and $ \, T - t_0 = 0.5 $.
In order to see the graph of $ \, V_{\mathrm{MoC}} - V_{\mathrm{FD}} \, $ clearer{\rm ,} an essentially larger scale on
the vertical axis is used there.}
\label{Fig_1}
\end{figure}

\begin{figure}
\begin{center}
\includegraphics[ width = 8cm, height = 5.4cm ]{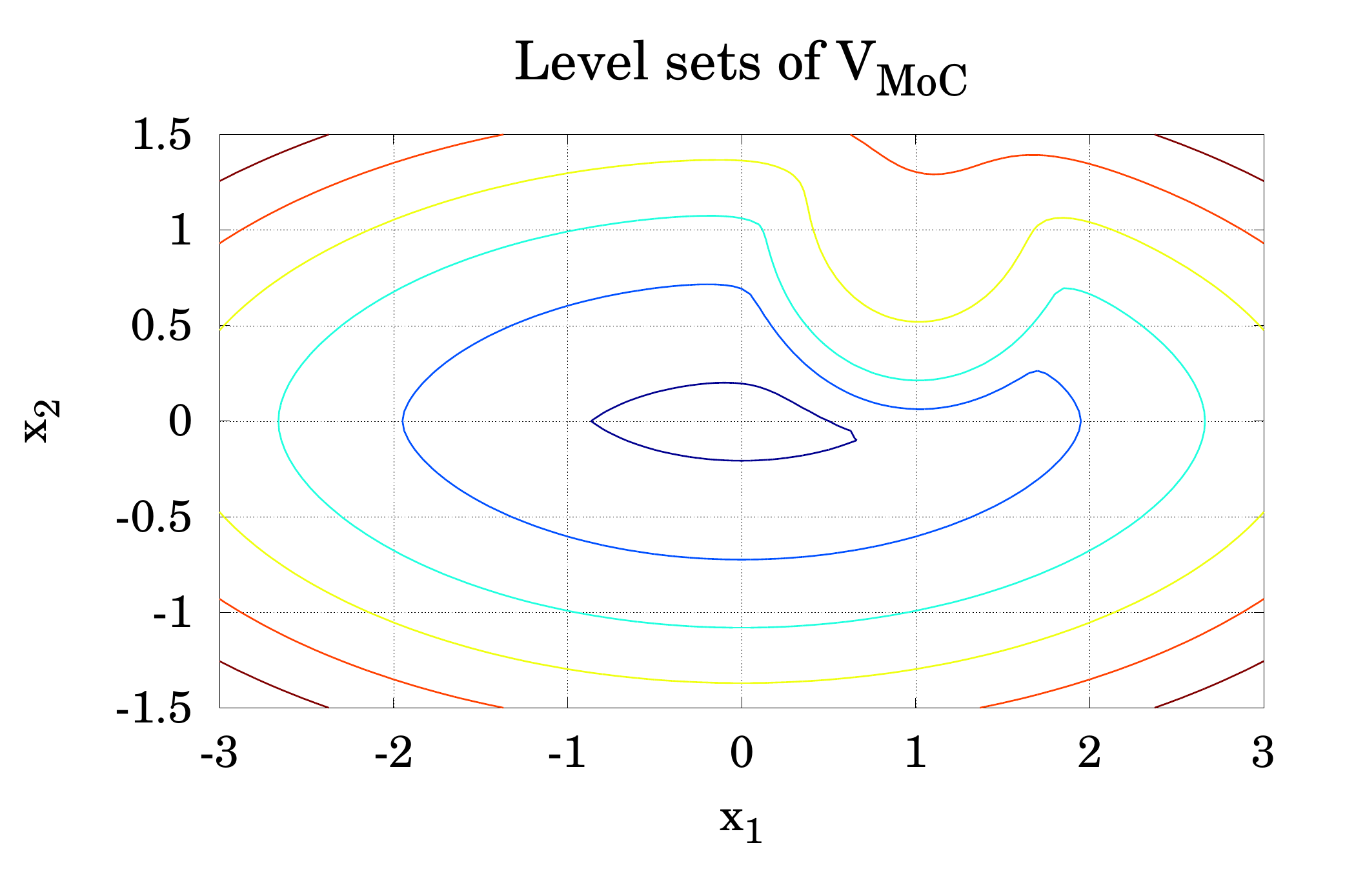}
\includegraphics[ width = 8cm, height = 5.4cm ]{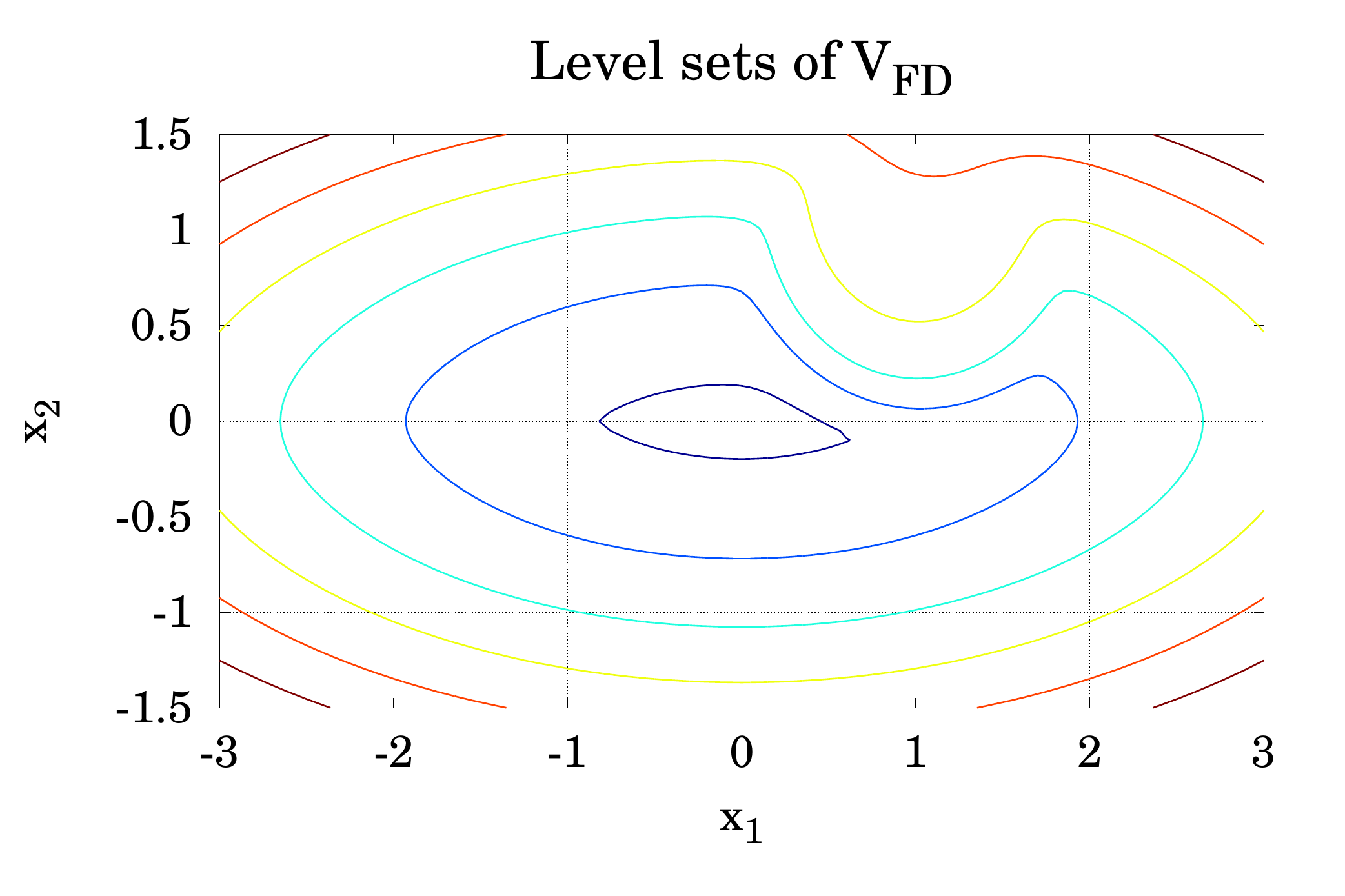}
\end{center}
\bf \caption{\it Level sets of the value function approximations~$ V_{\mathrm{MoC}} $ and $ V_{\mathrm{FD}} $ in
Example~{\rm \ref{Exa_43}} for $ n = 2 $ and $ \, T - t_0 = 0.5 $.}
\label{Fig_2}
\end{figure}

\begin{figure}
\begin{center}
\includegraphics[ width = 8cm, height = 5.4cm ]{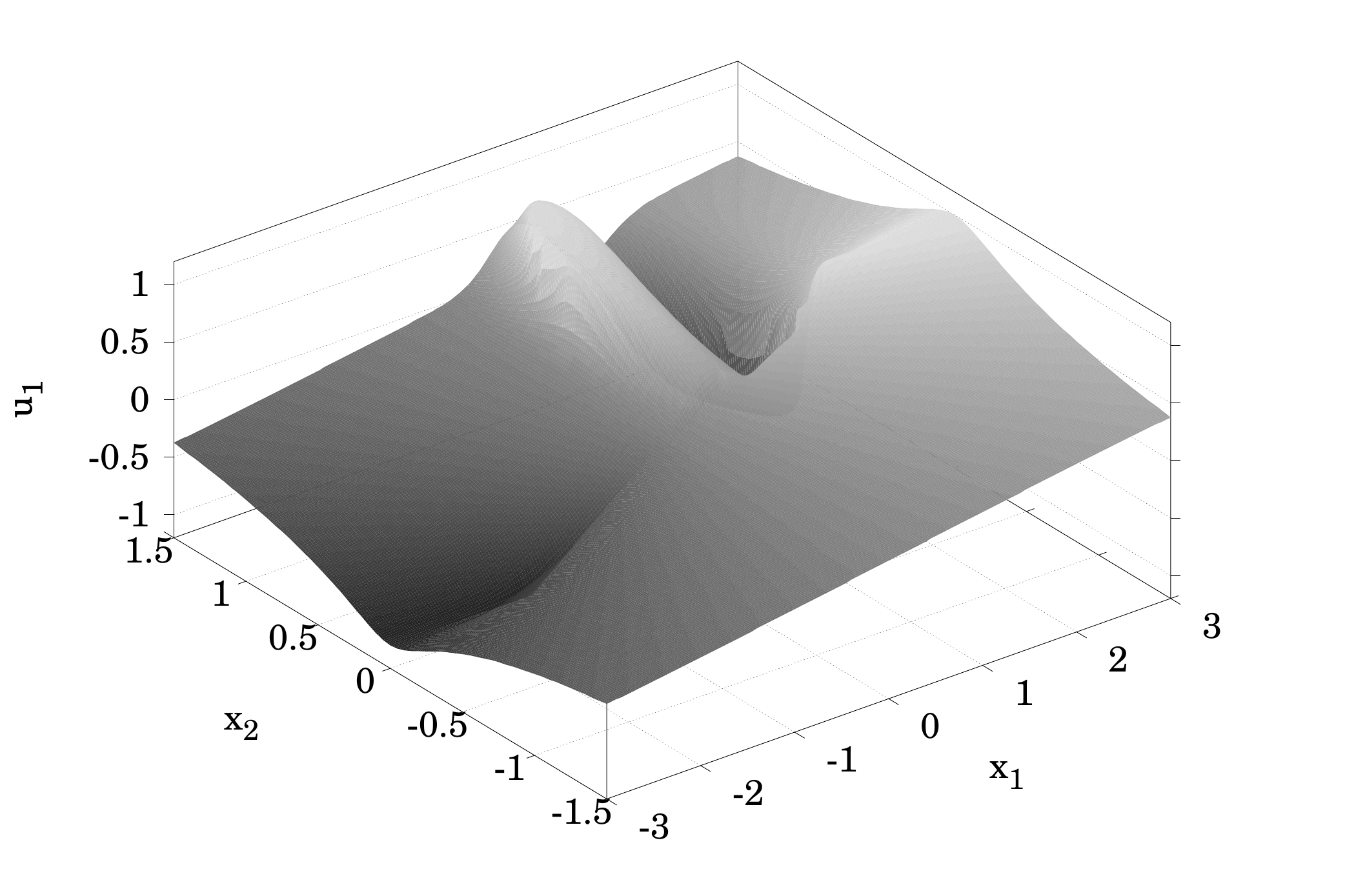}
\includegraphics[ width = 8cm, height = 5.4cm ]{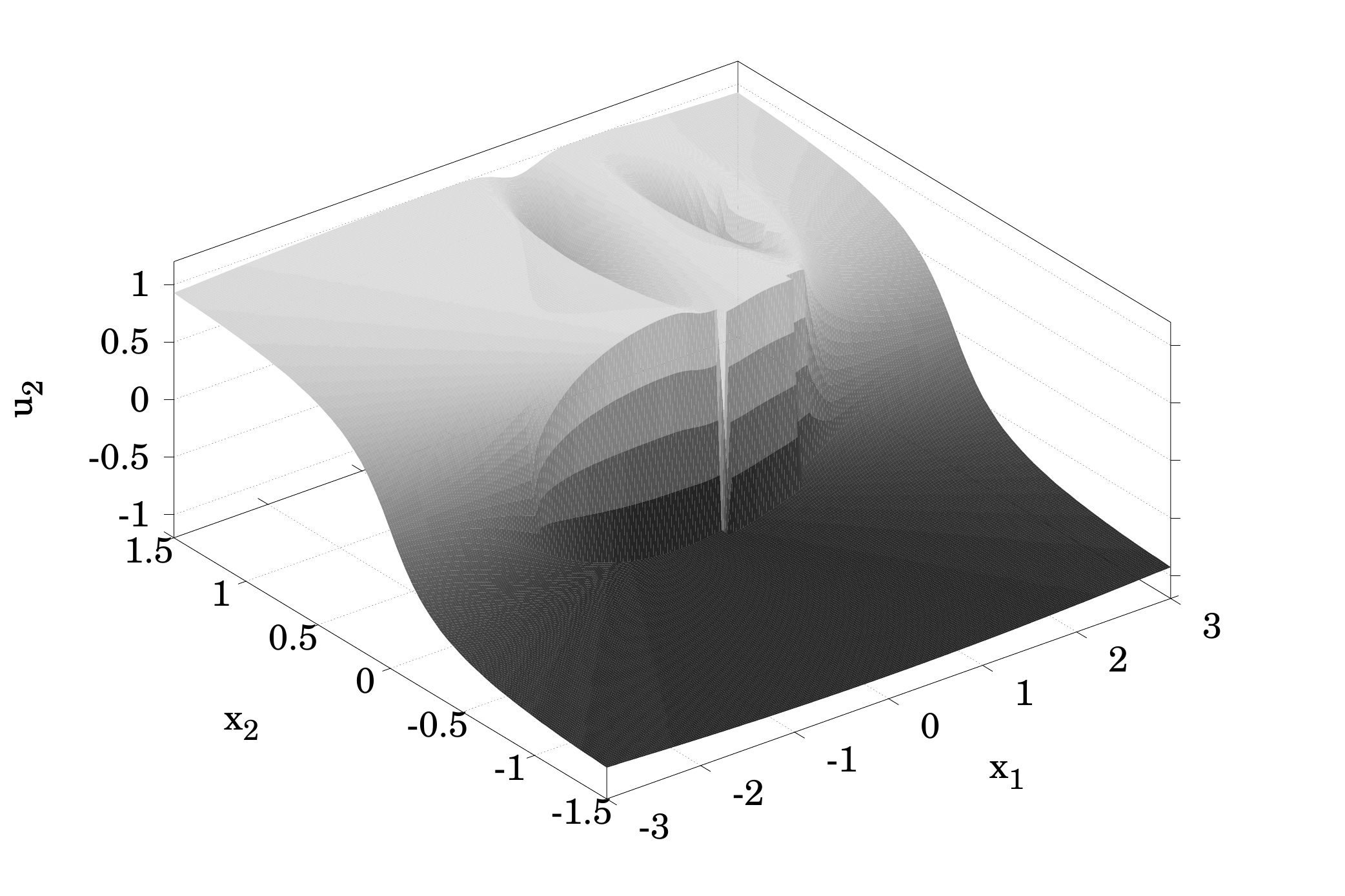}
\end{center}
\bf \caption{\it The optimal feedback control strategy related to $ V_{\mathrm{MoC}} $ in Example~{\rm \ref{Exa_43}} for
$ n = 2 $ and $ \, T - t_0 = 0.5 $.}
\label{Fig_3}
\end{figure}

The Cauchy problems~(\ref{F_79}),~(\ref{F_80}) have been solved numerically via the fifth-order Runge--Kutta algorithm from
the C++ library of \cite{PressNR2007}. We have also verified the obtained results by using the implicit Rosenbrock method
\cite[\S 17.5.1]{PressNR2007} (which is in general more stable but much more computationally expensive), and a good
agreement has been observed. When launching the Runge--Kutta routine, the initial guess for the stepsize was set as $ 10^{-3} $,
and the absolute and relative tolerances were specified as $ 10^{-5} $. The initial states were taken from the uniform grid on
the rectangle $ \, [-3, 3] \times [-1.5, 1.5] \, $ with the spatial step $ 0.05 $. The two-dimensional unit sphere of initial
adjoint vectors was parametrized by one angle with values in the interval~$ [0, 2 \pi) $. The uniform grid for the latter
consisted of $ 1000 $ points. The maximum of the cost functional was chosen directly around this grid (taking the denser grid of
$ 2000 $ points has led to very close results, as illustrated in Fig.~\ref{Fig_4}). The approximate runtime of computing the value
function by the method of characteristics (as shown in the first subfigure of Fig.~\ref{Fig_1}) has been 400.798~s totally and
0.054301314~s per point. The runtime can be decreased if the optimization over the two-dimensional unit sphere is performed
by means of an advanced one-dimensional maximization/minimization algorithm (see, for instance, \cite[Chapter~10]{PressNR2007})
after a random choice of a certain amount of starting points. Note that unconstrained maximization/minimization routines can be
reasonably used for optimization over sphere parametrizations due to the periodicity of the latter.

For computing the finite-difference approximation~$ V_{\mathrm{FD}} $, we have used the C++ package ROC--HJ
\cite{ROCHJ2017}. We chose the greater computational region $ \, [-5, 5] \times [-3, 3] \, $ in order to reduce
boundary cutoff errors in the relevant subdomain $ \, [-3, 3] \times [-1.5, 1.5] $. The spatial step was
taken uniformly as $ 5 \times 10^{-3} $, and the time step was set as $ 5 \times 10^{-4} $, so that
the Courant--Friedrichs--Lewy condition (sufficient for convergence) held. The total runtime has been around 431.475~s,
which is greater than the runtime of computing $ V_{\mathrm{MoC}} $. Therefore, the characteristics approach turns out
to be more efficient than the Lax-Friedrichs scheme for the considered problem even in the small-dimensional case.

Fig.~\ref{Fig_1} and \ref{Fig_2} show that $ V_{\mathrm{MoC}} $ and $ V_{\mathrm{FD}} $ in principle agree between each other,
while the greatest difference occurs nearly at points of nonsmoothness. In Fig.~\ref{Fig_1}, the absolute
difference~$ \, V_{\mathrm{MoC}} - V_{\mathrm{FD}} \, $ is indicated instead of a relative one, because the latter can be
far away from zero at some points where both $ V_{\mathrm{MoC}} $ and $ V_{\mathrm{FD}} $ are close to zero.

\begin{figure}
\begin{center}
\includegraphics[ width = 9cm, height = 6cm ]{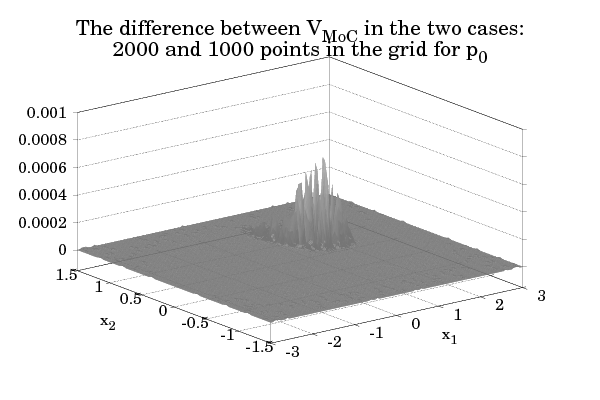}
\end{center}
\bf \caption{\it If{\rm ,} in Example~{\rm \ref{Exa_43}} for $ n = 2 $ and $ \, T - t_0 = 0.5, \, $ the number of points in
the grid on the two-dimensional unit sphere of initial adjoint vectors is increased from $ 1000 $ to $ 2000 ${\rm ,} then
the value function approximation~$ V_{\mathrm{MoC}} $ changes only negligibly. In order to see the graph of this small difference
clearer{\rm ,} a sufficiently large scale on the vertical axis is used.}
\label{Fig_4}
\end{figure}

For comparison, we have also approximated the value function by solving the boundary value problems of Pontryagin's
principle, although this is not a completely correct approach in general (recall Remark~\ref{Rem_12_0}). The characteristic
equations from (\ref{F_79}) were considered under the initial and terminal conditions $ \: x^*(t_0) = x_0 $,
$ \: p^*(T) \, \uparrow\uparrow \, \mathrm{D} \sigma(x^*(T)) $. In the related numerical algorithm, we chose $ p_0 $ from
the uniform grid consisting of $ 1000 $ points on the unit sphere~(\ref{F_84}) with the aim to maximize the dot product
$$
\left< \frac{1}{\| p^*(T) \|} \: p^*(T), \:\, \mathrm{D} \sigma \, (x^*(T)) \right>,
$$
i.\,e., to make the directions of the vectors~$ p^*(T) $ and $ \mathrm{D} \sigma \, (x^*(T)) $ as close to each other as possible.
The characteristic system was integrated via the same Runge--Kutta routine (with the same initial stepsize guess, absolute and
relative tolerances) as mentioned above. The time~$ \, T - t_0 = 0.5 \, $ and uniform grid for the initial states were also
taken the same as before. The corresponding value function approximation~$ V_{\mathrm{MoC, BVP}} $ is illustrated in
Fig.~\ref{Fig_5} together with the difference~$ \, V_{\mathrm{MoC}} - V_{\mathrm{MoC, BVP}} $. The latter is always nonnegative and
can be interpreted as an error function for $ V_{\mathrm{MoC, BVP}} $. The approximate runtime of computing $ V_{\mathrm{MoC, BVP}} $
together with the related feedback control strategy has been 399.190~s totally and 0.054083458~s per point, which is almost the same as
for $ V_{\mathrm{MoC}} $. However, the errors of $ V_{\mathrm{MoC, BVP}} $ with respect to $ V_{\mathrm{MoC}} $ are not negligible at
some of the points for which the boundary value problem of Pontryagin's principle admits multiple solutions and where the value
function is nonsmooth. By comparing the last subfigures of Fig.~\ref{Fig_1} and \ref{Fig_5}, we conclude that these errors can be
noticeably greater than the corresponding values of the absolute difference~$ \, \left| V_{\mathrm{MoC}} - V_{\mathrm{FD}} \right| $.

\begin{figure}
\begin{center}
\includegraphics[ width = 8cm, height = 5.4cm ]{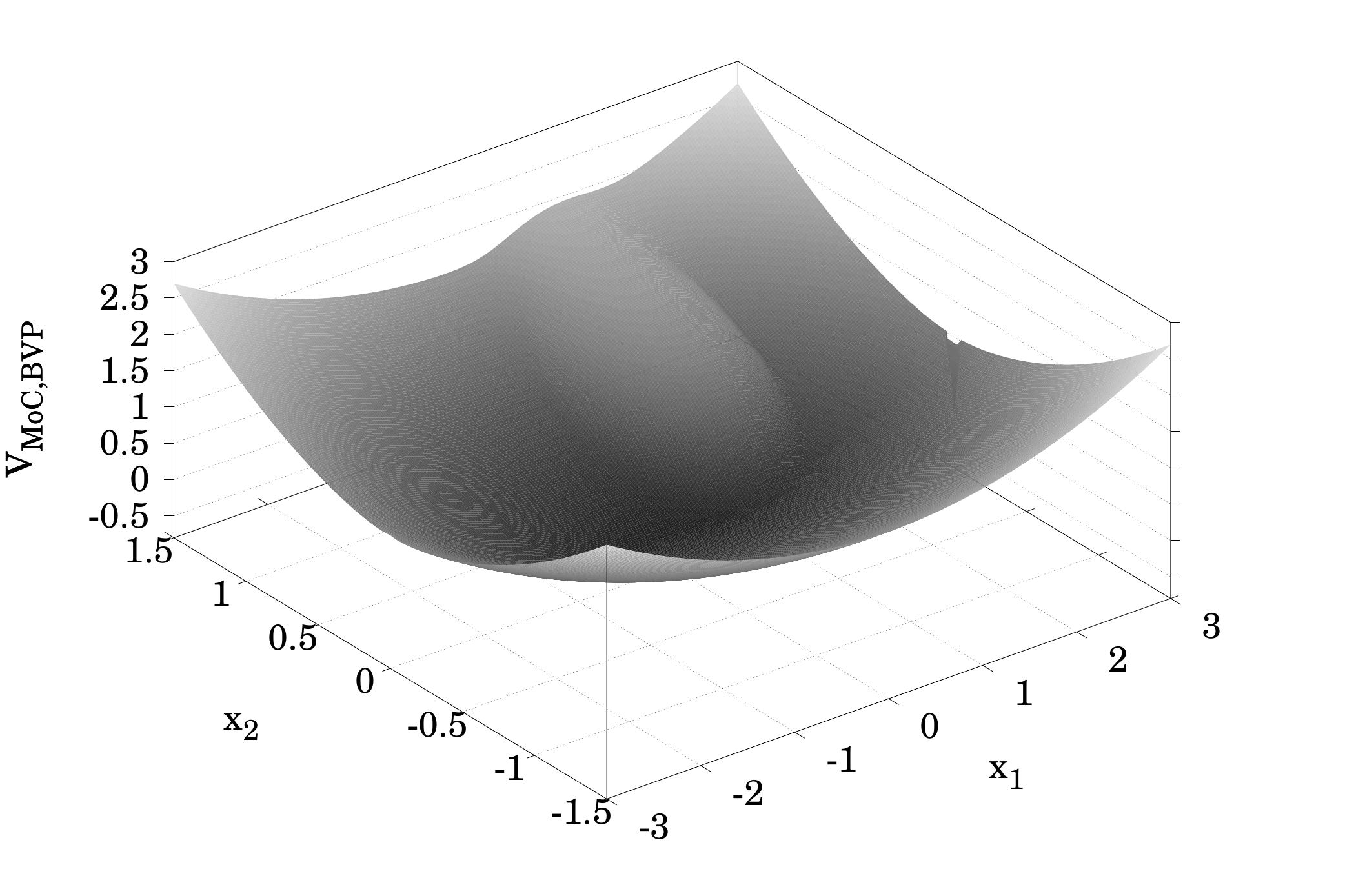}
\includegraphics[ width = 8cm, height = 5.4cm ]{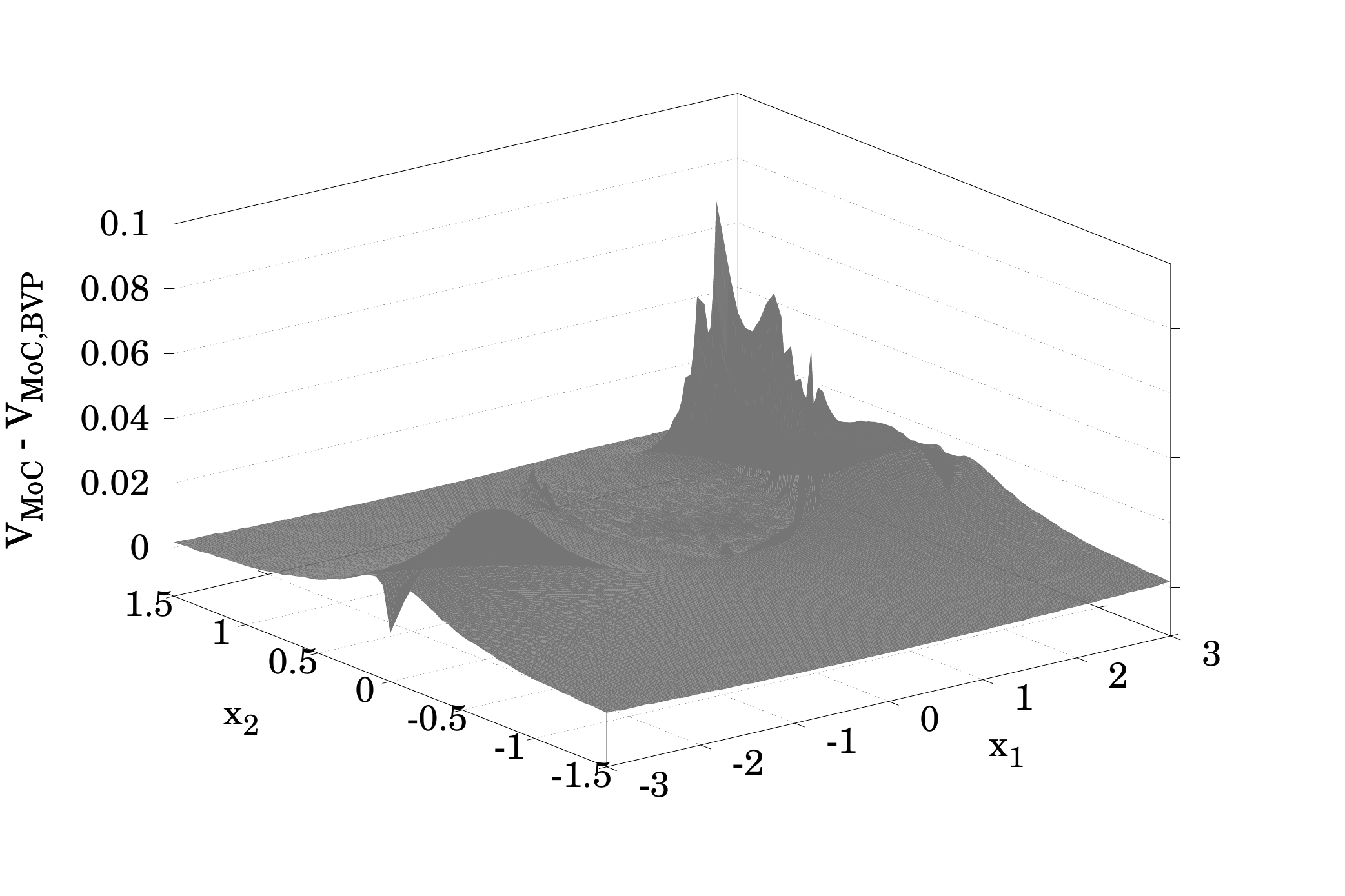}
\end{center}
\bf \caption{\it The value function approximation~$ V_{\mathrm{MoC, BVP}} $ and its
error~$ \, V_{\mathrm{MoC}} - V_{\mathrm{MoC, BVP}} \, $ in Example~{\rm \ref{Exa_43}} for $ n = 2 $ and $ \, T - t_0 = 0.5 $.
In order to see the graph of $ \, V_{\mathrm{MoC}} - V_{\mathrm{MoC, BVP}} \, $ clearer{\rm ,} an essentially larger scale on
the vertical axis is used there.}
\label{Fig_5}
\end{figure}

Finally, consider the high-dimensional case $ n = 5 $ ($ 5 $-dimensional state space) for which the curse of dimensionality
makes grid-based methods almost inapplicable. Take also $ \, T - t_0 = 0.5 $. By using the method of characteristics
(Theorem~\ref{Thm_20}, Example~\ref{Exa_41}), we have constructed the reduction of the corresponding value function
approximation~$ V_{\mathrm{MoC}} $ to the plane $ \: x_3 = x_4 = x_5 = 0 $. It is illustrated in Fig.~\ref{Fig_6} together
with some related level sets. The vector of the first two initial state coordinates~$ x_1, x_2 $ was chosen from the same uniform
grid on the rectangle $ \, [-3, 3] \times [-1.5, 1.5] \, $ as mentioned above. The $ 10 $-dimensional characteristic system was
integrated by means of the same Runge--Kutta routine (with the same initial stepsize guess, absolute and relative tolerances) as
used before. The optimization over the $ 5 $-dimensional unit sphere of initial adjoint vectors was performed via Powell's
algorithm from the C++ library of \cite{PressNR2007} (this is a zero-order method that does not require computation of
derivatives), and the tolerance parameter was specified as $ 10^{-7} $. The sphere was parametrized in the standard way:
\begin{equation}
\left\{ \begin{aligned}
& x_1 \:\: = \:\: r \: \prod\limits_{i = 1}^{n - 1} \sin \theta_i, \\
& x_j \:\: = \:\: r \: \cos \theta_{j - 1} \: \prod\limits_{i = j}^{n - 1} \sin \theta_i, \quad j = \overline{2, n - 1}, \\
& x_n \:\: = \:\: r \: \cos \theta_{n - 1}, \\
& r = 1, \quad 0 \leqslant \theta_1 < 2 \pi, \quad 0 \leqslant \theta_j \leqslant \pi, \quad j = \overline{2, n - 1}.
\end{aligned} \right.  \label{F_86}
\end{equation}
Due to the periodicity in the angles~$ \, \theta_i $, $ i = \overline{1, n - 1}, \, $ it was reasonable to use Powell's method of
unconstrained optimization with a random choice of some number of starting points. For each optimization process, we
randomly generated $ 5 $ starting points according to the uniform distribution with respect to the angles (taking $ 15 $
random starting points has led to an identical value function approximation). The runtime has been 1685.499~s totally and
0.228356456~s per point, which seems to be suitable, taking into account the high-dimensional case and weakness of
the computational resources we have used. The runtime can be substantially smaller for more powerful machines, especially
when parallelization is done. \qed

\begin{figure}
\begin{center}
\includegraphics[ width = 8cm, height = 5.4cm ]{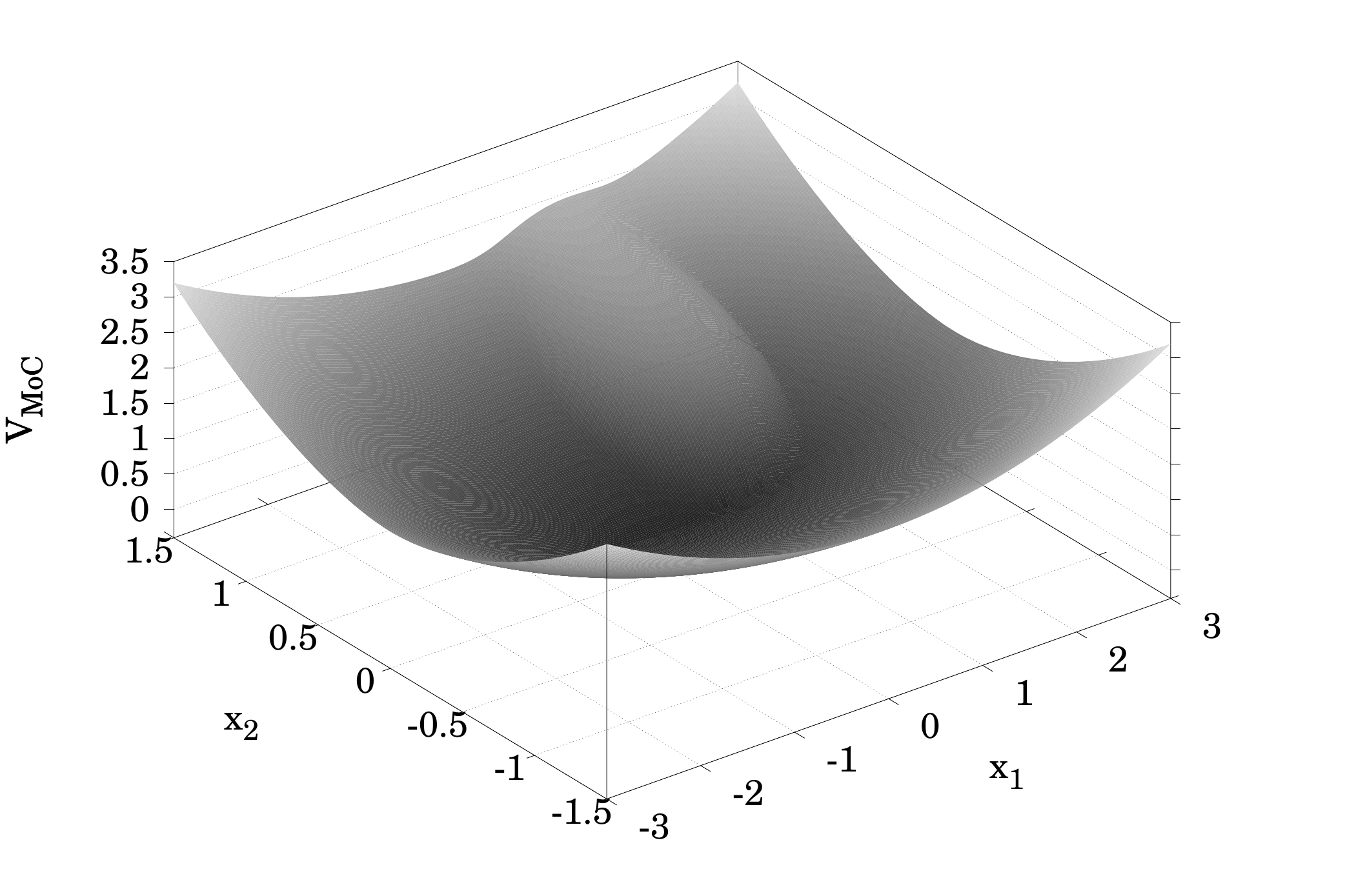}
\includegraphics[ width = 8cm, height = 5.4cm ]{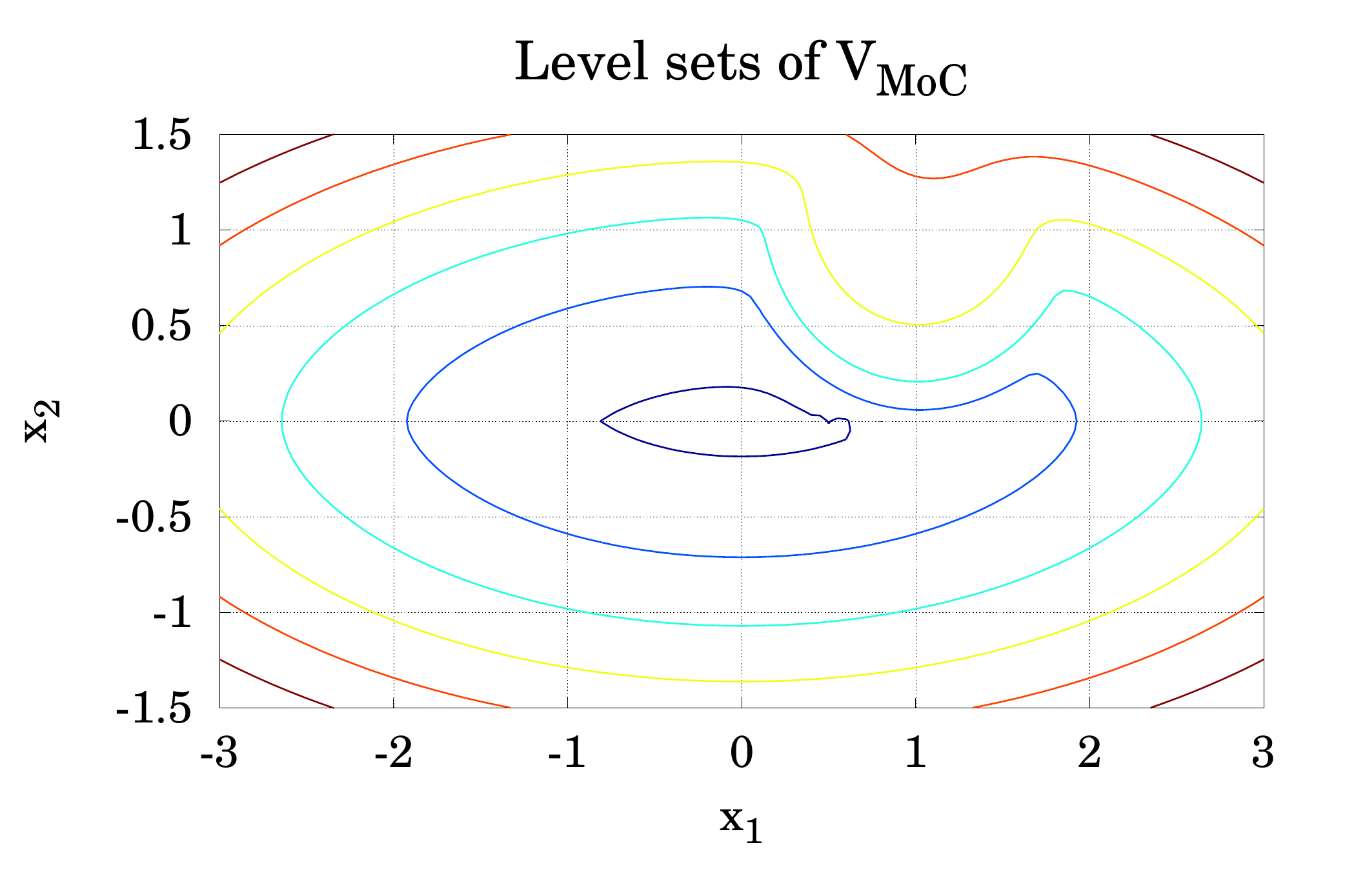}
\end{center}
\bf \caption{\it The value function approximation~$ V_{\mathrm{MoC}} $ on the plane $ \: x_3 = x_4 = x_5 = 0 \: $ in
Example~{\rm \ref{Exa_43}} for $ n = 5 $ and $ \, T - t_0 = 0.5 $.}
\label{Fig_6}
\end{figure}
\end{example}

\begin{example}  \label{Exa_44}  \rm
Consider the game~(\ref{F_69_0}) from Example~\ref{Exa_40} with
\begin{equation}
\arraycolsep=1.5pt
\def\arraystretch{2}
\begin{array}{c}
\alpha \, = \, 1, \quad a \, = \, 0.2, \quad b \, = \, 0.1, \quad u^0 \: = \: (0, 0)^{\top}, \\
t_0 \, = \, 0, \quad T \, = \, 2.
\end{array}  \label{F_87}
\end{equation}
The formulae~(\ref{F_69}) determine the programmed maximin function~$ V^* $ and closed-loop game value function~$ V $,
while the saddle feedback control maps are represented by (\ref{F_71}),~(\ref{F_72}). In (\ref{F_72}), let us select
the unique values as $ u^0 $ from $ U_1 $ and $ (0, 0)^{\top} $ from $ U_2 $.

Fig.~\ref{Fig_7} indicates the reductions of the functions~$ V, V^* $ and their difference~$ V - V^* $ to
the plane $ \: x_3 = x_4 = 0 $, $ t_0 = 0 \: $ (it is in fact enough to fix $ \, T - t_0 = 2 \, $ instead of specifying
the particular initial and final instants $ t_0 = 0 $, $ T = 2 $). Some related level sets are depicted in Fig.~\ref{Fig_8}.
The corresponding reductions of the saddle feedback control strategies are illustrated in Fig.~\ref{Fig_9}.

\begin{figure}
\begin{center}
\includegraphics[ width = 11cm, height = 7cm ]{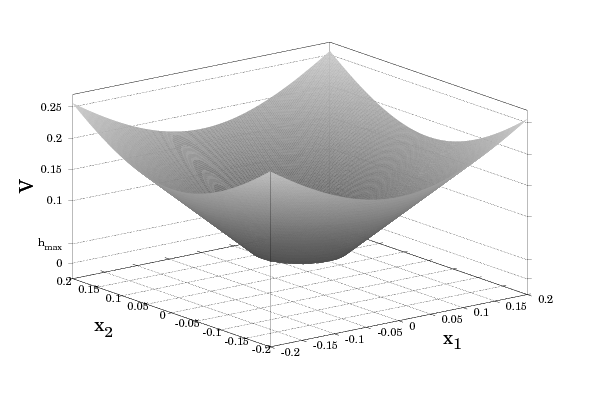}
\includegraphics[ width = 11cm, height = 7cm ]{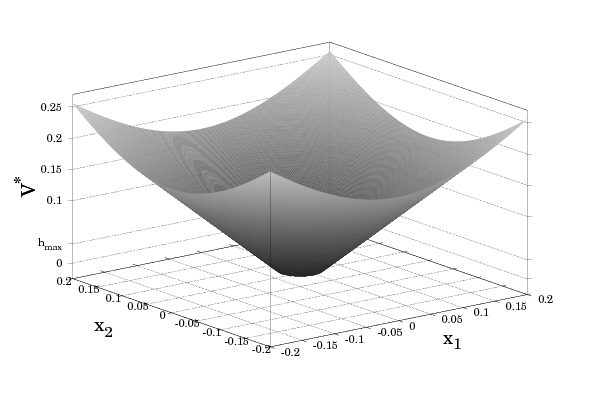} \\
\includegraphics[ width = 9cm, height = 6cm ]{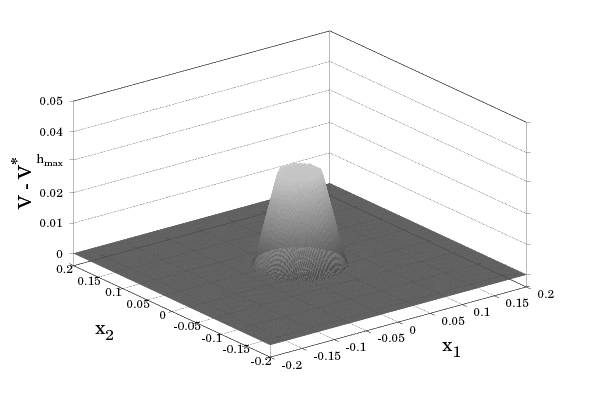}
\end{center}
\bf \caption{\it The reductions of the closed-loop game value function~$ V ${\rm ,} programmed maximin function~$ V^* $ and
their difference~$ V - V^* $ to the plane $ \: x_3 = x_4 = 0 ${\rm ,} $ \, T - t_0 = 2 \: $ in Example~{\rm \ref{Exa_44}}.
In order to see the graph of $ V - V^* $ clearer{\rm ,} a larger scale on the vertical axis is used there.}
\label{Fig_7}
\end{figure}

\begin{figure}
\begin{center}
\includegraphics[ width = 8cm, height = 5.4cm ]{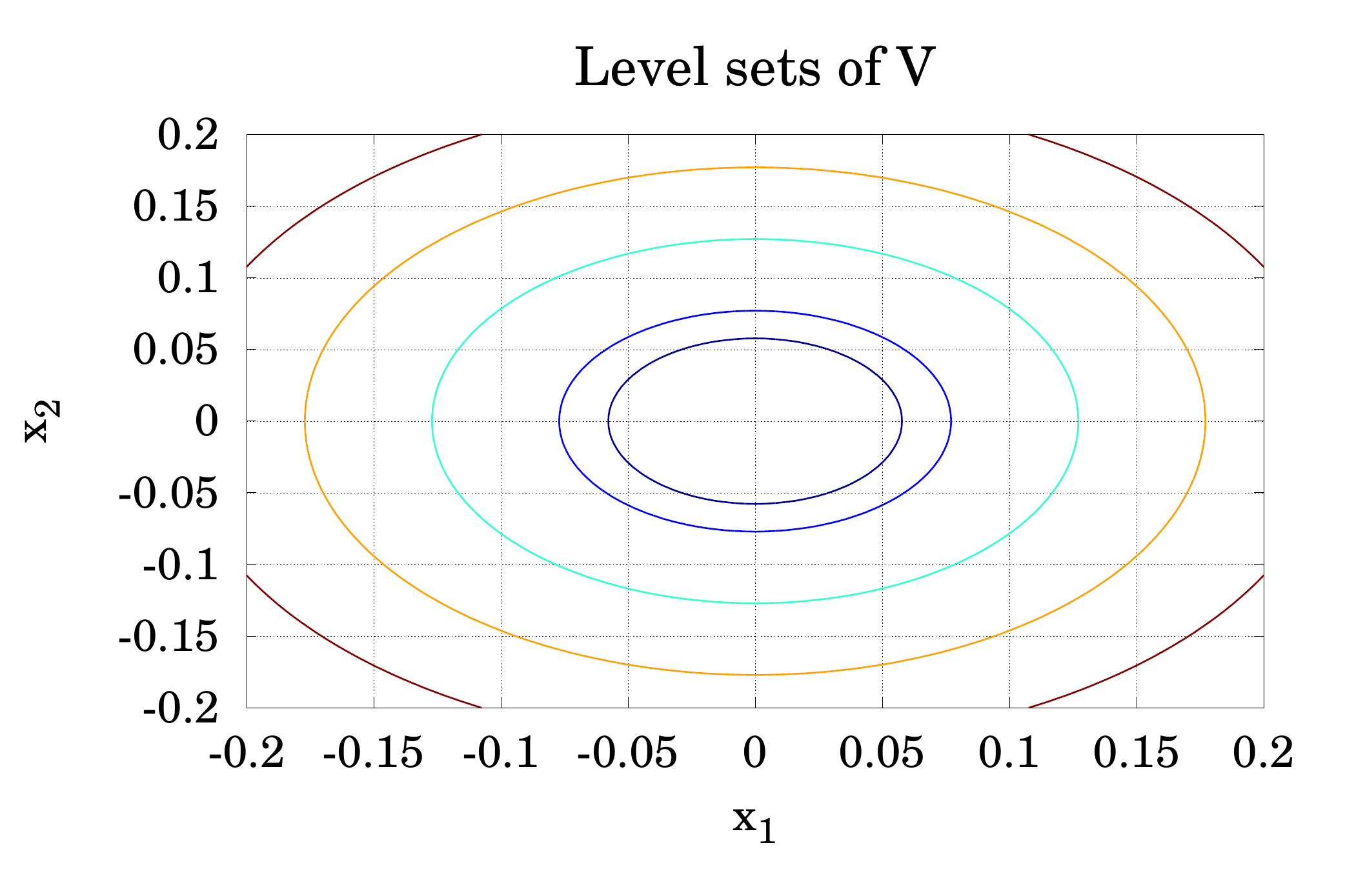}
\includegraphics[ width = 8cm, height = 5.4cm ]{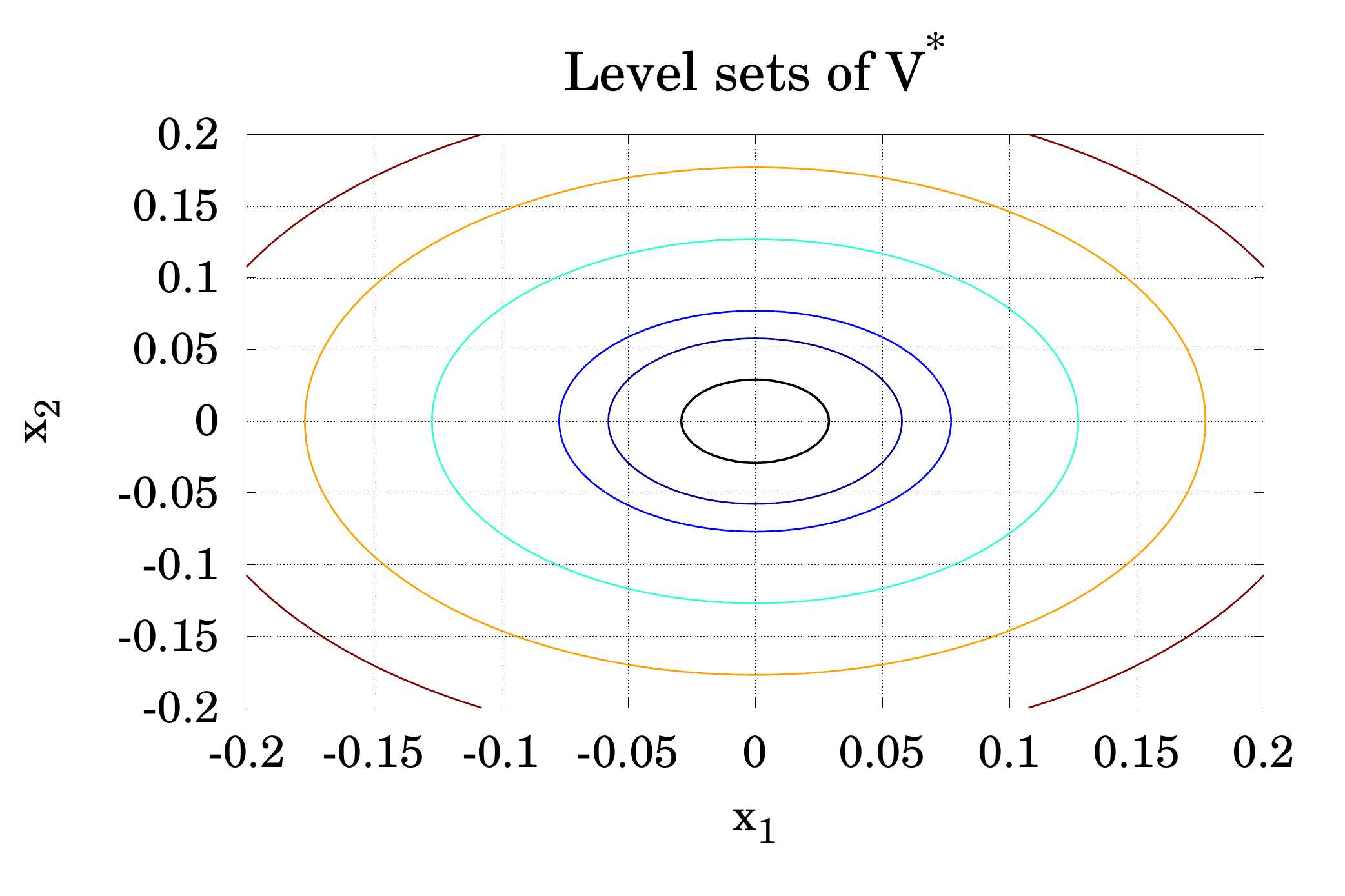}
\end{center}
\bf \caption{\it Level sets of the reductions of the functions $ V $ and $ V^* $ to the plane $ \: x_3 = x_4 = 0 ${\rm ,}
$ \, T - t_0 = 2 \: $ in Example~{\rm \ref{Exa_44}}.}
\label{Fig_8}
\end{figure}

\begin{figure}
\begin{center}
\includegraphics[ width = 8cm, height = 5.4cm ]{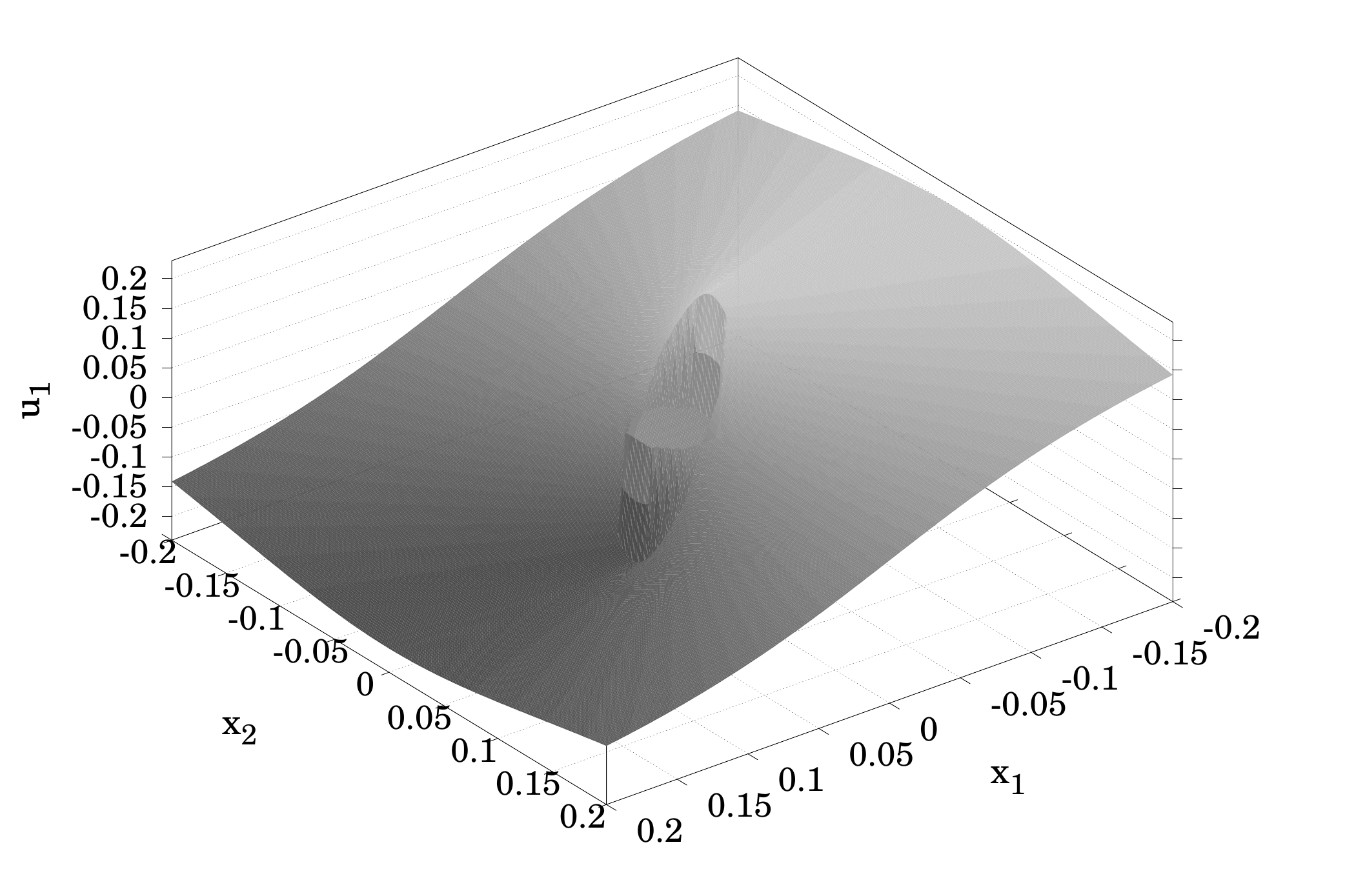}
\includegraphics[ width = 8cm, height = 5.4cm ]{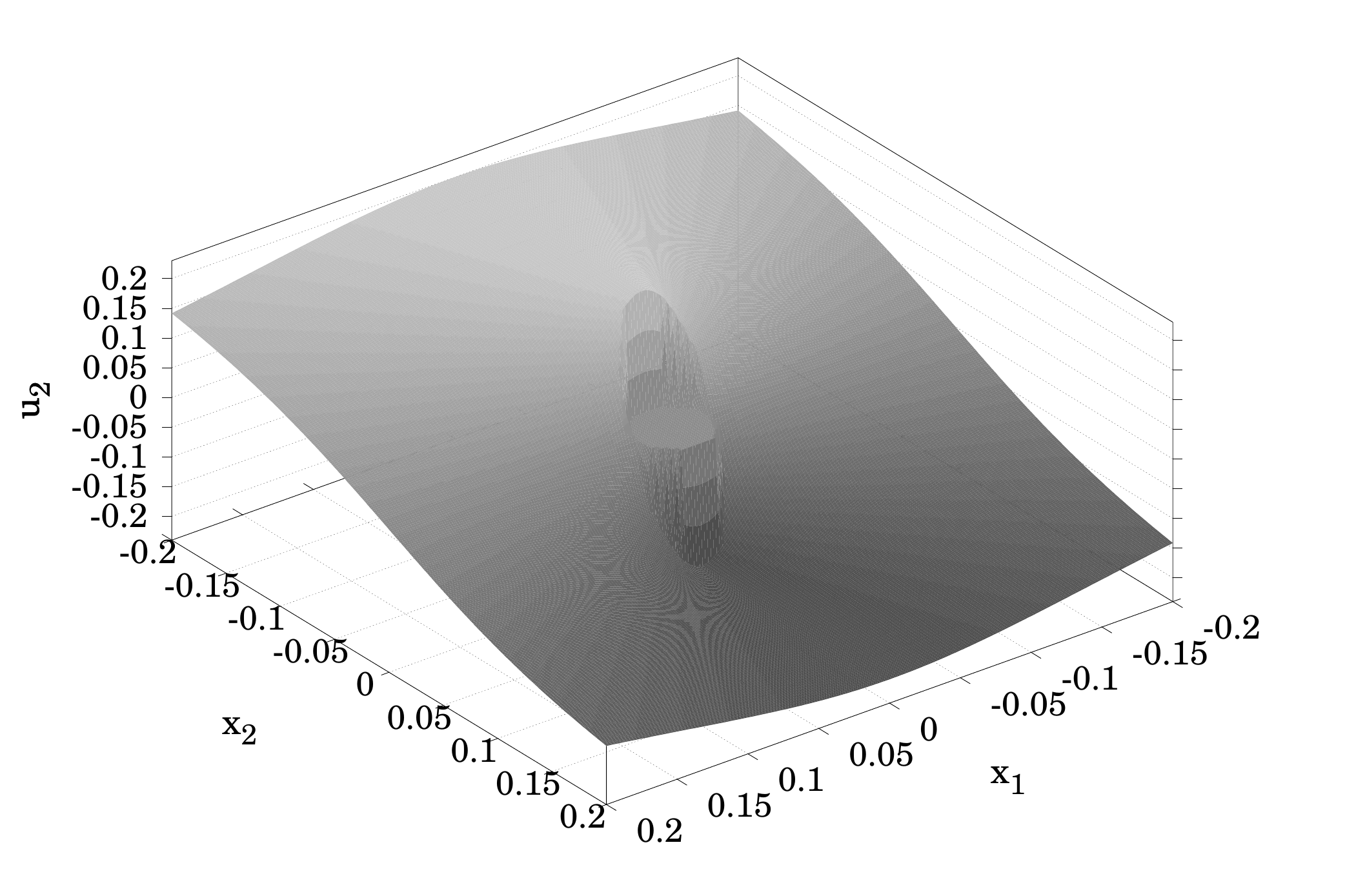} \\
\includegraphics[ width = 8cm, height = 5.4cm ]{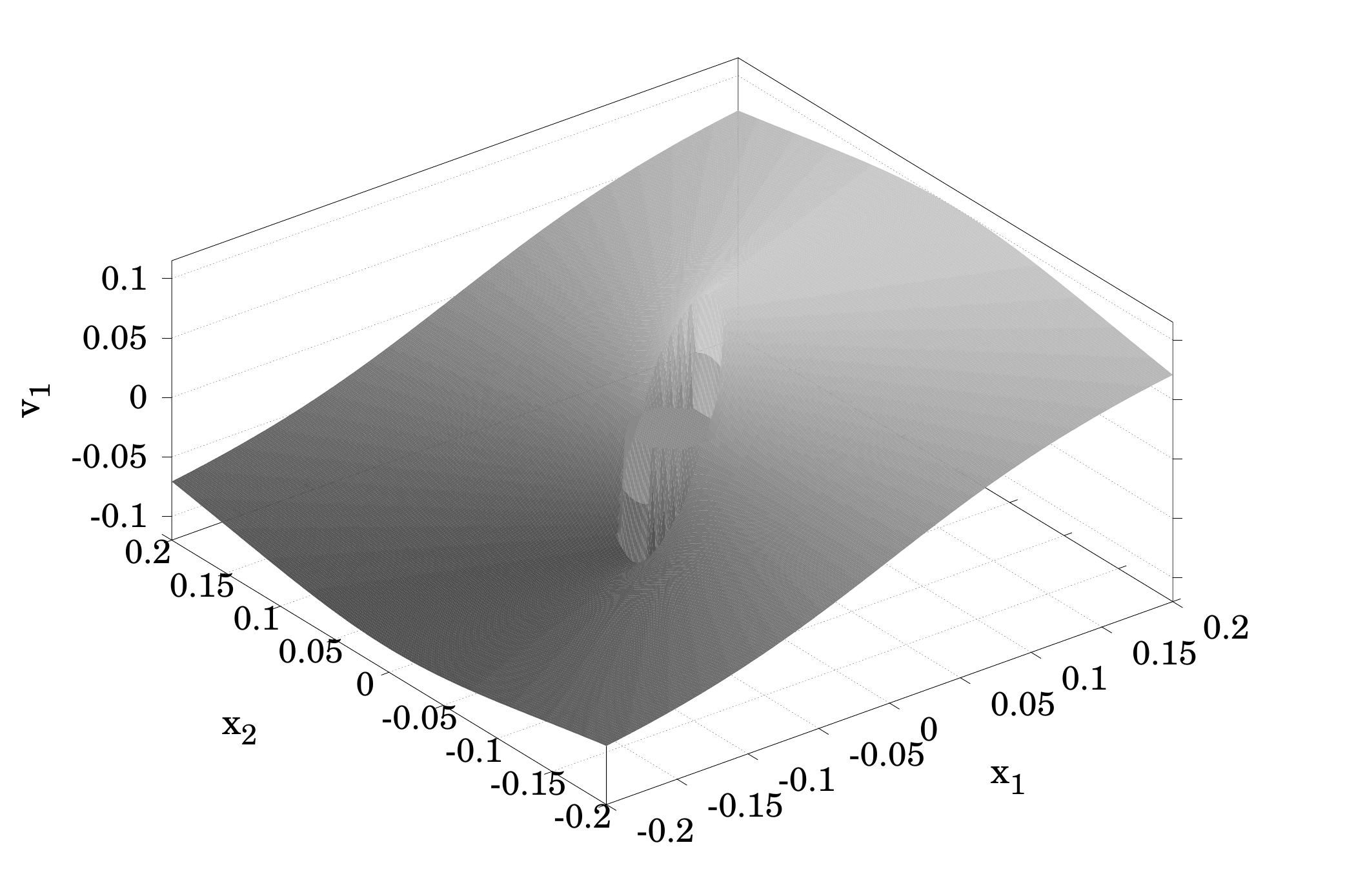}
\includegraphics[ width = 8cm, height = 5.4cm ]{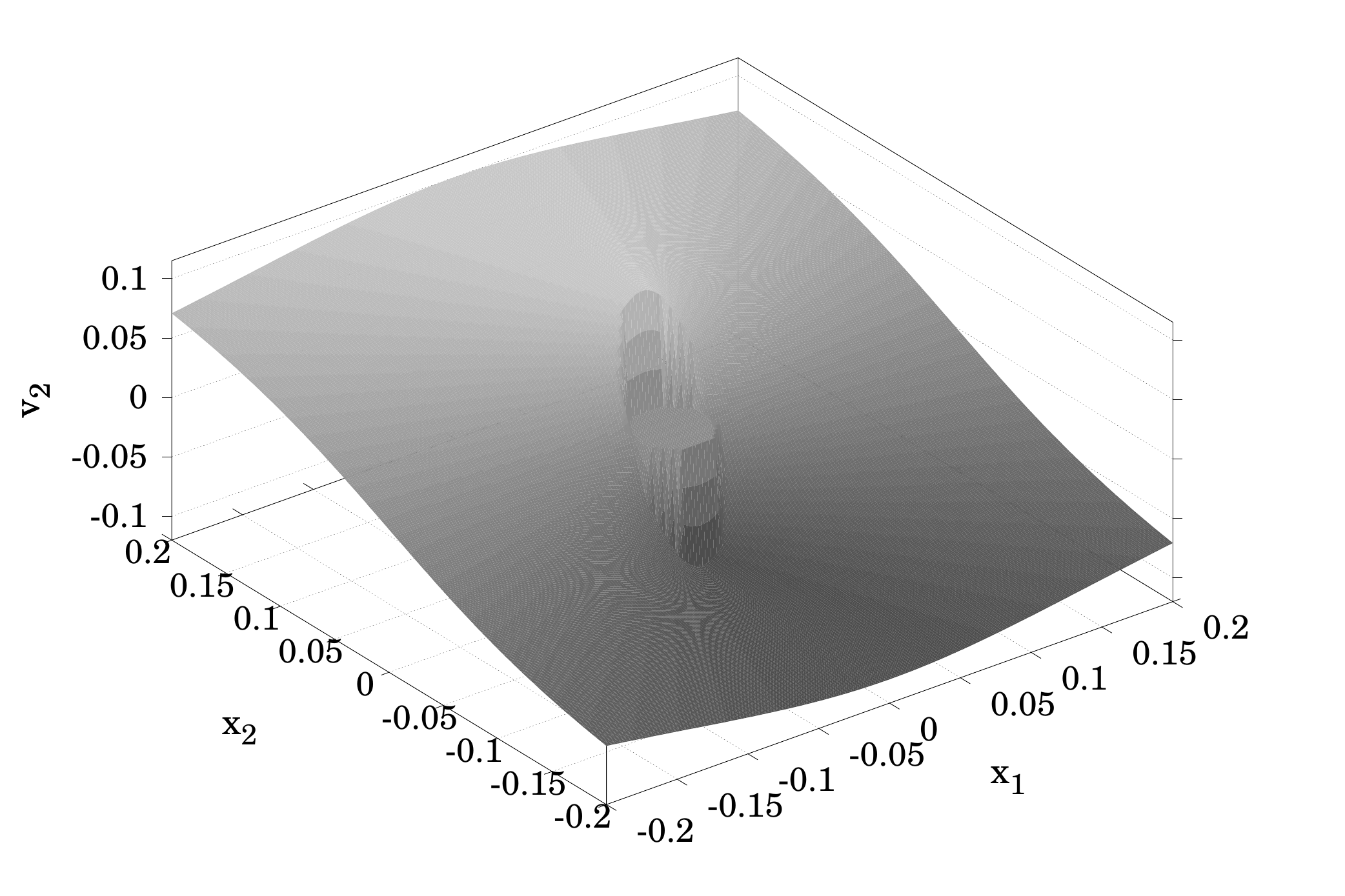}
\end{center}
\bf \caption{\it The reductions of the saddle feedback control strategies $ \, u = (u_1, u_2)^{\top} \, $ and
$ \, v = (v_1, v_2)^{\top} \, $ to the plane $ \: x_3 = x_4 = 0 ${\rm ,} $ \, T - t_0 = 2 \: $ in Example~{\rm \ref{Exa_44}}.
For the sake of convenience{\rm ,} different viewpoints for the horizontal axes are used in the graphs of $ u_1, u_2 $ and
$ v_1, v_2 $.}
\label{Fig_9}
\end{figure}

The vector of the first two initial state coordinates~$ x_1, x_2 $ was taken from the uniform grid on the square
$ \, [-0.2, 0.2] \times [-0.2, 0.2] \, $ with the spatial step $ \, 4 \times 10^{-3} $. The two-dimensional unit
sphere~$ L_2 $ was parametrized by one angle with values in the interval~$ [0, 2 \pi) $. The uniform grid for
the latter consisted of $ 10^4 $ points. In our implementation of (\ref{F_71}), the maximum over $ L_2 $ was computed
directly around this grid. Furthermore, the number
\begin{equation}
h_{\max} \:\, \stackrel{\mathrm{def}}{=} \:\, \max\limits_{t \, \in \, [t_0, T]} \, h(t)  \label{F_88}
\end{equation}
was computed directly around the uniform grid on $ [t_0, T] $ consisting of $ 10^4 + 1 $ points. We obtained
$ \, h_{\max} \approx 0.0307 $. The approximate runtime of computing the mentioned reductions of the programmed
maximin function, closed-loop game value function and saddle feedback control strategies has been 6.994~s totally and
0.000685619~s per point.

The last subfigure of Fig.~\ref{Fig_7} shows that the programmed maximin and closed-loop game value functions differ from
each other in some neighborhood of the origin ($ V - V^* $ reaches the value $ h_{\max} $ there). \qed
\end{example}

\begin{example}  \label{Exa_45}  \rm
Consider the problem~(\ref{F_73_2}) from Example~\ref{Exa_41_2} with
\begin{equation}
\arraycolsep=1.5pt
\def\arraystretch{2}
\begin{array}{c}
c(x) \:\: = \:\: -1 \: - \: 3 \:
\exp \left( -4 \: \| x \: - \: (1, \, 1, \, 0, \, 0, \ldots, \, 0) \|_{\mathbb{R}^n}^2 \right) \:\: < \:\: 0, \\
\sigma \: \equiv \: 0, \quad \eta(x) \: = \: \dfrac{1}{2} \, \left< Ax, x \right>, \\
A \:\: = \:\: \mathrm{diag} \: [0.25, \, 1, \, 0.5, \, 0.5, \, \ldots, \, 0.5] \:\: \in \:\: \mathbb{R}^{n \times n}.
\end{array}  \label{F_85_2}
\end{equation}
Then
\begin{equation}
\begin{aligned}
\{ x \in \mathbb{R}^n \: \colon \: \mathrm{D} \eta(x) \, = \, 0 \} \:\: & = \:\: \{ x' \} \:\: = \:\: \{ 0 \} \\
\:\: & = \:\: \mathrm{Arg} \, \min\limits_{x \in \mathbb{R}^n} \, \eta(x) \:\: \subset \:\:
\mathrm{Arg} \, \min\limits_{x \in \mathbb{R}^n} \, \sigma(x) \:\: = \:\: \mathbb{R}^n,
\end{aligned}  \label{F_88}
\end{equation}
i.\,e., the condition~(\ref{F_82_2}) holds. By using the final statement of Example~\ref{Exa_41_2}, one can compute
the related value function~$ V $ and optimal feedback (closed-loop) control strategy~$ u^*_{\scriptsize \mbox{c-l}} $ at any
selected position~$ \: (t, x) \, \in \, [0, T) \times \mathbb{R}^n $.

Now let the deterministic system~(\ref{F_76}) be perturbed by stochastic noise, so that the resulting system becomes
\begin{equation}
\left\{ \begin{aligned}
& dx(t) \:\: = \:\: c(x(t)) \, u(t) \, dt \: + \: \Lambda \, dw(t; t_0), \\
& x(t_0) \, = \, x_0, \\
& u(t) \: \in \: U \: = \: \left\{ v \in \mathbb{R}^n \: \colon \: \| v \| \leqslant 1 \right\}, \\
& t \in [t_0, T],
\end{aligned} \right.  \label{F_76_3}
\end{equation}
where $ \: (t_0, x_0) \, \in \, [0, T) \times \mathbb{R}^n \: $ is a deterministic initial position,
$ \Lambda \in \mathbb{R}^{n \times n} $ is a constant noise intensity matrix, $ w(\cdot; t_0) $ is an $ n $-dimensional
standard Brownian motion (Wiener process) on the time interval~$ [t_0, T] $, and the stochastic ordinary differential
equations are understood in the It\^o sense. An open-loop control strategy can also be a stochastic process if it is
obtained from a closed-loop map. It is reasonable to assess the control quality through the mean value
\begin{equation}
\mathbb{E} \left[ \sigma(x(T)) \: + \: \int\limits_{t_0}^T \eta(x(t)) \, dt \right] \:\: = \:\:
\mathbb{E} \left[ \int\limits_{t_0}^T \eta(x(t)) \, dt \right] \:\: = \:\:
\int\limits_{t_0}^T \mathbb{E} [\eta(x(t))] \, dt.  \label{F_89}
\end{equation}
The lower this value, the higher the control quality. The control goal can be informally interpreted as mitigating
the random vibrations whose strength is described by the mean running cost~(\ref{F_89}).

Set the noise intensity matrix as
\begin{equation}
\Lambda \:\: = \:\: \mathrm{diag} \: [\varepsilon, \, \varepsilon, \, 0, \, 0, \, \ldots, \, 0] \:\: \in \:\:
\mathbb{R}^{n \times n}, \quad \varepsilon = 0.3,  \label{F_90}
\end{equation}
i.\,e., the noise is diagonal and appears only in the first two state dynamic equations (it is degenerate for $ n \geqslant 3 $).
Furthermore, take the state space dimension, time horizon and initial position as follows:
\begin{equation}
n = 6, \quad T = 2, \quad t_0 = 0, \quad x_0 \: = \: (-0.5, \, 0.5, \, 0.3, \, -0.3, \, 0.3, \, -0.3)^{\top}.
\label{F_91}
\end{equation}

Let $ \: u^*_{\scriptsize \mbox{o-l}} \, \colon \, [t_0, T] \to U \: $ be an optimal open-loop control strategy in
the deterministic problem~(\ref{F_76}),~(\ref{F_77_2}) (corresponding to the case~$ \varepsilon = 0 $) with
the data~(\ref{F_85_2}), (\ref{F_91}), and recall that
$ \: u^*_{\scriptsize \mbox{c-l}} \, \colon \, [0, T] \times \mathbb{R}^n \, \to \, U \: $ denotes an optimal feedback
control law in this problem. Consider the two stochastic systems
\begin{equation}
\left\{ \begin{aligned}
& dx(t) \:\: = \:\: c(x(t)) \: u^*_{\scriptsize \mbox{o-l}}(t) \: dt \:\, + \:\, \Lambda \, dw(t; t_0), \quad
t \in [t_0, T], \\
& x(t_0) \, = \, x_0,
\end{aligned} \right.  \label{F_92}
\end{equation}
\begin{equation}
\left\{ \begin{aligned}
& dx(t) \:\: = \:\: c(x(t)) \: u^*_{\scriptsize \mbox{c-l}}(t, x(t)) \: dt \:\, + \:\, \Lambda \, dw(t; t_0), \quad
t \in [t_0, T], \\
& x(t_0) \, = \, x_0,
\end{aligned} \right.  \label{F_93}
\end{equation}
governed by the introduced open-loop and closed-loop strategies. Let $ J^*_{\varepsilon, \: \scriptsize \mbox{o-l}} $ and
$ J^*_{\varepsilon, \: \scriptsize \mbox{c-l}} $ be the mean values~(\ref{F_89}) for the trajectories of (\ref{F_92}) and
(\ref{F_93}), respectively. One can expect that the feedback control should mitigate the random vibrations better than
the open-loop control, i.\,e.,
$ \, J^*_{\varepsilon, \: \scriptsize \mbox{c-l}} < J^*_{\varepsilon, \: \scriptsize \mbox{o-l}}, \, $ while these quantities
are obviously equal to each other in the deterministic case~$ \varepsilon = 0 $.

In order to estimate $ J^*_{\varepsilon, \: \scriptsize \mbox{c-l}} $ for a high state space dimension~$ n $ (making nonlocal
approximations of $ u^*_{\scriptsize \mbox{c-l}} $ extremely difficult), we have implemented the piecewise constant control
policy that is recomputed every $ \, \Delta t_{\mathrm{recomp}} = 0.05 \, $ time units as the value of $ u^*_{\scriptsize \mbox{c-l}} $
at the current position (via the algorithm of Example~\ref{Exa_41_2}). Denote the related mean running cost~(\ref{F_89}) by
$ \hat{J}^*_{\varepsilon, \: \scriptsize \mbox{c-l}} $. Such an approximation technique is inspired by model predictive control (MPC)
approaches \cite{Wang2009,KangWilcox2017}. If the recomputation step~$ \Delta t_{\mathrm{recomp}} $ is small enough, it is also
reasonable to expect that
\begin{equation}
\hat{J}^*_{\varepsilon, \: \scriptsize \mbox{c-l}} \, < \, J^*_{\varepsilon, \: \scriptsize \mbox{o-l}}.  \label{F_94}
\end{equation}

The first subfigure of Fig.~\ref{Fig_10} indicates the graphs of the estimated mean value functions
$ \: [t_0, T] \ni t \, \longmapsto \, \mathbb{E} [\eta(x(t))] \: $ for the open-loop control system~(\ref{F_92}) and for
the specified MPC implementation of the feedback control system~(\ref{F_93}). The corresponding standard deviation estimates are
illustrated in the second subfigure. The graph of $ \: [t_0, T] \ni t \, \longmapsto \, \eta(x(t)) \: $ for the system~(\ref{F_92})
in the deterministic case~$ \varepsilon = 0 $ is also shown for comparison. These graphs have been constructed by performing
$ N = 1000 $ Monte Carlo iterations, and $ x^{[i]}(\cdot) $ denotes the state trajectory at the $ i $-th iteration,
$ i = \overline{1, N} $. The conjecture~(\ref{F_94}) indeed agrees with the presented numerical simulation results:
$ \: \hat{J}^*_{\varepsilon, \: \scriptsize \mbox{c-l}} \, \approx \, 0.116 \, < \, 0.188 \, \approx \, J^*_{\varepsilon, \: \scriptsize \mbox{o-l}} $.

\begin{figure}
\begin{center}
\includegraphics[ width = 13cm, height = 6cm ]{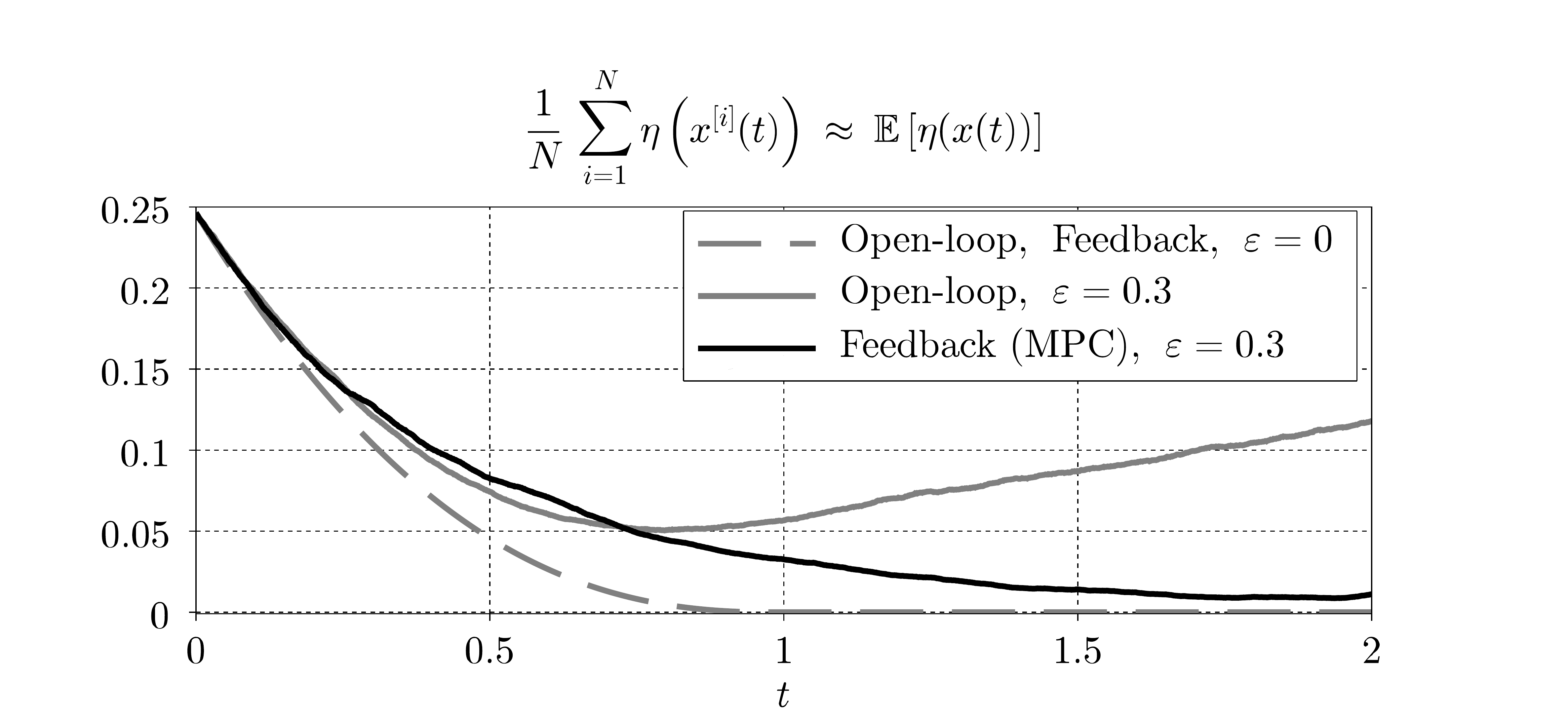}
\includegraphics[ width = 13cm, height = 6cm ]{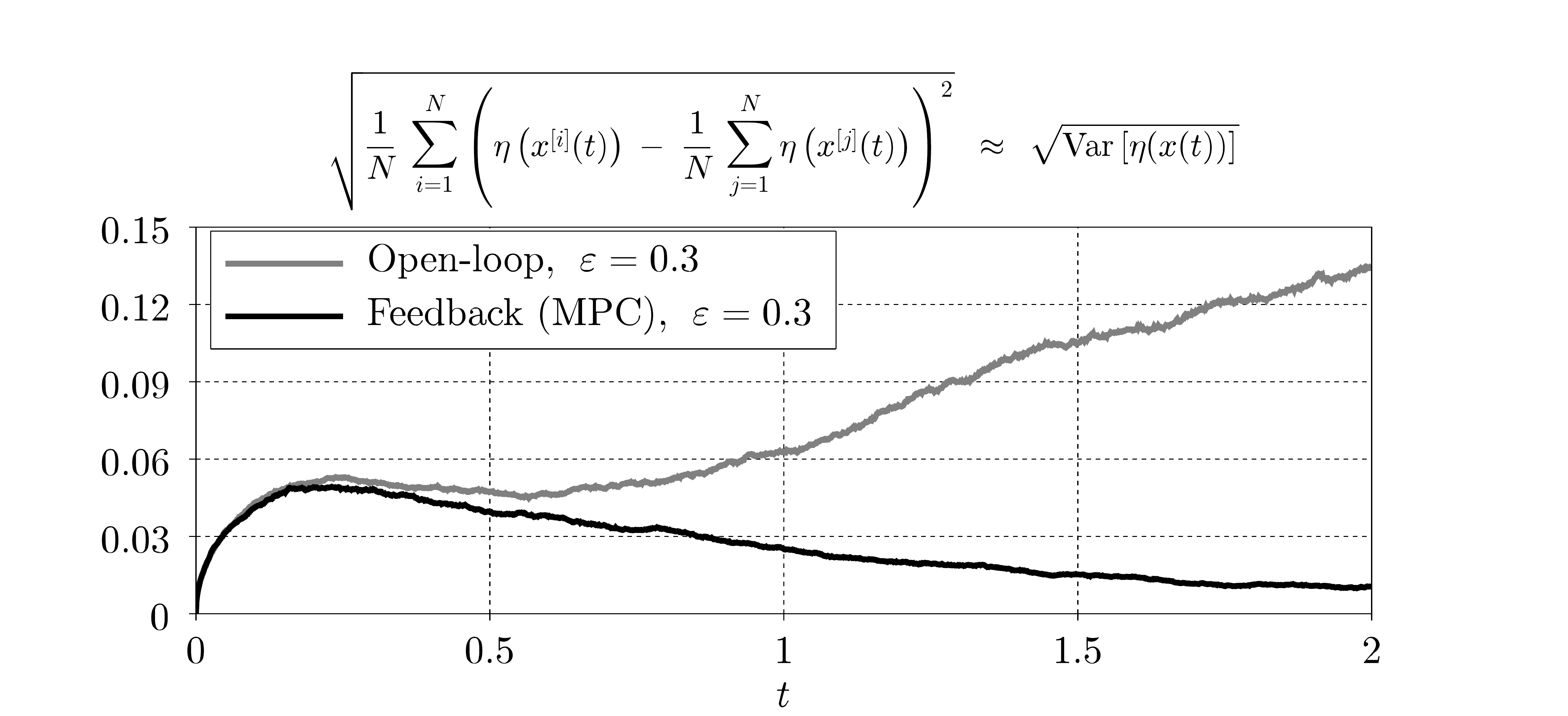}
\end{center}
\bf \caption{\it Estimates of the mean values~$ \mathbb{E} [\eta(x(t))] $ and standard deviations
$ \sqrt{\mathrm{Var} \, [\eta(x(t))]} $ for the open-loop control system~{\rm (\ref{F_92})} and for the specified MPC implementation
of the feedback control system~{\rm (\ref{F_93})} in Example~{\rm \ref{Exa_45}}.}
\label{Fig_10}
\end{figure}

When implementing the algorithm of Example~\ref{Exa_41_2} (for computing the optimal feedback strategy~$ u^*_{\scriptsize \mbox{c-l}} $
at selected positions), the numerical criterion of reaching the state
$ \: x \, = \, x' \, = \, 0 \, \in \, \mathrm{Arg} \, \min\limits_{x \in \mathbb{R}^n} \, \eta(x) \: $ was chosen as
$ \: \eta(x) - \eta(x') \, = \, \eta(x) \, < \, 5 \cdot 10^{-6} \: $ (the exact deterministic optimal control has zero action at
$ x = x' $). The $ 12 $-dimensional deterministic characteristic system was integrated via the same Runge--Kutta routine (with
the same initial stepsize guess, absolute and relative tolerances) as mentioned in Example~\ref{Exa_43}. The optimization over
the $ 7 $-dimensional unit sphere of the vectors~$ \left( p_0, \tilde{p}^* \right) $ was performed by using Powell's algorithm of
\cite{PressNR2007} with the tolerance parameter~$ 10^{-5} $. No special method was needed to exclude the value~$ \tilde{p}^* = 0 $,
since it was not approached enough (at least with the tolerance~$ 10^{-5} $) in the executed optimization iterations. The sphere was
parametrized in the standard way. The related representation was obtained from (\ref{F_86}) by replacing $ n = 6 $ with $ n + 1 = 7 $.
For each optimization process, $ 5 $ starting points were randomly generated in line with the uniform angles distribution.

The It\^o stochastic differential equations have been solved by means of the Euler--Maruyama scheme that coincides with
the Milstein scheme for the constant and diagonal noise intensity matrix~(\ref{F_90})~\cite{KloedenPlaten1995,Carletti2006}.
The corresponding time step was taken as $ \, \Delta t_{\mathrm{SDE}} = 10^{-5} $. Under certain smoothness and Lipschitz
continuity conditions on the drift vector function and noise intensity matrix function, the Milstein scheme has the first
strong convergence order (while the order of the Euler--Maruyama scheme in general equals~$ 0.5 $ if the noise intensity
matrix is not constant). One can expect that the first order of accuracy should be preserved in our MPC implementation of
the system~(\ref{F_93}), because the related control policy is piecewise constant and the ratio
$ \: \Delta t_{\mathrm{recomp}} / \Delta t_{\mathrm{SDE}} \, = \, 5 \cdot 10^3 \: $ is integer. When integrating the open-loop control
system~(\ref{F_92}), one may face an additional numerical error caused by the discontinuity of $ u^*_{\scriptsize \mbox{o-l}} $
at the time of entering the state~$ x = x' $. Since this has been the only discontinuity of $ u^*_{\scriptsize \mbox{o-l}} $ on
the whole time interval~$ [t_0, T] $, the resulting error has not appeared to be significant.

As follows from Fig.~\ref{Fig_10}, the feedback MPC approach allows for a rather successful random vibrations mitigation by
periodic recomputation of the control according to the current position. The open-loop control cannot be adapted in this way
and is therefore unstable. Starting from some time, it has zero action as if $ x = x' $, while the noise is likely to move
the state away from~$ x' $. \qed
\end{example}

\section{Conclusion}

This paper extends the considerations of the works~\cite{DarbonOsher2016,ChowDarbonOsherYin2017} concerning
curse-of-dimensionality-free numerical approaches to solve certain types of Hamilton--Jacobi equations arising in optimal control
problems, differential games and elsewhere. A rigorous formulation and justification for the extended Hopf--Lax formula of
\cite{ChowDarbonOsherYin2017} is provided together with novel theoretical and practical discussions including useful recommendations.
By using the method of characteristics, the solutions of some problem classes under convexity/concavity conditions on Hamiltonians
(in particular, the solutions of Hamilton--Jacobi--Bellman equations in optimal control problems) are evaluated separately at different
initial positions. This allows for the avoidance of the curse of dimensionality, as well as for choosing arbitrary computational
regions. The corresponding feedback control strategies are obtained at selected positions without approximating the partial
derivatives of the solutions. The numerical simulation results demonstrate the high potential of the proposed techniques.

Despite the indicated advantages, the related approaches still have a limited range of applicability (recall
Remarks~\ref{Rem_14}, \ref{Rem_14_0}, \ref{Rem_23}--\ref{Rem_42}), and their extensions to Hamilton--Jacobi--Isaacs equations in
zero-sum two-player differential games are currently developed only for rather narrow classes of linear control systems (as shown in
Section~4 and Appendix). That is why further extensions are worth investigating. In particular, it is relevant to find a wider
description for the classes of optimal control problems, where the finite-dimensional optimization in the algorithms of computing
the value functions can be performed over bounded sets. Regarding the second main conjecture of \cite{ChowDarbonOsherYin2017}
(i.\,e., the extended Hopf formula) that may be applied to some classes of nonlinear differential games, its rigorous formulation and
justification also remains an open problem.

Finally, note that the paper \cite{KangWilcox2017} proposed to solve boundary value problems for characteristic systems numerically
with the aim to approximate the value functions of optimal control problems at separate sparse grid nodes. However, this may
lead to incorrect results at the initial positions for which the boundary value problems have multiple solutions and the value functions
are nonsmooth (recall Fig.~\ref{Fig_5} in Example~\ref{Exa_43}). Therefore, it seems promising to combine the approaches of the current
work with the sparse grid techniques of \cite{KangWilcox2017} so as to construct global approximations of the value functions and
corresponding feedback control laws in domains of relatively high dimensions.

\section*{Acknowledgements}

This work was supported in part by AFOSR/AOARD grant FA2386-16-1-4066.

\appendix

\section{Appendix}

Let us introduce one more class of differential games for which a single programmed iteration is enough to reach
the closed-loop game value. The corresponding result was derived in \cite[\S V.2]{SubbotinChentsov1981}.

Consider the linear differential game~(\ref{F_40}),~(\ref{F_41}) under Assumptions~\ref{Ass_25},~\ref{Ass_29}.
In addition to the notations presented in the introduction and Section~4, adopt that
\begin{equation}
\arraycolsep=1.5pt
\def\arraystretch{2}
\begin{array}{c}
\| y \|_j^0 \:\, \stackrel{\mathrm{def}}{=} \:\, \max\limits_{i \: \in \: \{1, 2, \ldots, j\}} |y_i|, \quad
\| y \|_j^1 \: \stackrel{\mathrm{def}}{=} \: \sum\limits_{i = 1}^j |y_i|, \\
\mathcal{B}_j^0(y, \varepsilon) \:\: \stackrel{\mathrm{def}}{=} \:\: \left\{ v \in \mathbb{R}^j \: \colon \:
\| v - y \|_j^0 \, \leqslant \, \varepsilon \right\}
\end{array}  \label{F_50_2}
\end{equation}
for all $ \: j \in \mathbb{N} $, $ \: y \, = \, (y_1, y_2, \ldots, y_j) \, \in \, \mathbb{R}^j \: $ and
$ \varepsilon \geqslant 0 $. Several more conditions are imposed.

\begin{assumption}  \label{Ass_31}
There exist a vector~$ y_0 \in \mathbb{R}^k ${\rm ,} numbers $ r \in [0, +\infty) ${\rm ,} $ \bar{r} \in [0, +\infty) $ and
continuous functions
$$
\arraycolsep=1.5pt
\def\arraystretch{1.5}
\begin{array}{c}
\hat{y} \: \colon \: [0, T] \, \to \, \mathbb{R}^k, \quad \hat{z} \: \colon \: [0, T] \, \to \, \mathbb{R}^k, \\
h_1 \: \colon \: [0, T] \, \to \, [0, +\infty), \quad h_{21} \: \colon \: [0, T] \, \to \, [0, +\infty), \\
h_{22} \: \colon \: [0, T] \, \to \, [0, +\infty),
\end{array}
$$
such that the set~$ \mathcal{M} $ from Assumption~{\rm \ref{Ass_29}} is determined by
\begin{equation}
\mathcal{M} \:\: = \:\: y_0 \: + \: \mathcal{B}_k^0(O_k, r) \: + \: \mathcal{B}_k(O_k, \bar{r})  \label{F_54}
\end{equation}
and{\rm ,} for all $ \: t \in [0, T] ${\rm ,} $ \, l \in L_k, \: $ one has
\begin{equation}
\arraycolsep=1.5pt
\def\arraystretch{2}
\begin{array}{c}
\{ \Phi(T, t) \, B_1(t) \, U_1(t) \}_k \:\: = \:\: \mathcal{B}^0_k \left( \hat{y}(t), \, h_1(t) \right), \\
\int\limits_t^T \max\limits_{u_2 \, \in \, U_2} \left< l, \: \{ \Phi(T, \xi) \, B_2(\xi) \, u_2 \}_k \right> \: d \xi \:\: = \:\:
h_{21}(t) \: + \: h_{22}(t) \, \| l \|^1_k \: + \: \left< l, \, \hat{z}(t) \right>, \\
h_{22}(t) \:\: \leqslant \:\: \int\limits_t^T h_1(\xi) \, d \xi \: + \: r.
\end{array}  \label{F_55}
\end{equation}
\end{assumption}

Denote also
\begin{equation}
h_3(t) \:\: \stackrel{\mathrm{def}}{=} \:\: \int\limits_t^T h_1(\xi) \, d \xi \: + \: r \: - \: h_{22}(t) \quad
\forall t \in [0, T].  \label{F_56}
\end{equation}

The sought-after result can now be formulated.

\begin{theorem}{\rm \cite[\S V.2]{SubbotinChentsov1981}}  \label{Thm_32}
Under Assumptions~{\rm \ref{Ass_25}, \ref{Ass_29}, \ref{Ass_31},} the closed-loop game value for
{\rm (\ref{F_40}), (\ref{F_41})} at any position $ \: (t_0, x_0) \: \in \: [0, T] \times \mathbb{R}^n \: $ is represented as
\begin{equation}
V(t_0, x_0) \:\: = \:\: \max \, \left\{ V^*(t_0, x_0), \:\:
\max\limits_{t \, \in \, [t_0, T]} \, \{ h_{21}(t) \, - \, h_3(t) \} \: - \: \bar{r} \right\},  \label{F_57}
\end{equation}
where $ V^* $ is the programmed maximin function specified in Proposition~{\rm \ref{Pro_36}}.
\end{theorem}

Theorem~\ref{Thm_32} can be applied in the two subsequent examples.

\begin{example}{\rm \cite[\S V.2]{SubbotinChentsov1981}}  \label{Exa_33}  \rm
Consider the problem that appears from the game of Example~\ref{Exa_39} just by replacing
$ \, U_1 = \mathcal{B}_2(O_2, a_1) \, $ with $ \, U_1 = \mathcal{B}^0_2(O_2, a_1) $. Then the closed-loop game value
function is also represented as in (\ref{F_58}), but the programmed maximin function~$ V^* $ is now different. \qed
\end{example}

\begin{example}{\rm \cite[\S V.2]{SubbotinChentsov1981}}  \label{Exa_34}  \rm
For the game
\begin{equation}
\left\{ \begin{aligned}
& \dot{x}_1(t) \:\: = \:\: x_3(t) \: + \: v_1(t), \\
& \dot{x}_2(t) \:\: = \:\: x_4(t) \: + \: v_2(t), \\
& \dot{x}_3(t) \:\: = \:\: -\alpha \, x_3(t) \: + \: u_1(t) \: + \: v_3(t), \\
& \dot{x}_4(t) \:\: = \:\: -\alpha \, x_4(t) \: + \: u_2(t) \: + \: v_4(t), \\
& x(t) \: = \: (x_1(t), \, x_2(t), \, x_3(t), \, x_4(t))^{\top} \: \in \: \mathbb{R}^4, \\
& u(t) \: = \: (u_1(t), \, u_2(t))^{\top} \: \in \: U_1 \: = \: \mathcal{B}^0_2 \left( \omega^0, a \right), \\
& v(t) \: = \: (v_1(t), \, v_2(t), \, v_3(t), \, v_4(t))^{\top} \\
& \quad \:\:\:\:
\in \: U_2 \: = \: (-\mathcal{B}_2(\omega_*, b_*)) \, \times \, \mathcal{B}^0_2(\omega^*, b^*), \\
& \alpha > 0, \:\: b_* \geqslant 0, \:\: b^* \geqslant 0, \:\: a \geqslant b^* \:\: \mbox{are scalar constants}, \\
& \omega^0, \: \omega_*, \: \omega^* \:\: \mbox{are constant vectors in $ \mathbb{R}^2 $}, \\
& t \in [0, T], \quad  \mbox{$ T > 0 $ is fixed}, \\
& k = 2, \quad \mathcal{M} = \{ O_2 \}, \\
& \sigma(x(T)) \:\: = \:\: \| \{ x(T) \}_2 \|_2 \:\: = \:\: \sqrt{x_1^2(T) \, + \, x_2^2(T)} \:\: \longrightarrow \\
& \qquad\qquad\qquad\qquad\qquad\quad \:\:\,
\longrightarrow \:\: \inf_{u(\cdot)} \, \sup_{v(\cdot)} \:\: \mbox{or} \:\:
\sup_{v(\cdot)} \, \inf_{u(\cdot)} \, ,
\end{aligned} \right.  \label{F_59_0}
\end{equation}
Theorem~\ref{Thm_32} leads to the representation
\begin{equation}
\begin{aligned}
& V \left( t_0, \, x^0 \right) \:\: = \:\: \max \, \left\{ V^* \left( t_0, \, x^0 \right), \:\:
\max\limits_{t \, \in \, [t_0, T]} \{ (T - t) \, b_* \: - \: R_{\alpha}(T, t) \, (a - b^*) \} \right\} \\
& \forall \: \left( t_0, \, x^0 \right) \: \in \: [0, T] \times \mathbb{R}^4,
\end{aligned}  \label{F_59}
\end{equation}
where $ R_{\alpha} $ is defined as in (\ref{F_69_2}) and $ V^* $ is determined according to Proposition~\ref{Pro_36}. \qed
\end{example}

\end{document}